\documentclass{amsart}
\usepackage[all,knot]{xy}
\xyoption{arc} 
\usepackage{amsfonts}
\usepackage{array}
\usepackage{euscript}
\usepackage{tikz}
\usepackage{epsfig, psfrag, epic}
\usepackage{graphicx}
\usepackage{xy}
\usepackage{multirow}
\usepackage{amsmath}
\usepackage{enumitem}
\usepackage{color}
\usepackage{mathtools}
\usepackage{xcolor}
\usepackage[colorlinks=true,allbordercolors=white, citecolor=blue]{hyperref}

\usepackage[utf8]{inputenc}
\usepackage{subcaption}
\input colordvi   

\title{Mod-two cohomology rings of alternating groups}

\newtheorem{theorem}{Theorem}[section]
\newtheorem{corollary}[theorem]{Corollary}
\newtheorem{lemma}[theorem]{Lemma}
\newtheorem{proposition}[theorem]{Proposition}

\theoremstyle{definition}
\newtheorem{definition}[theorem]{Definition}

\newtheorem{remark}[theorem]{Remark}

\newcommand{\sq}{{Sq}}

\definecolor{darkgreen}{rgb}{0,0.5,0}

\newcommand{\R}{{\mathbb R}}

\newcommand{\tr}{\odot}

\newcommand{\si}{{\mathcal S}}
\newcommand{\alt}{{\mathcal A}}

\newcommand{\fn}{{\rm FN}}
\newcommand{\fna}{{\rm FNA}}
\newcommand{\F}{{\mathbb F}}

\newcommand{\Sh}{\text{Sh}}

\newcommand{\refT}[1]{Theorem~\ref{T:#1}}

\newcommand{\refP}[1]{Proposition~\ref{P:#1}}

\newcommand{\refF}[1]{Figure~\ref{F:#1}}

\newcommand{\Conf}[2]{\textrm{Conf}_{#1}(\mathbb{R}^{#2})}
\newcommand{\UConf}[2]{\overline{\textrm{Conf}}_{#1}(\mathbb{R}^{#2})}
\newcommand{\Ind}[2]{\textrm{Ind}_{#1}(\mathbb{R}^{#2})}
\newcommand{\UInd}[2]{\overline{\textrm{OrInd}}_{#1}(\mathbb{R}^{#2})}

\newcommand\Item[1][]{%
  \ifx\relax#1\relax  \item \else \item[#1] \fi
  \abovedisplayskip=0pt\abovedisplayshortskip=0pt~\vspace*{-\baselineskip}}

\begin{document}

\author[C. Giusti]{Chad Giusti}
\address{Department of Mathematical Sciences, University of Delaware}
\email{cgiusti@udel.edu}
\author[D. Sinha]{Dev Sinha}
\address{Mathematics Department, 
University of Oregon}
\email{dps@uoregon.edu}

\begin{abstract}
We  calculate the direct sum of the mod-two cohomology of all alternating groups, with both cup and transfer product structures, 
which in particular determines the additive structure and ring structure of the cohomology of individual groups.   
We show  that there are no nilpotent elements in the 
cohomology rings of individual alternating groups.
We  calculate the action of the Steenrod algebra and discuss individual component rings.  
A  range of techniques are developed, including  an almost Hopf ring structure associated to the embeddings
of products of alternating groups and Fox-Neuwirth resolutions, which are new techniques.  We also extend understanding of 
the Gysin sequence relating the cohomology
of alternating groups to that of symmetric groups and calculation of restriction to elementary abelian subgroups.
\end{abstract}

\maketitle

% !TEX root = AltGroupsModTwo.V3.tex

\section{Introduction}

Alternating groups are a fundamental series of simple groups whose cohomology has been mysterious, even additively,
with only a few calculations for small alternating groups \cite{AMM90, AdMi94, King10} which have been carried out.  
%for over fifty years after the mod-$p$ cohomology of symmetric groups was determined additively by Nakaoka \cite{Naka61}.
%While some remarkable but necessarily ad-hoc  methods have 
%yielded the mod-two cohomology of many sporadic simple groups, the cohomology of 
%alternating groups and linear groups over finite fields at their characteristic have resisted calculation.  
We present the mod-two cohomology of alternating groups in Theorem~\ref{T:presentation} by 
giving generators and relations using cup product as well as a restriction coproduct and transfer
product associated to the standard inclusion $\alt_n \times \alt_m \to \alt_{n+m}$.   We give an explicit additive basis for mod-two cohomology compatible
with this presentation, and also address cup product structure for individual alternating groups.

We show that our product and coproduct structures on the direct sum of cohomology of series of groups such as alternating groups comprise an ``almost 
Hopf semi-ring'' structure (lacking one of two bialgebra structures -- see Definition~\ref{D:distributivity}) and we use this structure along with a canonical involution
to propagate cohomology and establish relations.
Such a suite of product and coproduct
structures forms a Hopf semi-ring in the setting of symmetric groups,
as first developed by Strickland and Turner \cite{StTu97}.  Their definition generalizes the induction product in representation theory
developed by Atiyah \cite{Atiy66} and Zelevinsky  \cite{Zele81}.  Taking advantage of such additional structure on  a direct sum 
of (generalized) cohomology rings has 
been fruitful in a number of settings, and there are connections between our definitions and homology of configuration spaces more
broadly \cite{Chur12,Knud16} as well as conjecturally in the study of Hilbert schemes \cite{LeSo01}.

We determined this Hopf semi-ring structure for symmetric groups in \cite{GSS12}, which in turn allows us here to 
calculate the Gysin sequence relating cohomology of alternating and symmetric groups in order to find additive bases.
Determining multiplicative structure is much more technically challenging than in the symmetric group setting, 
requiring a range of new techniques.  
We use resolutions based on geometric ideas of  Fox and Neuwirth \cite{FoNe62, GiSi12} to produce generating cohomology classes.  
We closely analyze the failure of the second bialgebra structure and develop a notion of charge to salvage product-coproduct compatibility.
We further develop the involution which switches charge, integrating it into our multiplicative structures.  
%We develop an ``almost Hopf ring'' structure for more general series of groups, % with embeddings of products, 
%which by Theorem~\ref{bialg} is very close a Hopf ring structure for alternating groups.  
%It seems likely that other series will deviate further from satisfying the Hopf ring axioms.
And most substantially
we show that  restriction to the cohomology elementary abelian $2$-subgroups is injective, and calculate these restriction maps  to
establish a full set of relations.  Our detection result
implies that there are no nilpotent elements in mod-two cohomology of alternating groups, which
 resolves a long-standing open question.
%Throughout, our understanding of the cohomology of symmetric groups as a Hopf ring \cite{GSS12} provides essential input.

%Another technique we develop is
%cochain-level models based on ideas of  Fox and Neuwirth \cite{FoNe62, GiSi12}.  We combine
%these novel techniques with new calculations of the  Gysin sequence 
%relating alternating and symmetric groups and of restriction maps to elementary abelian subgroups.  
%The Gysin sequence analysis alone suffices to resolve additive structure, the first such determination.

Once the technicalities are overcome, we can can draw  
parallels between the presentations of  cohomology for alternating and symmetric groups, given in
 Theorems  \ref{T:presentation} and \ref{T:symmgenandrel}.
The generators predominantly map to (sums of) one  
another in the Gysin sequence.  There is a notion of level for generators, namely the $\ell$ in $\gamma_{\ell,m}$, and
while the description of cohomology for symmetric groups is uniform, that of alternating
groups is irregular for small levels. 
At  level three or greater, there are two Hopf ring generators for alternating
groups for each generator for symmetric groups.  These two generators map to one another under the involution and annihilate each other under cup product.  
This structure is unstable, in that only the sum of such pairs lift to larger alternating groups.  At level two, there are still two sets of generators, but
instead of annihilating each other under cup product there are exceptional relations which start at $\alt_4$ and
then propagate in a sense, as determined by coproduct structure.
Finally, at level one only one set of generators occurs, but we need a separate set of generators for each component.
In comparison with symmetric groups that set lacks the degree-one generator, which is the Euler class in the Gysin sequence, a
central character in the story. 

Transfer products are relatively simple, with relations mostly 
governed by the notion of charge.  In particular, the transfer product of  classes fixed under involution vanishes.
Cup products are complicated, with complexity driven by both basic relations and Hopf ring structure.
The presence of basic relations is
in contrast to the setting of symmetric groups, where all cup product
relations follow from Hopf ring distributivity along with the fact that cup product vanishes between
cohomology of different symmetric groups.

Much as  ring structure determines additive structure, our description using two products determines cup product ring structure alone,
for example yielding an elementary algorithm for finding ring generators.  We carry out a calculation for $\alt_8$, 
giving an isomorphism with the computer calculation of \cite{King10} and discussing discrepancy with \cite{AdMi94}.  We also explain why 
techniques which yield more 
global results for symmetric groups do not apply in this setting.  
We  calculate Steenrod operations on our generators, which determines the global structure.

We develop our four distinct techniques for understanding the cohomology of alternating groups 
in the next four sections, before applying them for  calculations in the last five sections.  In comparison with symmetric groups, alternating
groups present substantial technical challenges at each step of calculation.

We thank Alejandro Adem and Paolo Salvatore for helpful comments throughout our time working on this project, and two referees for careful readings
and useful comments.
%\pagebreak

%%% Local Variables:
%%% TeX-master: "AltGroupsModTwo.V3.tex"
%%% End: % Introduction
\tableofcontents
% !TEX root = AltGroupsModTwo.V3.tex

\section{Product and coproduct structures on series of groups}\label{almosthopf}

\begin{definition}
A product series of finite groups is a collection $\{ G_i \}_{i \geq 0}$ with embeddings $e_{n,m} : G_n \times G_m \hookrightarrow G_{n+m}$
which are associative and commutative up to conjugation, and with $G_0$ the trivial group.
\end{definition}

Examples include symmetric groups, general linear groups over finite fields,  series of Coxeter groups, as well as alternating groups.  
We consider alternating groups in a slightly different way.  The mod-two cohomology of an odd alternating group is isomorphic to that of the preceding even group, so  we 
 set $G_i = \alt_{2i}$ to define a product series.  For the rest of this paper, when we refer to $\alt_n$ we assume $n$ is even.  
For work on cohomology of alternating groups at odd primes, % and thus considering odd alternating groups,
commutativity up to conjugation does not hold. 
The lack of commutativity on cohomology  is controlled by a standard involution, which also plays a substantial
role in the mod-two setting, as we develop in Section~\ref{involution}.

\begin{definition}\label{D:distributivity}
An almost-Hopf semi-ring is a graded vector space $V$ with two associative, commutative products $\odot$ and $\cdot$, and a cocommutative
coproduct $\Delta$
so that $(\cdot, \Delta)$ defines a bialgebra, and $\cdot$ distributes over $\odot$ with respect to the coproduct.  Explicitly, distributivity
means the following diagram commutes, where $\mu_\odot$ and $\mu_\cdot$ are the bilinear maps which correspond to the multiplications 
and
$\tau$ is the twist map that exchanges the middle two factors:
$$\xymatrix{
V^{\otimes 3}  \ar[d]^-{id \otimes \mu_\odot } \ar[rr]^-{\tau \circ (\Delta \otimes id)}
& & V^{\otimes 4}    \ar[r]^-{\mu_\cdot \otimes \mu_\cdot} &  V \otimes V \ar[d]_-{\mu_\odot} \\
V \otimes V  \ar[rrr]^-{\mu_\cdot} & & &  V. \\
}$$

%$$\xymatrix{
%V\otimes V \ar[dd]^-{\odot}
%&V\otimes V\otimes V \otimes V\ar[l]_-{\cdot \otimes \cdot}\\
%&V \otimes V \otimes V \ar[d]^-{\text{id}\otimes\odot}\ar[u]_-{(\text{id}\otimes T \otimes\text{id})\circ(\Delta\otimes\text{id}\otimes\text{id})}
%\\
%.V&V\otimes V\ar[l]^-{\cdot}
%}$$
\end{definition}

In formulas, distributivity means 
$$a \cdot (b \odot c) = \sum_{\Delta a = \sum a'_i \otimes a''_i}  (a'_i \cdot b) \odot (a''_i \cdot c).$$

We say ``semi-ring'', also called ``rig'' (ring without negatives), because we do not require an antipode for the additive product.
Our calculations are all mod-two, so we have no signs in formulae.   To interpret these basic definitions in general, we assume that signs are encorporated
in definitions, for example as part of the definition of the twist map above.

Distributivity facilitates inductive calculations, especially in the graded setting, and implies 
that there are bases of the form $p_1 \odot p_2 \odot \cdots \odot p_i$ where each $p_i$ is a product with respect to the $\cdot$ 
multiplication.  We call such a Hopf monomial basis, and we call the $p_i$ the constituent $\cdot$-monomials.

A Hopf semi-ring is a semi-ring object in the category of coalgebras, which entails all of the above and a Hopf algebra structure for $(\odot, \Delta)$.
(We do not know  a categorical definition for almost Hopf semi-rings.)
Strickland and Turner \cite{StTu97} show that
the generalized cohomology of symmetric groups forms a Hopf semi-ring, which inspires the following. %which we used extensively in \cite{GSS12}.  
%General product series of groups only have an almost Hopf semi-ring structure.  
Recall that if $H$ is a finite index subgroup of $G$ the the induced map of the inclusion is a covering map, as seen
clearly through the model $BH = EG/H \to BG = EG/G$.

\begin{definition}\label{basic}
Let $\{ G_i \}_{i \geq 0}$ be a product series of finite groups.  Define two products and a coproduct on $\bigoplus_i H^*(BG_i)$ as follows:
\begin{itemize}
\item $\odot$ is the transfer map in cohomology associated to the cover $Be_{i,j}  : BG_i \times BG_j \to BG_{i + j}$.  
We call this the transfer or induction product.  The unit is given by the unit in the cohomology of $BG_0$.
\item $\cdot$ is the cup product, defined to be zero between different summands.  The unit is the direct sum of all unit classes in the cohomology of 
$BG_i$.
\item $\Delta_{i,j}$ is the natural map associated to the cover $Be_{i,j}$, and $\Delta=\bigoplus \Delta_{i,j}$, which defines a coproduct.
The counit is the projection onto the cohomology of $BG_0$.
\end{itemize}
\end{definition}

Though we only calculate mod-two cohomology in this paper, we record the following in general.  For fields other than $\F_2$ the
tensor product, transposition maps, etc. are defined in the category of graded vector space, equipped with appropriate signs.

\begin{theorem}\label{hopfring}
The direct sum of cohomology with field coefficients of a product series of finite groups, $\bigoplus_i H^*(BG_i, k)$, forms an almost Hopf semi-ring.
\end{theorem}

\begin{proof}
Recall for example from Theorem~II.1.9 of \cite{AdMi94} 
that conjugation by $G$ on  the standard simplicial model of $BG$ induces the trivial map on group cohomology.  Thus coassociativity
and cocommutativity of the coproduct $\Delta$ are  immediate
from the assumption of associativity and commutativity of product maps up to conjugation.

As transfer maps commute with isomorphism of covering spaces, associativity and 
commutativity respectively of the transfer product follow from the fact that the 
conjugation isomorphisms between the two copies of $G_n \times G_m \times G_p$ (respectively $G_n \times G_m$) in $G_{n + m + p}$ (respectively
$G_{n+m}$)  define isomorphisms of covering spaces of  $BG_{n + m + p}$ (respectively $BG_{n+m}$).

Because the coproduct is induced from a map of spaces, it forms a bialgebra with cup product.

For distributivity, consider the following diagram.

$$\xymatrix{
BG_i \times BG_j   \ar[d]^-{Be_{i,j}} \ar[rr]^-{D_{BG_i} \times D_{BG_j} } 
 & & BG_i \times BG_i \times BG_j \times  BG_j     \ar[rr]^-{(e_{i,j} \times id) \circ \tau}  & & BG_{i+j}  \times BG_i \times BG_j \ar[d]_-{id \times Be_{i,j}} \\
BG_{i+j}  \ar[rrrr]^-{D_{BG_{i+j}}}  & & & & BG_{i+j} \times BG_{i+j},
}$$
%$$\xymatrix{
%BG_i \times BG_j \ar[dd]\ar[r]%^-{(\text{id}\times T \times \text{id})\circ(\Delta \times \Delta)}
%&BG_i \times BG_i \times BG_j \times BG_j \ar[d]\\ &BG_{i+j} \times BG_i \times BG_j\ar[d] \\
%BG_{i+j} \ar[r]^-{\Delta}&BG_{i+j} \times BG_{i+j}
%}$$
where in general $D_X$ denotes the diagonal map on $X$.  The vertical maps are covering maps.  Moreover,
the composite of the top horizontal maps and the bottom horizontal map define a pull-back of covering maps.
Taking cohomology and applying natural maps horizontally and transfer maps vertically 
yields a commutative diagram because transfers commute with natural maps for pull-backs.  
Taking cohomology yields the distributivity diagram of Definition~\ref{D:distributivity} (reflected across a vertical line). 
%(Notice for example that $V^{\otimes 3}$ is in the upper left corner of the distributivity diagram and corresponds to 
%$BG_{i+j} \times BG_i \times BG_j$ in the upper right corner of this diagram.  Note as well that $D_X$ induces the cup product $\cdot$ on cohomology.)
\end{proof}

Because transfers exist in generalized cohomology theories, these structures translate to that setting, though as usual coproduct structures can be 
problematic when the ground ring is not a field. 
Strickland and Turner developed these structures to study Morava $E$-theory of symmetric groups \cite{StTu97}. 
    
We use extensively the fact that transfer maps for covering spaces commute with natural maps for pull-backs,
especially for covering maps are between classifying spaces of finite groups.   So in the next proposition we record some standard facts about such pull-backs.

\begin{proposition}\label{pullback}
For a finite group $G$ and subgroups $H, K$, choose models for $BH$ and $BK$ as $EG/H$ and $EG/K$ so that their maps to 
$BG$ are covering maps.  Then a model for the (homtopy) pull-back 
$$\xymatrix{
PB  \ar[r]  \ar[d] & BK \ar[d]^{B\iota_K}\\
BH \ar[r]_{B\iota_H}  &BG
}$$
is given by $PB = (EG \times G) / (H \times K)$ where $H$ acts by $h \cdot (e,g) = (he, hg)$ and $K$ acts by $(e,g) \cdot k = (e, gk^{-1})$.

The maps from the pull-back are by defined by identifying for example $BK = (EG \times G)/ (G \times K)$ where $G$ acts diagonally.

%The fiber of $f$ is $G/K$ as an $H$-set.

 The components of the pull-back are indexed by double-cosets $H \backslash G / K$, and the component indexed by $H g K$ is 
 $B(H \cap g K g^{-1})$.
\end{proposition}

For reference, we give a second proof of the result first established by Strickland and Turner \cite{StTu97}.  

\begin{theorem}\label{intersectionprop}
The cohomology of symmetric groups with field coefficients  is a Hopf semi-ring, extending the almost Hopf semi-ring structure 
of Theorem~\ref{hopfring}.
\end{theorem}

\begin{proof}
After Theorem~\ref{hopfring} we need only check that $\odot$ and $\Delta$ form a bialgebra.  
For symmetric groups, the intersection of $\si_n \times \si_m$ with conjugates of $\si_i \times \si_j$ in $\si_d$ (where 
$d = n+m = i+j$)  are all possible 
$\si_p \times \si_q \times \si_r \times \si_s$ with $p+q = n$, $r+s = m$, $p + r = i$ and $q + s = j$.
By Proposition~\ref{pullback}, the following diagram is thus a pull-back square of covering spaces
$$\xymatrix{
\bigsqcup B\si_p \times B\si_q \times B\si_r \times B\si_s \ar[r]  \ar[d]
&B\si_{n}  \times B\si_m \ar[d] \\
B\si_i \times B\si_j \ar[r]
&B\si_{d}.
}$$
 Starting at $H^*(B\si_i \times B\si_j)$ and mapping to $H^*(B\si_n \times B\si_m)$ by
composing restriction and transfer in
two ways, which agree because this is a pull-back, establishes the result.
\end{proof}

The analogous product $\odot$ and coproduct $\Delta$ in the cohomology of alternating groups do not form a bialgebra, but we 
control the discrepancy in Theorem~\ref{bialg} and find their behavior just as good for purposes of calculation.
There is also an antipode for the transfer product in the cohomology of symmetric groups, 
as can be seen through a divided powers structure \cite{GSS19}, but we
do not require it.  

These almost Hopf semi-ring structures have substantial connections to other product structures in the literature, and 
we expect applications.
When one applies the Strickland-Turner result to $K$-theory of symmetric groups, which by the Atiyah-Segal theorem is the completion
of their representation ring, the transfer product coincides with induction
product, and the coproduct is given by restriction.  These were first studied together as a Hopf algebra by Zelevinsky \cite{Zele81}, 
but the full Hopf semi-ring structure was not utilized.  Given the substantial impact of Hopf semi-ring structure on group cohomology, we expect 
it would be fruitful to more fully develop in representation theory, especially in areas such as the theory support varieties which are at the 
interface.  Nick Proudfoot has also conjectured that this Hopf semi-ring structure compatible
with product structure on the direct sum of cohomology of Hilbert schemes \cite{LeSo01}.

In \cite{GSS12} we show that the transfer product structure for cohomology of symmetric groups, which as we use below is modeled by configurations 
in $\R^\infty$, is defined for other configuration spaces.  
This structure is part of recent descriptions of rational homology of unordered configurations by Knudsen 
\cite{Knud16}.
Transfer product with unit classes for cup products can also be used to split (co)homological
stability isomorphisms \cite{Chur12}.  

 For general product series of groups including alternating groups, 
the coproduct and transfer product   do not define a bialgebra.  
Nonetheless, the two products  bind the cohomology of the $G_i$, 
and distributivity provides control up to computability of the coproduct.  As in the Hopf semi-ring setting,
 have inductive computability of coproduct for alternating groups,
as  transfer product and coproduct 
are close to forming a bialgebra as described in Theorem~\ref{bialg}, whose proof is a modification of that of 
Theorem~\ref{intersectionprop}.  Even in cases where coproduct is not likely inductively computable, such as general linear 
groups over finite fields, we expect  almost Hopf semi-ring structure to be useful.

%%% Local Variables:
%%% TeX-master: "AltGroupsModTwo.V3.tex"
%%% End: % Product and coproduct structures on series of groups
% !TEX root = AltGroupsModTwo.V3.tex

\section{Relationships between the cohomology of alternating and symmetric groups}\label{SymmGroupSection}

\subsection{The cohomology of symmetric groups}  

Our foundation is a thorough understanding of cohomology of symmetric groups, the focus  of
 \cite{GSS12} whose main result  is the following.  
 
 %\pagebreak

\begin{theorem}\label{T:symmgenandrel}
As a Hopf semi-ring, $\bigoplus_{n \geq 0} H^{*}( B\si_{2n}; \F_{2})$ is generated under the products $\cdot$ and $\odot$ from
Definition~\ref{basic} by 
classes $\gamma_{\ell, m} \in H^{m(2^{\ell} - 1)}(B \si_{ m 2^{ \ell}})$,
with $\ell, m \geq 0$, where   $\gamma_{\ell,0} = 1_0$ is the unit for transfer product
and $\gamma_{0,m} = 1_m$ is the unit for cup product on component $m$.
The coproduct of $\gamma_{\ell, m}$ is given by
$$\Delta \gamma_{\ell, m} = \sum_{i+j = m} {\gamma_{\ell, i}} \otimes {\gamma_{\ell, j}}.$$
Relations between transfer products of these generators are given by 
$$\gamma_{\ell, n} \tr \gamma_{\ell, m} = \binom{n+m}{n} \gamma_{\ell, n+m},$$
which implies that the $\gamma_{\ell, 2^k}$ constitute a set of Hopf semi-ring generators.
Cup products of generators
on different components are zero, and there are no other relations between cup products of  these generators.
\end{theorem}

%We will prefer the smaller Hopf generating set of the $\gamma_{\ell, 2^k}$.
We will regularly refer to this and other results from Sections~5~and~6  of  \cite{GSS12}.  Just as in a ring one can always
reduce to monomials, in a Hopf (semi-)ring, one can reduce to Hopf monomials, formed by taking $\odot$ products of $\cdot$-products
of generators.
In Section~6 of \cite{GSS12} we give a convenient graphical representation of the Hopf monomial basis
which we call skyline diagrams, which give unique representations in this setting.  
We will primarily argue using Hopf ring monomial language, but we 
indicate as we go along how our calculations would look using skyline diagrams, and illustrate such in figures. 

In forthcoming work with Guerra and Salvatore, we identify a divided powers structure associated to transfer product, which streamlines
the description above as the free divided powers Hopf ring on classes $\gamma_\ell$, which correspond to our $\gamma_{\ell, 1}$, and
the divided powers $\gamma_{\ell[m]}$ correspond to our $\gamma_{\ell,m}$.

%For example, the following notion
%describes the size of the ``grounding blocks'' of a diagram.

%The scale is the smallest symmetric group which occurs in the image of a non-trival iterated coproduct.

\subsection{Restriction and transfer maps}

As alternating groups are subgroups of symmetric groups, there are both restriction maps and transfers relating their cohomology.

\begin{definition}
If $H$ is a subgroup of $G$ with inclusion map $\iota$ understood 
we let $res$ denote the natural restriction map ${B \iota}^*$ on cohomology
and let $tr$ denote the transfer map ${B \iota}^!$ on cohomology.  

We let $\iota_n$ denote the standard inclusion of $\alt_n$ in $\si_n$.  
\end{definition}

The almost-Hopf semi-ring structures developed in the previous section are for the most part compatible with these maps.

\begin{proposition}\label{transfercompatible}
Restriction maps $res: H^*(B\si_n) \to H^*(B\alt_n)$ preserve coproducts, and transfer maps $tr: H^*(B\alt_n) \to H^*(B\si_n)$ preserve transfer products.
\end{proposition}

\begin{proof}

Consider the commuting square of covering maps
$$\xymatrix{
B\alt_i \times B\alt_j \ar[r]^-{Be_{i,j}}  \ar[d]_-{ B \; \iota_i \times \iota_j }
&B\alt_{i+j} \ar[d]^-{ B \iota_{i+j}}\\
B\si_i \times B\si_j \ar[r]_-{Be_{i,j}}
&B\si_{i+j}.
}$$
That natural maps in cohomology commute for squares of spaces gives the first result, and that 
 transfer maps  in a square of covering maps commute gives the second result. \end{proof}

\begin{remark} This diagram does not define a pull-back of covering spaces.  
Instead, the space $B\alt_i \times B\alt_j$ is a double cover of the pullback, 
reflecting the fact that $\alt_i \times \alt_j$ is of index two in  $\alt_{i+j} \cap (\si_i \times \si_j)$.
 We use this fact in Proposition~\ref{transferneutral} to prove the vanishing of transfer products of classes restricted from 
the cohomology of symmetric groups. This fact is also related to the failure of $\Delta$ and $\odot$ to form a bialgebra.
\end{remark}

Transfer maps do not preserve cup products in general, though we will see from the main calculation of Theorem~\ref{T:presentation} that 
they do preserve cup products for classes ``of uniform charge.''

Restriction maps of course preserve cup products, since they are defined by maps of spaces.
While restriction maps do not preserve transfer products, there is some compatibility.

\begin{proposition}\label{restrictiontransfer}
$$res(\alpha) \odot \beta = res(\alpha \odot tr(\beta)).$$
\end{proposition}

\begin{proof}
Apply restriction maps horizontally and transfer maps vertically to the following diagram of covering maps, which is a pullback.
$$\xymatrix{
B\si_i \times B\alt_j \ar[d]_-{id \times B \iota_j}
&B\alt_i \times B\alt_j \ar[l]^-{B \iota_i \times id} \ar[dd]^-{Be_{i,j}}\\
B\si_i \times B\si_j\ar[d]_-{Be_{i,j}}\\
B\si_{i+j}
&B\alt_{i+j}\ar[l]^-{B \iota_{i+j}}\\
}$$
\end{proof}

\subsection{Cohomology of double covers}\label{doublecovers}

The standard restriction and transfer maps between the cohomology of alternating and symmetric groups give rise to a short exact Gysin sequence 
whose analysis forms the backbone of  our calculations.    We record  basic results about such Gysin sequences
here, since they govern our passage from understanding of cohomology of
symmetric groups to that of alternating groups.

Let $p: E \to B$ be a two-fold covering map, which we consider as a principal $C_2$-bundle.  
Since $C_2 \cong O(1)$, we can apply the Gysin sequence for the associated line bundle, which reads
\begin{equation*}%\label{longexact}
\cdots \overset{\cdot \; e}\to H^k(B) \overset{\text{res}}\longrightarrow H^k(E) \overset{\text{tr}}\to H^k(B) \overset{\cdot \; e}\longrightarrow 
H^{k+1}(B) \overset{\text{res}}\longrightarrow \cdots,
\end{equation*}
where $e$ is the Euler class of the line bundle.

Decomposing into short exact sequences yields
\begin{equation}\label{shortexact}
0 \to H^{*}(B )/ e \overset{p^*}\longrightarrow H^{*}(E) \overset{\text{tr}}\longrightarrow {\rm Ann}(e) \to 0,
\end{equation}
where ${\rm Ann}(e)$ is the annihilator ideal.

As a principal $C_2$ bundle, the cohomology of $H^{*}(E)$ has a $C_2$ action, which we  denote by $x \mapsto \overline{x}$.  
Moreover, the restriction and transfer maps are equivariant,
with the cohomology of the base having trivial action.  %We will see how   

Over $k = {\mathbb F}_2$ a finitely-generated $C_2$-representation is a direct sum of  the  trivial  representation $k$
and  the regular representations $\rho$.
Moreover, a short exact sequence $0 \to A \to B \to C \to 0$ decomposes as direct sum of isomorphisms, $k \to k$ or $\rho \to \rho$, and
the exact sequence $0 \to k \to \rho \to k \to 0$.
By decomposing the Gysin  sequence in this way  we determine $H^*(E)$ as a $C_2$-representation.

\begin{definition}
A Gysin decompsition for $H^*(B)$ in a Gysin sequence as in Equation~(\ref{shortexact}) 
is a linearly independent set $\mathcal{G}$ which is the union of two possibly overlapping subsets 
$\mathcal{G} = \mathcal{G}_{quot} \cup \mathcal{G}_{ann}$,
where
\begin{itemize}
\item $\mathcal{G}_{quot}$ maps to a basis of  $H^{*}(B )/ e$, 
\item $\mathcal{G}_{ann}$ is a basis for the annihilator ideal ${\rm Ann}(e)$
\item  $\mathcal{G}_{quot} \cap \mathcal{G}_{ann}$ maps to a basis for
the image of the annihilator ideal in  $H^{*}(B )/ e$.
\end{itemize}
\end{definition}

\begin{definition}
Given a set $\mathcal{G}$ which is a union $\mathcal{G} = \mathcal{G}_1 \cup \mathcal{G}_2$, let 
$\mathcal{B}_o = \mathcal{G} \backslash (\mathcal{G}_{1} \cap \mathcal{G}_{2})$ and $\mathcal{B}_+ = \mathcal{G}_{1} \cap \mathcal{G}_{2}$.
Define $V_\mathcal{G} = (\bigoplus_{\mathcal{B}_o } k) \oplus  (\bigoplus_{\mathcal{B}_+} \rho).$

We call the two summands of $V_{\mathcal{G}}$ the $k$ and the $\rho$ summands.  
\end{definition}

A basis for $V_{\mathcal{G}}$ as a vector space is thus $\mathcal{B}_o  \cup \mathcal{B}_+  \cup \mathcal{B}_-$, where 
$\mathcal{B}_-$ is the image of $\mathcal{B}_+$ under involution.

\begin{proposition}\label{gysinsplit}
If  $\mathcal{G} =  \mathcal{G}_{quot} \cup \mathcal{G}_{ann}$ is a Gysin decompsition for $H^*(B)$  then as 
a $C_2$ representation, $H^*(E) \cong V_{\mathcal G}$. 
\end{proposition}

\begin{proof}
For  $y \in \mathcal{G}_{quot} \cap \mathcal{G}_{ann}$, 
we let $x \in H^*(E)$ be a class which maps to it under transfer.  At the cochain level $x + \overline{x}$
is the image of $y$ under $p^*$, which is non-zero since $y \in \mathcal{G}_{quot}$.  
Thus $\overline{x} \neq x$.  Letting $\langle A \rangle$ be the span of a set $A$  we define
 $ S_y$, a subsequence of the Gysin sequence, to be 
 $ 0 \to \langle y \rangle \to \langle x, \overline{x} \rangle \to \langle y  \rangle \to 0$ with $y \mapsto x + \overline{x}$.  

For $y \in \mathcal{G}_{quot}  \backslash (\mathcal{G}_{quot} \cap \mathcal{G}_{ann})$ we define the subsequence $ S_y$ to be $ 0 \to \langle y \rangle
\to \langle p_*(y) \rangle \to 0 \to 0$
and for $y \in \mathcal{G}_{ann}  \backslash (\mathcal{G}_{quot} \cap \mathcal{G}_{ann})$ we set 
$S_y = 0 \to 0 \to \langle x \rangle  \to \langle  y \rangle  \to 0$ where $x$ maps
to $y$ under transfer.  Since 
$\mathcal{G}_{quot}$  and $\mathcal{G}_{ann}$ are bases for its first and last terms, the 
sequences  $\bigoplus_{y \in  \mathcal{G}} S_y$ span  the Gysin sequence, splitting it by the linear independence of $\mathcal{G}$.  
This splitting yields
the stated presentation of $H^*(E)$.
\end{proof}

\subsection{The Gysin sequence for alternating and symmetric groups}\label{gysinsection}

We apply the ideas of the previous section to the Gysin sequence relating the cohomology of alternating and symmetric groups, which reads
\begin{equation*}%\label{shortexact}
0 \to H^{*}(B \si_{n})/ e \overset{\text{res}}\longrightarrow H^{*}(B \alt_{n}) \overset{\text{tr}}\longrightarrow {\rm Ann}(e) \to 0,
\end{equation*}
where ${\rm Ann}(e)$ is the annihilator ideal.

%Because it is induced by a map of spaces, the restriction map preserves cup products, but we will see in Proposition~\ref{transferneutral}
%that it does not preserve transfer products.  On the other hand, by Proposition~\ref{transfercompatible} the transfer map preserves transfer products,but not cup products.  

%To calculate this exact sequence, we appeal to our understanding of the cohomology of symmetric groups, using the Hopf semi-ring structure.  
%While this gives a clean description as a Hopf semi-ring, from which it is straightforward to produce an additive basis with multiplication rules,
%Understanding  generators and relations for cohomology of individual symmetric groups under cup product alone 
%is straightforward from this presentation, but is combinatorially involved.  
%Fortunately, it is manageable to use the Hopf semi-ring description for symmetric groups to produce a similar description 
%for alternating groups, and we lay the foundations for that approach in this section.

When $n=2$ the associated line bundle in the Gysin sequence is the tautological line bundle over $B\si_2 = \R P^\infty$, and  in the
notation of Theorem~\ref{T:symmgenandrel}
the Euler class $e$ is $\gamma_{1,1}$, 
the generator of the cohomology.   For general $n$, the Euler class must 
restrict to this, so $e = \gamma_{1,1} \odot 1_{n-2}$.  We record how it multiplies, which is immediate from Hopf semi-ring distributivity,
or can be read off from Theorem~6.8 of \cite{GSS12}.  Recall that  gathered monomials are those expressed with the minimal number of $\odot$ products,
and that these form a basis.

\begin{lemma}\label{eMultiply}
The cup product of $e = \gamma_{1,1} \odot 1_{n-2}$ with a 
gathered Hopf ring monomial produces
a linear combination of monomials where each occurrence of ${\gamma_{1,m}}^k$  (for $k \geq 0$, $m>0$) as a constituent
cup monomial is replaced by ${\gamma_{1,1}}^{k+1} \odot {\gamma_{1,m-1}}^{k}$.
\end{lemma}

\begin{figure}[htp]
\begin{center}
\begin{tikzpicture}[line cap=round,line join=round,x=1.0cm,y=1.0cm, scale=1]
\draw [line width=1pt, color=black] (0,0) -- (1,0) -- (1,1) -- (0,1) -- cycle;
\draw [line width=1pt, color=black] (1,0) -- (3,0);
%\draw [line width=1pt, dash pattern = on 3pt off 3pt, color=black] (1,1.75) -- (1,2.75);
%\draw [line width=1pt, dash pattern = on 3pt off 3pt, color=black] (2,1.75) -- (2,2.75);
%\draw [line width=1pt, dash pattern = on 3pt off 3pt, color=black] (3,1.75) -- (3,2.75);

\draw [line width=1pt, color=black] (2.9,0.7) -- (2.92,0.7) -- (2.92,0.72) -- (2.9, 0.72) -- cycle;

\draw [line width=1pt, color=black] (3.5,0) -- (4.5,0) -- (4.5,1) -- (3.5,1) -- cycle;
\draw [line width=1pt, color=black] (3.5,1) -- (4.5,1) -- (4.5,2) -- (3.5,2) -- cycle;
\draw [line width=1pt, color=black] (3.5,2) -- (4.5,2) -- (4.5,3) -- (3.5,3) -- cycle;
\draw [line width=1pt, color=black] (4.5,0) -- (5.5,0) -- (5.5,1) -- (4.5,1) -- cycle;
\draw [line width=1pt, color=black] (5.5,0) -- (6.5,0);

\draw [line width=1pt, color=black] (6.3,0.9) -- (6.8,0.9);
\draw [line width=1pt, color=black] (6.3,1.1) -- (6.8,1.1);

\draw [line width=1pt, color=black] (7,0) -- (8,0) -- (8,1) -- (7,1) -- cycle;
\draw [line width=1pt, color=black] (7,1) -- (8,1) -- (8,2) -- (7,2) -- cycle;
\draw [line width=1pt, color=black] (7,2) -- (8,2) -- (8,3) -- (7,3) -- cycle;
\draw [line width=1pt, color=black] (7,3) -- (8,3) -- (8,4) -- (7,4) -- cycle;
\draw [line width=1pt, color=black] (8,0) -- (9,0) -- (9,1) -- (8,1) -- cycle;
\draw [line width=1pt, color=black] (9,0) -- (10,0);

\draw [line width=1pt, color=black] (9.5,1) -- (9.9, 1);
\draw [line width=1pt, color=black] (9.7,0.8) -- (9.7, 1.2);

\draw [line width=1pt, color=black] (10.5,0) -- (11.5,0) -- (11.5,1) -- (10.5,1) -- cycle;
\draw [line width=1pt, color=black] (10.5,1) -- (11.5,1) -- (11.5,2) -- (10.5,2) -- cycle;
\draw [line width=1pt, color=black] (10.5,2) -- (11.5,2) -- (11.5,3) -- (10.5,3) -- cycle;
\draw [line width=1pt, color=black] (11.5,0) -- (12.5,0) -- (12.5,1) -- (11.5,1) -- cycle;
\draw [line width=1pt, color=black] (11.5,1) -- (12.5,1) -- (12.5,2) -- (11.5,2) -- cycle;
\draw [line width=1pt, color=black] (12.5,0) -- (13.5,0);

%\draw [line width=1pt, color=black] (4,0) -- (6,0) -- (6,1.5) -- (4,1.5) -- cycle;
%\draw [line width=1pt, color=black] (4,1.5) -- (6,1.5) -- (6,3) -- (4,3) -- cycle;
%\draw [line width=1pt, color=black] (4,3) -- (6,3) -- (6,4) -- (4,4) -- cycle;
%\draw [line width=1pt, dash pattern = on 3pt off 3 pt, color=black] (5,3) -- (5,4);
\end{tikzpicture}
\caption{$(\gamma_{1,1}\odot 1_4) \cdot (\gamma_{1,1}^3 \odot \gamma_{1,1} \odot 1_2) = \gamma_{1,1}^4 \odot \gamma_{1,1} \odot 1_2 + 
\gamma_{1,1}^3 \odot \gamma_{1,1}^2 \odot 1_2$}
\label{F:product}
\label{emult}
\end{center}
\end{figure}

The resulting monomials where there was already a ${\gamma_{1,2k+1}}^{k+1}$ present as a $\odot$-factor would be zero because
$\odot$ is exterior, as happens
for the ``last term'' of the product in Figure \ref{F:product}.

We next recall and extend some definitions from Section~7 of \cite{GSS12}.

\begin{definition}\label{1decomposition}
The level-one sub-Hopf ring of the cohomology of symmetric groups is that generated by the classes $\gamma_{1,m}$ and $1_m$.

The scale greater than one Hopf ring monomials are 
the Hopf ring monomials for which every cup-monomial is a non-trivial product of some $\gamma_{\ell, m}$ for $\ell > 1$ 
(with positive $\cdot$-exponents, so in particular no $\odot$-factors of $1_m$).

The $1$-decomposition of a Hopf ring monomial is its  expression as $m_1 \odot m_2$ where $m_1$ is a level-one monomial and $m_2$ is scale
greater than one.  %A Hopf ring monomial is level-one nontrivial if $m_1$ is non-trivial in this decomposition.
\end{definition}

So a level-one Hopf ring monomial will have $m_2 = 1_0$, and a scale greater than one Hopf ring monomial will have $m_1 = 1_0$.
For example, in Figure~\ref{F:product} all monomials are level-one, and in Figure \ref{F:alpha}  $\alpha$ is scale greater than one.  
In Figure~\ref{F:beta},
$\beta$ has $1$-decomposition $m_1 =  {\gamma_{1,2}}^2 \odot \gamma_{1,1}
\odot 1_2$ and
$m_2 = {\gamma_{2,1}}^2 \gamma_{1,2}$.  

 The  $1$-decomposition is unique if the monomials in question are gathered.  In skyline diagrams, the decomposition has $m_1$ comprised of 
columns with only $1 \times 1$ blocks and empty spaces, while every column of $m_2$ has at least one larger block.

\begin{proposition} \label{levelone}
The level-one Hopf ring is  isomorphic to the Hopf ring of symmetric polynomials in
one variable, so the component rings are all
polynomial.
\end{proposition}

\begin{proof}
This is immediate through the presentation for the Hopf ring of symmetric polynomials given in Proposition~2.8 of \cite{GSS12}.
The isomorphism realized through restriction to subgroups of the form
${\si_2}^k$ by Theorem~7.8 of \cite{GSS12}. 

Explicitly, the isomorphism sends $\gamma_{1,m}$ to the symmetric polynomial which is the product of all of the $m$ variables 
(on the $m$th component), and more generally sends a skyline diagram build from $1 \times 1$ blocks (essentially a Young diagram) to the
symmetrization of the monomial whose powers of the variables are the heights of the columns of the diagram.
\end{proof}

  For example, the multiplication depicted
in Figure~\ref{emult} corresponds to the product of the symmetrization of $x_1$ -- that is $(x_1 + x_2 + x_3)$ -- 
and the symmetrization of ${x_1}^3 x_2$.

%Recall that the scale of a Hopf ring monomial in the cohomology of symmetric groups is the minimum width which occurs in the
%coproduct of the monomial. 
\begin{definition}\label{GysinDef}

Let  $\mathcal{G}_{ann}$  be the set of gathered Hopf ring monomials of scale  greater than one.

Let $\mathcal{G}_{quot}$ denote the set of gathered Hopf ring monomials for which the constituent $\cdot$-monomial
of the form ${\gamma_{1,m}}^k$ with the largest power of $k$ has $m > 1$ or $k=0$.

\end{definition}

\begin{figure}[htp]
\begin{subfigure}{0.4\textwidth}
\begin{center}
\begin{tikzpicture}[line cap=round,line join=round,x=1.0cm,y=1.0cm, scale=1]
\draw [line width=1pt, color=black] (0,0) -- (4,0) -- (4,1.75) -- (0,1.75) -- cycle;
\draw [line width=1pt, color=black] (0,1.75) -- (4,1.75) -- (4,2.75) -- (0,2.75) -- cycle;
\draw [line width=1pt, dash pattern = on 3pt off 3pt, color=black] (1,1.75) -- (1,2.75);
\draw [line width=1pt, dash pattern = on 3pt off 3pt, color=black] (2,1.75) -- (2,2.75);
\draw [line width=1pt, dash pattern = on 3pt off 3pt, color=black] (3,1.75) -- (3,2.75);

\draw [line width=1pt, color=black] (4,0) -- (6,0) -- (6,1.5) -- (4,1.5) -- cycle;
\draw [line width=1pt, color=black] (4,1.5) -- (6,1.5) -- (6,3) -- (4,3) -- cycle;
\draw [line width=1pt, color=black] (4,3) -- (6,3) -- (6,4) -- (4,4) -- cycle;
\draw [line width=1pt, dash pattern = on 3pt off 3 pt, color=black] (5,3) -- (5,4);
\end{tikzpicture}
\caption{$\alpha = \gamma_{3,1}\gamma_{1,4} \odot {\gamma_{2,1}}^2\gamma_{1,2}$}
\label{F:alpha}
\end{center}
\end{subfigure}
\hfill
\begin{subfigure}{0.4\textwidth}
\begin{center}
\begin{tikzpicture}[line cap=round,line join=round,x=1.0cm,y=1.0cm, scale=1]
\draw [line width=1pt, color=black] (0,0) -- (2,0) -- (2,1.5) -- (0,1.5) -- cycle;
\draw [line width=1pt, color=black] (0,1.5) -- (2,1.5) -- (2,3) -- (0,3) -- cycle;
\draw [line width=1pt, color=black] (0,3) -- (2,3) -- (2,4) -- (0,4) -- cycle;
\draw [line width=1pt, dash pattern = on 3pt off 3pt, color=black] (1,3) -- (1,4);

\draw [line width=1pt, color=black] (2,0) -- (4,0) -- (4,1) -- (2,1) -- cycle;
\draw [line width=1pt, color=black] (2,1) -- (4,1) -- (4,2) -- (2,2) -- cycle;

\draw [line width=1pt, dash pattern = on 3pt off 3pt, color=black] (3,0) -- (3,2);
\draw [line width=1pt, color=black] (4,0) -- (5,0) -- (5,1) -- (4,1) -- cycle;

\draw [line width=1pt, color=black] (5,0) -- (6,0);

\end{tikzpicture}
\caption{$\beta = {\gamma_{2,1}}^2 \gamma_{1,2} \odot {\gamma_{1,2}}^2 \odot \gamma_{1,1} \odot 1_2$}
\label{F:beta}
\end{center}
\end{subfigure}
\caption{Skyline diagrams for classes in $H^{*}(B\si_{12})$.}
\end{figure}

%The monomials in  $\mathcal{G}_{ann}$  have all constituent $\cdot$-monomials with a factor of at least one $\gamma_{\ell,m}$ with $\ell > 1$.
%Their skyline diagrams for are those with no columns comprised entirely 
%of $1 \times 1$  blocks, as well as no ``empty spaces.''
Our choices of representatives satisfy $\mathcal{G}_{ann} \subset \mathcal{G}_{quot}$.
The skyline diagrams for classes in $\mathcal{G}_{quot}$ have their tallest pure $1 \times 1$-block column of width
at least  two.  For example, in Figure \ref{F:alpha}, $\alpha \in \mathcal{G}_{ann}$ and thus $ \mathcal{G}_{quot}$, while 
in Figure~\ref{F:beta} $\beta \in  \mathcal{G}_{quot}$.

\begin{theorem}\label{AnnIdeal}
The set $\mathcal{G} = \mathcal{G}_{quot} \cup \mathcal{G}_{ann}$ is a Gysin decomposition  for the cohomology of $B\si_n$ in 
the Gysin sequence for $\alt_n$ as an index-two subgroup.
\end{theorem}

\begin{proof}
The elements of $\mathcal{G}_{ann}$ have  coproduct where no terms are supported on $B\si_2$, so
they will be annihilated by $e$ by Hopf ring distributivity. 

Conversely, consider the $1$-decomposition of a Hopf ring monomial not in $\mathcal{G}_{ann}$, $m = m_1 \odot m_2$, which will thus 
have the level-one factor $m_1 \neq 1_0$. 
Using again that the coproduct of $m_2$ has no terms supported on $B \si_2$, the
 product of $m_1 \odot m_2$ with $e$ will be $((\gamma_{1,1} \odot 1_{k-1}) \cdot m_1) \odot m_2$, where $k$ is the width of $m_1$. 
 Because the component rings of the level-one Hopf subring are polynomial by Proposition~\ref{levelone},   the product 
with $\gamma_{1,1} \odot 1_{k-1}$ is injective on the collection of monomials $m_1$. 
 Thus multiplication by $e$ is injective on the span of monomials not in $\mathcal{G}_{ann}$.
%isomorphic have a constituent monomial of the form ${\gamma_{1,m}}^k$ with $k \geq 0$.  
%By Lemma~\ref{eMultiply}, the product of $e$ with a class
%with such constituent monomials is non-zero, as in particular the term with the greatest $k$ will give rise to a non-zero term.
Since  $e$ annihilates $\mathcal{G}_{ann}$ but  multiplies injectively on its complement in the Hopf monomial basis,  
 $\mathcal{G}_{ann}$ forms a basis for the annihilator ideal as claimed.

To show that $\mathcal{G}_{quot}$ spans the quotient by $e$  we 
induct on the difference $\delta(h)$ between the largest power of $\gamma_{1,1}$ which occurs in a monomial $h$ 
and the second largest power of $\gamma_{1,m}$ for $m \geq 1$ which occur in $h$.   
When $\delta(h)$  is zero or negative, a Hopf monomial is in $\mathcal{G}_{quot}$.

Consider a gathered Hopf ring monomial $h$ for which  ${\gamma_{1,1}}^k$ is the largest power of $\gamma_{1,m}$ which occurs.
Let $h'$ be defined by 
%replacing $\gamma_{1,1}^k$  by $\gamma_{1,1}^{k-1}$ in $h$,  or more generally 
replacing  ${\gamma_{1,1}}^k \odot {\gamma_{1,m}}^{k-1}$, with $m \geq 0$,
by ${\gamma_{1,m+1}}^{k-1}$.  
By Lemma~\ref{eMultiply} the  product of $e$ and $h'$ has $h$ as one term, and other terms with strictly smaller $\delta$.  Inductively, $h$ 
can be written as a sum of monomials in $\mathcal{G}_{quot}$.   (For example, referring to  Figure~\ref{F:product} the monomial 
$h = \gamma_{1,1}^4 \odot \gamma_{1,1} \odot 1_2$ has $\delta(h) = 3$.  When we multiply $h' = \gamma_{1,1}^3 \odot \gamma_{1,1} \odot 1_2$
by $e = \gamma_{1,1} \odot 1_4$ we get $h$ plus a term $\gamma_{1,1}^3 \odot \gamma_{1,1}^2 \odot 1_2$ whose $\delta$ is one.)

To show independence of $\mathcal{G}_{quot}$, we apply Lemma~\ref{eMultiply}.  
Any product of a Hopf monomial $m$ with 
$e$ is either zero, if $m$ has scale greater than one,  or produces at least one term which is not in $\mathcal{G}_{quot}$, namely the one
for which the $\gamma_{1,1}$ factor of $e$ is matched with the highest power of some $\gamma_{1,q}$ which occurs as a constituent $\cdot$-monomial.
This term, which has the greatest
power for $\gamma_{1,1}$ as a $\cdot$-monomial, uniquely determines $m$ among products of monomials with $e$.
Thus the only way for a multiple of $e$ to be in the span of 
$\mathcal{G}_{quot}$ is to be zero, establishing independence of $\mathcal{G}_{quot}$ in the
quotient by $e$.

Observe that $\mathcal{G}_{ann} \subset \mathcal{G}_{quot}$, so it is immediate that their intersection spans the image of the annihilator ideal
in the quotient, as needed to be a Gysin decomposition.
\end{proof}

The bases $\mathcal{G}_{ann}$ and $\mathcal{G}_{quot}$ are readily enumerable.
 To our knowledge, this is the first determination of  the cohomology of alternating groups as vector spaces, and we also determine
 $C_2$ action.
 The calculation trails such knowledge of symmetric groups by Nakaoka \cite{Naka61} by over fifty years, but was relatively short work
 using Hopf ring structure.  
 
\begin{corollary}\label{C:Annzero}
The annihilator ideal $Ann(e)$ is zero when $n = 4k + 2$.
\end{corollary}

At this point, we could describe the cohomology of alternating groups $\alt_{4k+2}$ as quotients of the corresponding cohomology of 
symmetric groups by Euler classes.  The cohomology of $\alt_{4k}$ is much more interesting, and it will be straightforward to understand the cohomology
of $\alt_{4k+2}$ from our  description of the general case.

 \subsection{The standard involution and an extension of almost-Hopf semi-ring structure}\label{involution}
 
We further develop the standard involution on the cohomology of $\alt_n$ coming from the its
 embedding in $\si_n$ as a normal subgroup of index two.

\begin{definition}
Denote by $\overline{x}$ the image of $x \in H^*(B\alt_n)$ under the action of conjugation by any element of $\si_n$ not in $\alt_n$, 
or equivalently by the non-trivial deck transformation of $B\alt_n$ as a cover of $B\si_n$
\end{definition}

The  following are immediate on the level of cochains.
 
 \begin{proposition}\label{P:transfer}
Restriction and transfer from and to the cohomology of symmetric groups are invariant under the standard involution, in that
$\overline{{\text res} ( y)} = {\text res} ( y)$ and ${\text tr}(\overline{x}) = {\text tr}(x)$.  Moreover, ${\text res} ({\text tr} ( x)) = x + \overline{x}$.
 \end{proposition}

We  call classes in the image of restriction neutral, since involution fixes them, and informally at the moment 
call those which have non-zero image under 
transfer charged.  We make charge more precise in two ways later, in which case the involution will reverse charge.

To make full use of this involution, we understand its interplay with  product and coproduct structures.
Let $\iota(x)$ denote the involution in homomorphism notation.

\begin{proposition}\label{P:involutiontransfer}
\begin{enumerate}
\item $ \overline{x \odot y} = \overline{x} \odot y$.
\item  $\Delta(x)$ is invariant under $\iota \otimes \iota$.
\item $\Delta(\overline{x}) = \left(\iota \otimes id \right) (\Delta(x))$ (which equals $\left(id \otimes \iota\right) \left(\Delta(x)\right)$ by the previous).
\end{enumerate}
\end{proposition} 

\begin{proof}
 Use the standard simplicial model for the inclusion $Be_{n,m}$ of $B(\alt_n \times \alt_m)$ in $B\alt_{n+m}$.  Conjugation on $B\alt_{n+m}$
 by any elements not  in $A_{n+m}$, in particular such elements which are in $\si_n \times id$ or $id \times \si_m$, yield the standard
 conjugation action.   
The inclusion $Be_{n,m}$  is thus equivariant up to homotopy with respect to the projection of 
$C_2 \times C_2$ to $C_2$ by quotienting by the diagonal subgroup.  This equivariance yields all
of the stated equalities in cohomology.
\end{proof}

A consequence of part (1) of Proposition~\ref{P:involutiontransfer} is that $(x + \overline{x}) \odot (y + \overline{y}) = 0$.  
By exactness of the Gysin sequence,
each of these factors is 
in the image of the restriction map from the cohomology of symmetric groups.  More generally we have the following.

\begin{proposition}\label{transferneutral}
A transfer product in the cohomology of alternating groups of two  classes restricted from the cohomology of symmetric groups  is zero.  
\end{proposition}

To prove this we first record the following for later use.

\begin{lemma}\label{L:pullpushzero}
Let $K$ be a subgroup of finite groups $G$ and $H$.  Suppose $K \subset K'$ of even index, with $K'$ also a subgroup of $G$ and $H$.
Then the composite $H^*(BG) \overset{res}{\to} H^*(BK) \overset{tr}{\to} H^*(BH)$ is zero on mod-two cohomology.
\end{lemma}

\begin{proof}
By assumption, $BK'$ forms an intermediate cover between $BK$ and both $BG$ and $BH$.  Thus both maps in 
the composite $H^*(BG) \to H^*(BK) \overset{tr}{\to} H^*(BH)$
factor through the cohomology of $BK$ to give
$$H^*(BG) \to H^*(BK') \to H^*(BK) \overset{tr}{\to} H^*(BK') \overset{tr}{\to} H^*(BH).$$
Because $K \subset K'$ of even index, the middle composite %$H^*(BK') \to H^*(BK) \overset{tr}{\to} H^*(BK')$ 
is zero on mod-two cohomology.
\end{proof}

\begin{proof}[Proof of Proposition~\ref{transferneutral}]
By definition, we consider the composite 
$$H^*\left(B(\si_i \times \si_j)\right) \overset{res}{\longrightarrow} H^* \left( B(\alt_i \times \alt_j) \right) \overset{tr}{\longrightarrow} H^* (B\alt_{i+j}).$$
As noted in the proof of Proposition \ref{transfercompatible}, 
$\alt_i \times \alt_j$ is an index two subgroup of the intersection $\si_i \times \si_j \cap \alt_{i+j}$ in $\si_{i+j}$.   
Lemma~\ref{L:pullpushzero}  applies to give the result.
\end{proof}

\begin{corollary}
$x \odot res(y) = \overline{x} \odot res(y)$.
\end{corollary}

Proposition~\ref{transferneutral} has significant consequences for the global structure of the cohomology of alternating groups, 
and in particular the inverse system it forms.
For symmetric groups, one can lift classes in this inverse system by taking transfer products with cup unit classes.
For alternating groups, transfer products of neutral classes with such unit classes 
will result in zero, and the transfer product of a charged class with a unit class yields a lift  of the
sum of the class and its conjugate.  That is, charged classes are inherently unstable, 
and the stability of neutral classes is not realized by transfer 
product structure as it is for symmetric groups.
%This gives rise to  a need for more almost-Hopf semi-ring generators generators,  in comparison to the symmetric group situation.

For cup products, the fact that the diagonal map $B\alt_i \to B\alt_i \times B\alt_i$ is equivariant with respect
to the involution on $B\alt_i$ and the diagonal involution on $B\alt_i \times B\alt_i$ gives the following.

\begin{proposition}\label{P:cupinvolution}
$\overline{x \cdot y} = \overline{x} \cdot \overline{y}$
\end{proposition}

\bigskip

These results lead to a coherent extension of almost Hopf semi-ring structure.  

\begin{definition}\label{extended}
Define $\widetilde{B\alt_0}$ to be $S^0 = \{ +, -\}$. For $n>0$ define the product $\tilde{\mu_{0,n}} : 
\widetilde{B\alt_0} \times B\alt_n \to B\alt_n$ to be the involution on $\{-\} \times B\alt_n$, and the identity map on  $\{+\} \times B\alt_n$.
Define the product $\widetilde{B\alt_0} \times  \widetilde{B\alt_0} \to \widetilde{B\alt_0}$ by the usual law for signed multiplication.

Let $1^+$ and $1^-$ be the corresponding generators of $H^0(\widetilde{B\alt_0})$.  Thus $1^+$ is the unit for transfer product,
and $1^- \odot x = \overline{x}$ so in particular $1^- \odot 1^- = 1^+$. The unit
for cup product on the zeroth component is $1_0 = 1^+ + 1^-$, which is pulled back from the cohomology of symmetric groups.

Let $H^*(B\alt_\bullet) = H^*(\widetilde{B\alt_0}) \oplus \bigoplus_{\substack{m \geq 1}} H^*(B\alt_{2m}).$ 
Extend the definitions of $\odot$ and $\Delta$ by using $\tilde{\mu_{0,n}}$ in place of the usual  $B \alt_0 \times B \alt_n = B\alt_n$.
Define the counit through
projection onto the cohomology of $\widetilde{B\alt_0}$  followed by projection onto the $1^+$ summand.

\end{definition}

\begin{proposition}\label{ConjugationHopfRing}
With maps  as above, $H^*(B\alt_\bullet)$ forms an almost Hopf semi-ring, extending the almost Hopf semi-ring structure on 
 $\bigoplus_{\substack{m \geq 0}} H^*(B\alt_{2m}).$ 
\end{proposition}

\begin{proof}
 Proposition~\ref{P:involutiontransfer} implies that the transfer product with $1^-$ is associative and commutative.
Bialgebra structure of cup product and coproduct is still
immediate because the coproduct is induced by a map of spaces.  
Proposition~\ref{P:cupinvolution} along with the fact that $1^- \cdot 1^+ = 0$ extends Hopf distributivity to apply to transfer products with $1^-$.
\end{proof}

Conversely,  this extended almost Hopf semi-ring structure encodes Propositions~\ref{P:involutiontransfer}  and \ref{P:cupinvolution}.
We can also check compatibility with our other results. Propositions~\ref{transfercompatible} and \ref{restrictiontransfer} extend  by 
Proposition~\ref{P:transfer}, through the fact that $tr(1^-) = tr(1^+)$.   The transfer product of $1_0$ with
the restriction of a class from the cohomology of symmetric groups is zero since such restrictions are  invariant under involution, 
thus extending Proposition~\ref{transferneutral}.

The coproduct of $x$ will now include the terms $1^- \otimes \overline{x} + \overline{x} \otimes 1^-$, making the statement of 
our next major result,
Theorem~\ref{bialg},  more uniform.
%There is now a Gysin sequence for $n=0$, since $BA'_0 \to B\si_0$ is two-sheeted cover, though
%the analysis is different because the covering is now trivial.  (We could also extend to $n=1$, though we do not consider odd $n$ in this paper.) 

\subsection{Coproduct of a transfer product}

Recall that analysis of the Gysin sequence, in particular Proposition~\ref{gysinsplit} and Theorem~\ref{AnnIdeal}, allow 
us to understand  the cohomology of $\alt_n$ as a $C_2$-representation under the conjugation.  
This presentation over $C_2$ is key to the interplay between transfer product and coproduct.

\begin{definition}
A polarized basis for a  $C_2$-representation is a basis $B = \{ B_+, B_-, B_o \}$ %(called  positive, negative and neutral) 
where the $C_2$-action interchanges $B_+$ and $B_-$ and fixes $B_o$.  
The positive projection, denoted $\rho^+(x)$ by abuse omitting $B$ from notation, is that onto the span of $B_+$.

For  $V$ with polarized basis, any
tensor power has an induced polarized basis where the neutral sub-basis is given by the tensor products of  
$B_o$ and by convention the positive sub-basis is given by products where the first non-neutral vector is positive.

Let $\mathcal{B} =  \{ \mathcal{B}_+, \mathcal{B}_-, \mathcal{B}_o\}$ be a choice of polarized basis  of $H^*(B \alt_n)$
arising from its decomposition through Proposition~\ref{gysinsplit} and Theorem~\ref{AnnIdeal}.
\end{definition}

Thus $\mathcal{B}_o$ corresponds to $\mathcal{G}_{quot} \backslash \mathcal{G}_{ann}$, and 
$\mathcal{B}_+$ and $\mathcal{B}_-$ correspond to $\mathcal{G}_{ann}$.  We will  refine our choices for these bases below, but name
them now as any choice will suffice for the following arguments.

\begin{theorem}\label{bialg}
The coproduct $\Delta(\alpha \odot \beta)$ is given by 
$$\Delta(\alpha \odot \beta) = (\mu_\odot \otimes \mu_\odot) \left( \tau \circ \rho^+\left(\Delta(\alpha) \otimes \Delta(\beta) \right) \right),$$
where $\rho^+$ is defined through the  polarized basis on $H^*(B\alt_\bullet)^{\otimes 4}$ induced by 
$\mathcal{B} =  \{ \mathcal{B}_+, \mathcal{B}_-, \mathcal{B}_o\}$,
and where $\tau$ is the standard transposition of second and third factors of the tensor product. \end{theorem}

Recalling that $\mu_\odot$ is the multiplication 
map for the transfer product, 
this differs from the usual statement that $(\odot, \Delta)$ form a bialgebra only by the polarization $\rho^+$.
We could express the equality in Theorem~\ref{bialg} succinctly as  $\Delta (\alpha \odot \beta) = \Delta \alpha \odot_{\rho^+} \Delta \beta$.
%though we will not need such notation in the rest of this paper.

Before proving this theorem, we prove a simpler related result which helps complete the picture of the relationship between
the cohomology of symmetric and alternating groups.  Recall  the Serre spectral sequence for the fibration $BH \to BG \to B(G/H)$ when $H$ is
a normal subgroup of $G$.

%ecall the Borel spectral sequence for the cohomology of the quotient of $X$ by a free action of $G$.  
%This is the Leray-Serre spectral sequence for the fibration $X \to X/G \to BG$, using the fact that $X/G \simeq X \times_G EG$,
%so  $E_2^{p,q} = H^p\left(BG; {\mathcal H}^q(X)\right).$

\begin{proposition}
The Serre spectral sequence for $\alt_n \to \si_n \to C_2$ collapses at $E_2$.  
\end{proposition}

We may also view this spectral sequence as the Borel spectral sequence for the homotopy orbits of the  $C_2$ action $B\alt_n$ by conjugation, 
with homotopy quotient $B\si_n$.  While we already know the cohomology
of all of the spaces in this fibre sequence,  this spectral sequence will be generalized  in proving Theorem~\ref{bialg}.

\begin{proof}
We have $E_2^{p,q} = H^p\left(BC_2; {\mathcal H}^q(B\alt_n)\right)$.
%Over $k = {\mathbb F}_2$ there are only two $C_2$-modules to consider, namely the trivial module and the regular representation. 
The cohomology of $BC_2$ with trivial coefficients is that of $\R P^\infty$, while as usual the cohomology of the regular representation is concentrated
in degree zero, of rank one.  

%As every conjugate pair $x^\pm  \in \mathcal{B}_\pm$ gives a copy of the regular representation, 
%they give rise to a single class which we call $x \in E^2_{0,*}$
%which then restricts to $x^+ + x^-$, consistent with previous definitions of these classes.  For every $y \in  \mathcal{B}_o$ we have a corresponding 
%$y \in E^2_{0,*}$ and more generally $y \cdot e^q \in E_2^{q,*}.$  

By Proposition~\ref{gysinsplit} and Theorem~\ref{AnnIdeal} there is a copy of the regular representation in the cohomology of $\alt_n$ for every
element of $\mathcal{G}_{ann}$ -- that is, 
scale greater than one Hopf monomial.  This gives rise to a single cohomology class in $E_2^{0,*}$.  To a Hopf ring monomial $h$ in the complement 
of $\mathcal{G}_{ann}$ in the Hopf monomial basis for $H^*(B\si_n)$,
we associate a class $h'$ in $E_2^{\delta(h),q}$.  Here, as in the proof of Theorem~\ref{AnnIdeal},  
$\delta(h)$ is the difference between the largest power of $\gamma_{1,1}$ 
which occurs as a constituent $\cdot$-monomial 
and the second  largest power of any $\gamma_{1.m}$ which so occurs, when the former is greater, or zero otherwise.  
When  $\delta(h)=0$, $h' = h$.  Otherwise, the class $h'$ is obtained by removing the largest power of
$\gamma_{1,1}$ from the monomial and then replacing the second largest ${\gamma_{1,m}}^k$ with ${\gamma_{1,m+1}}^k$, 
in order to obtain an element of $\mathcal{G}_{quot}$.  
This gives an additive isomorphism between the cohomology of $\si_n$ and the $E_2$-term of the spectral sequence.  
\end{proof}

\begin{proof}[Proof of Theorem~\ref{bialg}]
Recalling  the proof of Theorem~\ref{intersectionprop}, 
let $p+q = n$, $r + s = m$, $p + r = i$, and $q + s = j$, and set $d= i+j = n+m$. Let  $H_{p,q,r,s}$ be the intersection of $\alt_n \times \alt_m$ and 
the conjugate of $\alt_i \times \alt_j$ in $\alt_d$, which contains $\alt_p \times \alt_q \times \alt_r \times \alt_s$ as a subgroup of index two.

Consider the  diagram
$$\xymatrix{
\bigsqcup B\alt_p  \times B\alt_q \times B\alt_r \times B\alt_s 
\ar[d]_-{f = \bigsqcup f_{p,q,r,s} }  \ar[rd]^-{\qquad\bigsqcup  (B \; \iota_{p,r} \times \iota_{q,s}) \circ \tau} \\
\bigsqcup BH_{p,q,r,s} \ar[d]_-{h} \ar[r]^-{g = \bigsqcup g_{p,q,r,s}} &  B\alt_i \times B\alt_j \ar[d]^-{B \iota_{i,j}}\\ 
B\alt_{n} \times B\alt_m \ar[r]_-{B \iota_{n,m}} &B\alt_{i+j},
}$$
where all maps on classifying spaces are induced by  inclusions of subgroups, with some already named inclusions indicated.  
By Proposition~\ref{pullback}, 
the lower square is a pull back.   Thus the coproduct of a $\odot$-product, which is given by ${B \iota_{n,m}^*} \circ {B \iota_{i,j}}^! $,
is equal to $h^! \circ g^*$.

The failure of $(\odot, \Delta)$ to be a bialgebra is given by the existence of the $f_{p,q,r,s}$, so their analysis plays a key role.
As $f_{p,q,r,s}$ is a double cover, we use the Borel spectral sequence.  
By Proposition~\ref{gysinsplit} and Theorem~\ref{AnnIdeal} and the K\"unneth theorem,
we understand the cohomology of $B\alt_p \times B\alt_q \times B\alt_r \times B\alt_s$ as a module over
$k[C_2]$. %, which like that above has two kinds of classes at the $E_2$-page.  

There is a class in $E^{0, *}_2$ for every pair $x_1 \otimes x_2 \otimes x_3 \otimes x_4$ and 
$\overline{x_1} \otimes \overline{x_2} \otimes \overline{x_3} \otimes \overline{x_4}$ in $\mathcal{B}^{\otimes 4} \backslash {\mathcal{B}_o}^{\otimes 4}$.
(Here and elsewhere we use notation for the cohomology of different $\alt_n$ as if they were all the same vector space.)
The other classes in $E^{p,q}_2$ are given by $(y \cdot e^p)$ where $y \in {\mathcal{B}_o}^{\otimes 4}$ and $e \in E^{1,0}_2$ is a generator for the
cohomology of $BC_2$.  

We claim that this second set of classes are associated gradeds of classes which map to zero under $h^!$.
That is, choices for these classes with
all possible indeterminacies map to zero under $h^!$.  At the spectral sequence level
they are pulled back from the Borel spectral sequence for 
$$B (\alt_p \times \alt_q \times \alt_r \times \alt_s) \to B (\si_p \times \si_q \times \si_r \times \si_s) \to B (C_2 \times C_2 \times C_2 \times C_2),$$
where the map of base spaces is induced by the diagonal embedding $C_2 \to (C_2)^4$.
Thus these classes are associated gradeds of classes pulled back from $B (\si_p \times \si_q \times \si_r \times \si_s)$.  
Apply Lemma~\ref{L:pullpushzero} to the inclusions of  $H_{p,q,r,s}$ in both $\si_p \times \si_q \times \si_r \times \si_s$ and $\alt_n \times \alt_m$,
whose intersection in $\si_n \times \si_m$ contains $H_{p,q,r,s}$ with even index, to show that classes in  $H_{p,q,r,s}$  pulled back from 
$B (\si_p \times \si_q \times \si_r \times \si_s)$ map to zero under $h^!$.

To prove the theorem,  consider 
$\alpha = ( (\Delta_{p,r} \otimes \Delta_{q,s})\circ \tau) (x \otimes y) \in H^*(B\alt_p \times B\alt_q \times B\alt_r \times B\alt_s)$.   
As in the Gysin sequence for alternating and symmetric groups, its 
polarization will transfer under $f^!$ to a class $\beta$ with ${f_{p,q,r,s}}^*(\beta) = \alpha$ modulo ${\mathcal{B}_o}^{\otimes 4}$ .
Thus by the analysis above, its polarization transfers under $f^!$ to something whose difference from $g_{p,q,r,s}^* (x \otimes y)$ is in the kernel of $h^!$.
Therefore $h^! \circ g_{p,q,r,s}^* (x \otimes y)$ is equal to $ (h \circ f_{p,q,r,s})^!$ applied to the polarization of $\Delta_{p,r} \times \Delta_{q,s} (x \otimes y)$,
as claimed.
\end{proof}

\subsection{Strategy for constructing an almost Hopf semi-ring presentation}\label{mainstrategy}

%The Gysin sequence and involution developed in this section are key structures, as while the Gysin sequence determines additive 
%structure, the organization as a $C_2$-module 
%with respect to involution is in particular essential for expressing relations and explicitly controlling the failure of being a Hopf semi-ring.
We outline our  strategy for moving from our additive Gysin sequence description of the mod-two cohomology of symmetric 
groups to a multiplicative description.  We lay  groundwork  in the next two sections and prove the main results after that.

Recall Proposition~\ref{gysinsplit} which shows how identifying overlapping bases, $\mathcal{G}_{ann}$ and 
$\mathcal{G}_{quot}$, for the  the annihilator ideal of and quotient by the Euler class in a Gysin
sequence determines the sequence, as a representation over $C_2$.  
Recall Theorem~\ref{AnnIdeal} which identifies such bases  for the Gysin sequence relating cohomology of alternating and symmetric groups.
We understand how to express any such basis element in the cohomology of symmetric group as a Hopf ring monomial.  
So our plan is to construct lifts with respect to the transfer map or take images under restriction of relevant Hopf ring generators,
and track their (almost Hopf ring) products in the Gysin sequence appropriately, as follows.

%\begin{strategy}
\begin{enumerate}
\item  In Section~\ref{generators} we 
find $\gamma_{\ell,2^k}^+ \in H^*(B\alt_n)$ whose image under transfer is 
$\gamma_{\ell,2^k}$, using cochain representatives developed in Section~\ref{FN}.

\item We then argue using subgroup restriction, as developed in Section~\ref{restriction_section}, that 
cup monomials in $\gamma_{\ell,2^k}^+$ transfer to corresponding monomials in $\gamma_{\ell,2^k}$.   \label{steptwo}

\item 
For consistency with notation needed in Step~(\ref{stepfour}), let $\gamma_{1,2^k;2^k}$ be the restriction of $\gamma_{1,2^k}$.
Since any $x \in \mathcal{G}_{ann}$  is a  Hopf ring monomial in  
 $\gamma_{\ell, 2^k}$ for $\ell \geq 2$
 along with  $\gamma_{1,k}$ with $k \geq 2$, Step~(\ref{steptwo}) and elementary facts about restriction and transfer
 imply that there is a corresponding monomial in $\gamma_{\ell,2^k}^+$ and $\gamma_{1,2^k;2^k}$, which we call $x^+$,
 whose image under transfer is $x$.
Define $x^- = \overline{x^+}$, which also transfers to $x$.  
%Noting that $\mathcal{G}_{ann} \subset \mathcal{G}_{quot}$, 
%$x$ must restrict to $x^+ + x^-$.  
We will denote the set of $x^+$ above by $ \mathcal{B}_+$, a specific choice for basis whose existence has already been established,
and the set of $x^-$ by $\mathcal{B}_-$.
% by construction then span the $\rho$ summand of $H^*(B\alt_{2n})$, as given by Theorem~\ref{AnnIdeal} 
%and Proposition~\ref{gysinsplit}.
 \label{stepthree} 

\item  Next we use the $1$-decomposition of $y \in \mathcal{G}_{quot} \backslash \mathcal{G}_{ann}$, as
the transfer product of a polynomial 
in $\gamma_{1,k} \odot 1_{m-k}$ with $k \geq 2$ and an element of $\mathcal{G}_{ann}$ from a smaller alternating group.
By Proposition~\ref{restrictiontransfer} and Step~(\ref{stepthree}), its restriction $y^o$ will  be the transfer product of a 
cup product polynomial in the restrictions of  $\gamma_{1,k} \odot 1_{m-k}$,
which we call $\gamma_{1,k;m}$, and an element of $\mathcal{B}_+$ from a smaller alternating group. 
Denote the set of $y^o$ by $ \mathcal{B}_o$.  \label{stepfour}

\item  By Theorem~\ref{AnnIdeal} 
and Proposition~\ref{gysinsplit}, the union of $\mathcal{B}_+$, $\mathcal{B}_-$ and $\mathcal{B}_o$ 
form an  additive basis for $H^*(B\alt_n)$, respecting its Gysin decomposition.  So the $\gamma_{\ell,2^k}^+$
and $\gamma_{1,k;m}$, along with $1^-$ generate $H^*(B\alt_\bullet)$ as an almost Hopf semi-ring.

\item With additive basis built from specified classes through Hopf ring products, 
we make calculations of restrictions to elementary abelian subgroups in Section~\ref{detection_section},
showing that the direct sum of restrictions is injective.
Cup and transfer products exhibit  diffrent behavior on $ \mathcal{B}_+, \mathcal{B}_-$ and $\mathcal{B}_o$.
%which we refer to by their ``charge.''   
We will  see that transfer products of classes of  the same charge are naturally positive and of opposite charge are negative.
Cup products between  classes of the same charge behave mostly like corresponding cup products for 
symmetric groups, while cup products between classes of opposite charge will ``mostly'' be zero.

\item Finally in Section~\ref{presentation} we  establish relations, as well as coproduct calculations as needed to apply Hopf semi-ring distributivity,
giving the final presentation as an almost Hopf semi-ring.  The last two sections address Steenrod structure and the determination of cup product structure for individual alternating groups from the 
almost Hopf semi-ring presentation.

\end{enumerate}
%\end{strategy}

%A consequence of the last detection, by Quillen's theorem, is that there are no nilpotent elements in these cohomology rings.

%%% Local Variables:
%%% TeX-master: "AltGroupsModTwo.V3.tex"
%%% End: % Cohomology of symmetric groups, the Gysin sequence and the standard involution
% !TEX root = AltGroupsModTwo.V3.tex

\section{Fox-Neuwirth  models}\label{FN}

At key points, starting with the definition of our almost Hopf semi-ring generators, we require cochain-level calculations.
Rather than the cobar construction for group rings, we prefer cochain models based on the geometry of configuration
spaces, due to Fox and Neuwirth.   We first briefly recall these for  symmetric groups, as developed
in \cite{GiSi12}.

We choose the classifying space for the symmetric group $\si_{n}$
as the space of $n$ distinct points in $\R^{\infty}$, which we call $\UConf{n}{\infty}$, which
is the quotient of the labeled configuration space $\Conf{n}{\infty}$ by the symmetric group action permuting labels.
The finite-dimensional approximations $\UConf{n}{d}$ are manifolds with a beautiful cellular decomposition.

The points in any configuration in $\UConf{n}{d}$ are ordered by the dictionary order of their
coordinates.  If we consider the $i$th and $i+1$st points under this ordering, they share some $a_{i}$
of their first coordinates.  That is, $a_{i} = 0$ if their first coordinates are distinct, $a_{i} = 1$ if they share
their first coordinate but have distinct second coordinates, and so forth.  

\begin{definition}
Let $\Gamma = [a_{1}, \cdots, a_{n-1}]$ be a sequence of non-negative integers, and let $|\Gamma| 
= \sum a_{i}$.
Define $\UConf{\Gamma}{d}$ to be the collection of all configurations such that the $i$th and $i+1$st points
in the dictionary order in the configuration share their first $a_{i}$ coordinates but not their $(a_{i} + 1)$st.  We say
such points respect $\Gamma$.
\end{definition}

\begin{theorem}[after Fox-Neuwirth]\label{T:cells}
For any $\Gamma$ the subspace $\UConf{\Gamma}{d}$ is homeomorphic to a Euclidean ball of dimension $nd - |\Gamma|$. 
The images of the $\UConf{\Gamma}{d}$ are the interiors of cells in a 
CW structure on the one-point compactification  $\UConf{n}{d}^{+}$.
\end{theorem}

\begin{definition}
Let $(FN^{d}_{n})_{*}$ be the cellular chain complex associated to the cell structure defined by the 
$\UConf{\Gamma}{d}$, which by the above computes the homology of  $\UConf{n}{d}^{+}$.
\end{definition}

  These chain complexes model the cochains of the classifying spaces of symmetric groups as follows.  
  Assume that $d$ is even, in which case $\UConf{n}{d}$ is an orientable manifold of dimension $nd$.
 Alexander Duality
implies that the homology of its one-point compactification in degree $nd - i$ is isomorphic to its cohomology 
in degree $i$.  If $i < d$, the group $(FN^{d}_{n})_{nd - i}$ and differential is independent of $d$, so we may set the following.

\begin{definition}
Let ${FN_{n}}^{*}$ be the cochain complex which in degree $*= i$ is $(FN^{d}_{n})_{nd - i}$ for $d > i$ and with differential defined 
through the cellular chain structure on $(FN^{d}_{n})_{*}$.
\end{definition}

%Our discussion so far yields the following.

We show in \cite{GSS12} that he cohomology of ${FN_{n}}^{*}$ is that of $B \si_{n}$.
While the chain groups $(FN_{n})^{i}$ are simple, spanned by sequences of non-negative integers which add up to $i$, the boundary maps are 
complicated.   A main result of \cite{GiSi12} is explicit calculation of the differential.  

We now develop the alternating group analogues .
For alternating groups, our cochains will be a double cover of the cochains for
symmetric groups, and we use orientation or ``charge'' to express this.

\begin{definition}
A  {charged} sequence of non-negative integers is a finite sequence of non-negative integers along with a choice of sign. 
Given a sequence of non-negative integers $\Gamma$ we write $\Gamma^+$ for the  {positively charged} sequence 
associated to $\Gamma$ and $\Gamma^{-}$ for the  {negatively charged} sequence.  For convenience, we write  $\Gamma^o = \Gamma^+ + \Gamma^-$
in  chain groups.

Let ${\fna_{n}}^{i}$ be spanned by charged sequences of $(n-1)$ non-negative integers which sum to $i$.

\end{definition}

Identifying the differential requires a few combinatorial definitions. 

\begin{definition} 
%\begin{itemize}
 The $\ell$-blocks of a sequence $\Gamma = [a_1, \dots, a_{n-1}]$ are the ordered collection of possibly empty
subsequences $[a_i, a_{i+1}, \dots, a_{i+k}] \subset \Gamma$ such that  $a_{i-1}$ and $a_{i+k+1}$
are consecutive in the subset of entries which are less than or equal to $\ell$.
By convention,  set $a_0 = a_n = -\infty$, so  in particular an empty sequence always has a single empty $\ell$-block for any $\ell$.   

Denote by $\Gamma\langle {i}\rangle$ the sequence obtained from $\Gamma$ by adding one to its $i$th entry.
%\end{itemize}
\end{definition}

Sometimes we call a zero-block simply a block.
For example, the blocks of $\Gamma_{\rm ex} = [3,0,1,2,0,0,4,4]$ are 
$([3], [1, 2], \emptyset, [4, 4])$, while $\Gamma_{\rm ex}\langle 2\rangle$ has one-blocks $([3], \emptyset, [2], \emptyset, [4,4])$.

 In the cell $\UConf{\Gamma}{d}$, $\ell$-blocks of length $k$ correspond to collections of $k+1$ adjacent points in a configuration which 
 share more than their first $\ell$ coordinates.

\begin{definition}
  Let $\Gamma = [a_1, \dots, a_{n-1}]$ and fix $1 \leq i \leq n-1$. The sequence $\Gamma\langle i\rangle$ has a non-empty $a_i$-block of the form $\Lambda_i = [a_{i-r}, \dots, a_i + 1, \dots, a_{i+s}]$ for some $r, s \geq 0$.

  Partition the $(a_i+1)$-blocks of $\Lambda_i$ into the $p > 0$ of them appearing in the possibly empty subseqence $[a_{i-r}, \dots, a_{i-1}]$ 
  and the $q > 0$ of them in $[a_{i+1}, \dots, a_{i+s}]$. 
  Write $\Sh(\Gamma, i)$ for the collection of $(p,q)$-shuffles with action on $\Gamma\langle i \rangle$ by permuting the $a_i+1$ blocks of $\Lambda_i$ and leaving the remainder of the sequence fixed.

  Denote by $\Sh_+(\Gamma,i)$ and $\Sh_-(\Gamma,i)$ the subsets of even and odd shuffles in $\Sh(\Gamma, i)$, respectively.
\end{definition}

Continuing with $\Gamma_{\rm ex}$,  we find $\Lambda_2 = [3, 1, 1, 2]$ has 1-blocks 
$([3], \emptyset, [2])$, partitioned into $([3])$ and $(\emptyset, [2])$. 
The set $\Sh(\Gamma_{\rm ex}, 2)$ consists of the three $(1,2)$-shuffles, %$\{e, (1 2), (1 3 2)\}$, 
which act on $\Gamma_{\rm ex}\langle 2 \rangle$ by sending it to $[3, 1, 1, 2, 0, 0, 4, 4]$, $[1, 3, 1, 2, 0, 0, 4, 4]$, and $[1, 2, 1, 3, 0, 0, 4, 4]$.
%\begin{eqnarray*}
%  e \cdot \Gamma \langle 2 \rangle & = & [3, 1, 1, 2, 0, 0, 4, 4]\\
% (1 2) \cdot \Gamma \langle 2 \rangle & = & [1, 3, 1, 2, 0, 0, 4, 4]\\
% (1 3 2) \cdot \Gamma \langle 2 \rangle & = & [1, 2, 1, 3, 0, 0, 4, 4].\\
 % \end{eqnarray*}
 \begin{definition}\label{D:fnadiff}

Let $\Gamma$ be a sequence of non-negative integers. The mod-two differential in $\fna_n^*$ is given by 
\begin{equation*}
\delta(\Gamma^\pm) = \sum_{i=1}^{n-1} \delta_{i} (\Gamma^\pm) \;\;
\text{where} \;\;  \delta_{i} (\Gamma^\pm) =  
        \sum_{\sigma \in \Sh_+(\Gamma,i)} \sigma \cdot \Gamma\langle{i}\rangle^\pm + \sum_{\sigma \in \Sh_-(\Gamma,i)} \sigma \cdot \Gamma\langle{i}\rangle^\mp
\end{equation*}
\end{definition}

%For clarity, when we are referring to sequences corresponding to basis elements in $\fna_{n}^*$, we will use square brackets rather than parenthesis. 

This gives our cochain model for classifying spaces of alternating groups.

\begin{theorem}\label{T:fnagood}
The cohomology of ${\fna_{n}}^{*}$ is that of $B \alt_{n}$.
\end{theorem}

We  start proving this theorem by developing preferred classifying spaces.

\begin{definition}
Let $\Ind{n}{d}$ be the subset of $\Conf{n}{d}$ of configurations which are linearly independent.  
The span of such configurations has a canonical orientation, coming from the configuration as an ordered basis.

Let $\UInd{n}{d}$ be the quotient of $\Ind{n}{d}$ by the alternating group action on the labels, which thus
preserves the canonical orientation. 
\end{definition}

Thus $\UInd{n}{d}$ is the space of unlabeled configurations with an orientation on their span.   
Because $\Ind{n}{d}$ is $(d-n-1)$-connected,  $\UInd{n}{\infty}$ models $B\alt_{n}$.

\begin{definition}
Let $\Gamma = [a_{1}, \cdots, a_{n-1}]$ be a sequence of non-negative integers.
Define $\UInd{\Gamma+}{d}$ (respectively $\UInd{\Gamma-}{d}$)
to be the collection of all configurations with the following two properties
\begin{itemize}
\item the $i$th and $i+1$st points
in the dictionary order in the configuration share their first $a_{i}$ coordinates but not their $a_{i} + 1$st;
\item the orientation of the span 
defined by the dictionary order agrees (respectively, does not agree) with the
orientation of the span as a point in $\UInd{n}{d}$.
\end{itemize}
\end{definition}

\begin{proof}[Proof of \refT{fnagood}]
The subspace $\UInd{\Gamma\pm}{d}$ is not the interior of a cell, but is the complement within a cell %of dimension $nd -|\Gamma|$ 
of the subvariety of non-linearly independent configurations, which is of codimension
$d-n+1$.  Taking closures in the one-point compactification we obtain pseudo-cells.  The long exact sequences 
of pairs of these codimension $i$ pseudo-cells will behave as if they were cellular in and around degree 
$i$ as long as $d > i + n$.  More precisely, there will be maps of pairs between the skeleta defined by 
pseudo-cells,  with linearly dependent configurations removed,
and the skeleta defined by cells with all configurations.  The maps between these pairs will be isomorphisms in all of the degrees defining the cellular
chain complex.  Moreover, the set of pseudo-cells and their boundary behavior will be 
independent of $d$ as long as $d > i$.   

So as long as $d > i + n$, the $i$th cohomology of $\UInd{n}{d}$ agrees with that of $B \alt_{n}$ and is computed by the 
incidence of the $\UInd{\Gamma\pm}{d}$ as if they were cellular. 
The spectral sequences associated to filtration by $|\Gamma|$ form directed system
which stabilizes and thus converges in the limit.  The stable terms yield the chain groups
${\fna_{n}}^{i}$ as the $E_{0}^{i,0}$.

For the boundary maps, observe that each pseudo-cell is bounded by a family of such for which some pair of points 
consecutive in the dictionary ordering agree in one more dimension. Suppose that the $i$th and $(i+1)$st points in the pseudo-cell share their 
first $a_i$ coordinates. We know that the $i$th point has a smaller $(a_i+1)$st coordinate than the $(i+1)$st point, but we do not know 
the relationship between their $(a_i+2)$nd coordinates. Further, it is possible for either of these points to share $(a_1+1)$ coordinates 
with other points, in which case the ordering on this whole collection of points in the boundary is only partially determined by the data in the pseudo-cell. 
See, for example, \refF{boundary}. Thus, every shuffle of these two ordered families of points appears in the boundary of the pseudo-cell, 
with the orientation of the bounding pseudo-cell changing under odd shuffles.  
\end{proof}

\begin{center}
\begin{figure}
\psfrag{1}{\small{1}}
\psfrag{2}{\small{2}}
\psfrag{3}{\small{3}}
\psfrag{4}{\small{4}}
\psfrag{5}{\small{5}}
\psfrag{6}{\small{6}}
\begin{tabular}{c}

\xymatrix{
\includegraphics[width=3cm]{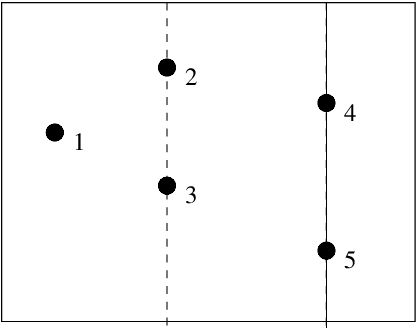}\ar[d]\ar@/^1pc/[dr]\ar@/^2pc/[drr]&\\
\includegraphics[width=2cm]{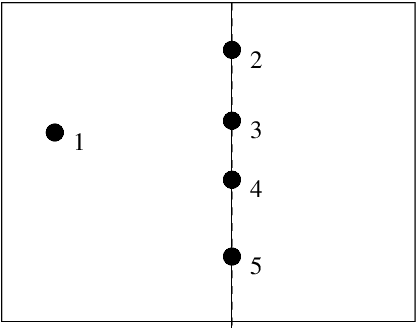}&\includegraphics[width=2cm]{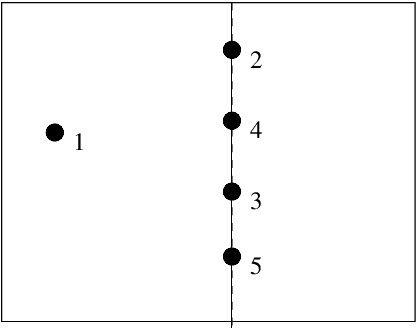}&\includegraphics[width=2cm]{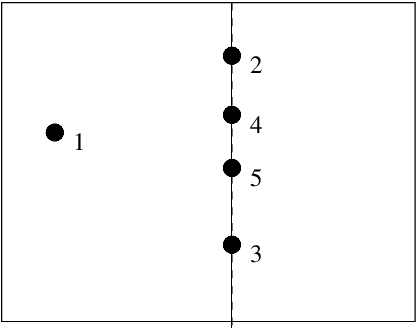} \raisebox{.75cm}{$\qquad\cdots$}\\
}

\end{tabular}
\caption{The boundary of a pseudocell in $\UInd{\Gamma\pm}{d}$ involves shuffles of ordered families of points.}

\label{F:boundary}
\end{figure}
\end{center}

For a sequence of non-negative integers $x$ we use the notation $x^o$ for the sum  $x^+ + x^-$ in this cellular chain complex.
The two-sheeted covering of $B\alt_n$ over $B\si_n$
has the effect of splitting the cells which correspond to cycles in $\fn_n^*$. For example, in ${\fn_4}^3$, we have
\begin{eqnarray*}
\delta([1,0,2])& =& 2[2,0,2] + [1,1,2] + [1,2,1] + [2,1,1]\\
\delta([2,0,1])& =& 2[2,0,2] + [1,1,2] + [1,2,1] + [2,1,1]
\end{eqnarray*}
so their sum is a mod-two cycle which by the symmetric group analogue of Theorem~\ref{T:cellular_transfer_product} below
 represents ${\gamma_{1,1}}^2 \odot \gamma_{1,1}$. 
While in ${\fna_4}^3$,
\begin{eqnarray*}
\delta([2,0,1]^\pm)& =& [3,0,1]^\pm + [3,0,1]^\mp + [2,1,1]^\pm + [1,2,1]^\pm + [1,1,2]^\pm + [2,0,2]^\pm + [2,0,2]^\mp\\
&=& [3,0,1]^o + [2,1,1]^\pm + [1,2,1]^\pm + [1,1,2]^\pm + [2,0,2]^o
\end{eqnarray*}
Thus $[2,0,1]^o + [1,0,2]^o$ is a cycle, pulled back from the cycle $[2,0,1] + [1,0,2] \in {\fn_4}^3$.
 However it must be trivial in cohomology as it represents the restriction of ${\gamma_{1,1}}^2 \odot \gamma_{1,1}$, which is
 divisible by the Euler class.  Explicitly we see
\begin{eqnarray*}
\delta([1,0,1]^\pm)& =& [2,0,1]^o + 4[1,1,1]^\pm + 2[1,1,1]^\mp + [1,0,2]^o.
\end{eqnarray*}

The distribution of even and odd permutations in a cell's boundary is generally not symmetric. 
Thus we  require  computations of the number of even and odd shuffles of $(1,\ldots,p, p+1, \ldots q)$, which we denote $|\Sh_\pm(p, q)|$. 
The results are elementary and will be stated without proofs, 
which can use either bijective arguments 
or the basic fact that $|\Sh_\pm(p,q)| = |\Sh_\pm(p-1,q)| + |\Sh_{\pm(-1)^p}(p, q-1)|$. 
The following are some of the key computations using such results.

\begin{lemma}\label{FNlemma}
  Suppose $\Gamma^\pm = [a_1, \dots, a_{n-1}]^\pm$ and $[a_{i-r}, \dots, a_{i-1}]$ and $[a_{i+1}, \dots, a_{i+s}]$ are $a_i$-blocks of $\Gamma$.

  If $a_j = a_i+1$ for every $j \in \{i-r, \dots, i-1, i+1, \dots i+s\}$, then $$\delta_i(\Gamma^\pm) = |\Sh_+(\Gamma,i)| \Gamma\langle i \rangle^\pm + |\Sh_-(\Gamma,i)| \Gamma\langle i \rangle^\mp.$$ 
  In particular, 
  $$ \delta_{i}(\Gamma^\pm) = 
  \begin{cases}
  \Gamma^o\langle i\rangle, & {\rm if} \; r = s = 0\\
  0, &  {\rm if} \; r = s > 0\\
  \Gamma\langle i \rangle^\pm, &  {\rm if} \; \ell > 1, r = 2^\ell-1, \; {\rm and} \; s < r.
  \end{cases}
  $$
 \end{lemma}

  \begin{lemma}\label{FNlemma2}
Let $\Gamma^\pm$ be as in Lemma~\ref{FNlemma}.  

If $a_j > a_i+1$ for all $j \in \{i-r, \dots, i-1\}$, and $a_j = a_i+1$ for $j \in \{i+1, \dots, i+s\}$, with $r > 1$ and $s > 0$, then
$$\delta_i(\Gamma^\pm) = \sum_{j=0}^{r+1} [a_1, \dots, a_{i-r-1}, \hat{a}_{(i-r, j)}, \dots, \hat{a}_{(i+1, j)} ,a_{i+s+1}, \dots, a_{n-1}]^{\pm(-1)^j},$$
where $\hat{a}_{(k, j)} = [a_{i-r}, \dots, a_{i-1}]$ if $k = i-r+j$ and $\hat{a}_{(k,j)} = a_i+1$ otherwise.
\end{lemma}

  For example, 
\begin{align*}
  \delta_5([2,0,2, 3, 0, 1, 1,0,1]^o) =&\; [2,0,2,3,1,1,1,0,1]^o + [2,0,1,2,3,1,1,0,1]^o \\
  &+ [2,0,1,1,2,3,1,0,1]^o + [2,0,1,1,1,2,3,0,1]^o.
  \end{align*}

The proofs are straightforward calculation.  For Lemma~\ref{FNlemma}, $\Lambda_i$ has $r+s+2$ empty $(a_i+1)$-blocks, 
and $\Sh(\Gamma, i) = \Sh(r+1,s+1)$, 
  each of which fixes the sequence $\Gamma\langle i\rangle$. 
For Lemma~\ref{FNlemma2}, $\Sh(\Gamma, i) = \Sh(r+1, 1)$, with action on $\Lambda_i$ resulting in all possible placements of the sequence 
  $[a_{i+1}, \dots, a_{i+s}]$ in a sequence of $r+1$ $(a_i+1)$s with alternating charges. 

%By choosing models which exhibit the double-cover, we 

Turning to more general results, we have the following. %, which gives a cochain-level model for  the isomorphisms of Theorems~\ref{T:cells} and \ref{T:fnagood}, 

\begin{theorem}\label{T:gysinmodel}
The restriction map sends 
$\Gamma$ to $\Gamma^+ + \Gamma^-$ and the transfer
map sends $\Gamma^\pm$ to $\Gamma$.
\end{theorem}

We will see in the next section that at the level of cohomology  there does not seem description of the Gysin sequence 
through adding or dropping 
labels alone, especially in the cohomology of $\alt_4$.
Most Hopf ring generators do behave as the cochains do in Theorem~\ref{T:gysinmodel}, but the irregular $\alt_4$ behavior propagates throughout.  

\begin{proof}
We may use $\Ind{n}{\infty}$, which is a subspace of our usual $\Conf{n}{\infty}$,  as a model for $E \si_{n}$.  The analysis of finite-dimensional
approximations proceeds as in the proof of \refT{fnagood}, with the filtration by $\Ind{\Gamma}{d}/\si_{n}$ yielding a chain complex as the limit of associated spectral
sequences.  The limiting
cochain complex is exactly ${FN_{n}}^{*}$.  With this model for the classifying spaces of symmetric groups, the restriction and transfer maps are ``cellular'' - that is, they are induced by filtration preserving maps which thus produce maps on limiting chain complexes - and are as stated.
\end{proof}

Finally, we  evaluate Fox-Neuwirth cochains directly by realizing duality through elementary chain-level intersection theory, as for example
developed in \cite{FMS17}.

Briefly,  let $X$ be a manifold and $W$ a codimension-$d$ manifold with an immersion $i$ to $X$.
We say that a smooth chain is transverse to an immersion when it is transverse in the usual sense when restricted to every pair 
of a face (including the interior of the simplex) and a codimension one 
subface as a manifold with boundary. 
Define the function $\tau_W$ 
on a smooth chain $\sigma: \Delta^d \to X$ transverse to $i$ as the cardinality mod-two 
of the pull-back of $i$ and $\sigma$.  When $i$ and $\sigma$ are embeddings, this counts intersection mod-two.
 
 Chains transverse to a fixed, finite set of immersions 
forms a subcomplex quasi-isomorphic to the singular chain complex.  
We view $\tau_W$ as an element of the dual cochains, which while being  ``partially defined'' can be used for many applications
in cohomology.

%This sub chain complex is quasi-isomorphic to the full chain complex, so the transverse cochain complex computes cohomology.   
Standard constructions in homology and cohomology theory are geometric in this model.
For example, if $f : Y \to X$ is transverse to $i$, then the natural map on cochains is given by $f^\# \tau_W = \tau_{f^{-1} W}$.  
The K\"unneth map is  given by geometric product.  Coboundary is given by  a
Stokes formula  $\delta \tau_W = \tau_{\partial W}$, which essentially follows classification of one-manifolds applied 
to $\sigma^{-1}(W)$ where $\sigma : \Delta^{d+1} \to X$.  Poincar\'e duality is almost a tautology, as we 
can allow for $W$ to have corners, in which for example case a triangulation of $X$ gives rise to 
cochains as well as chains. % and the Stokes formula shows that the cochain and
%chain complexes agree up to the standard change of degree.  

\begin{proposition}\label{intersection}
The isomorphism of \refT{fnagood} is realized by sending the chain $\Gamma^\pm$ to the (partially defined) cochain $\tau_{\UInd{\Gamma+}{d}}$.  
\end{proposition}

\begin{proof}
 In our setting, the inclusions of the $\UInd{\Gamma+}{d}$ extend to proper immersions.  The Stokes formula then shows that 
these cochains  form a subcomplex of the transverse cochain complex.    By Theorem~\ref{T:fnagood}, this subcomplex 
computes the appropriate cohomology.  
\end{proof}

%Note that a chain mapping transversally to the
%extension of  $\UInd{\Gamma+}{d}$ to a closed cell would only intersect the interior $\UInd{\Gamma+}{d}$ itself.

\begin{theorem}\label{T:cellular_coprod}
The coproduct $\Delta$ is cellular in $FNA^\ast_\ast$, given by sending 
$$[a_1, \cdots, a_n]^\pm  \mapsto 
\sum [a_1, \cdots, a_{i-1}]^+  \otimes [a_{i+1}, \cdots, a_{n}]\pm + [a_1, \cdots, a_{i-1}]^-  \otimes [a_{i+1}, \cdots, a_{n}]\mp,$$
where the sum is over all $i$ such that $a_i = 0$.
\end{theorem}

\begin{proof}
We use the fact that for $f : X \to Y$ transverse to $W$ we have
$f^{\#} \tau_W = \tau_{f^{-1} W}$.
Consider the cochain $\tau_W$ where $W$ is the immersion of the cell labelled by  $\Gamma = {[a_1, \cdots, a_n]}$, and let 
$Y = \UInd{n}{d} \times \UInd{m}{d}$, $X = \UInd{n+m}{d}$,  and 
$f$ be defined by ``stacking'' configurations.  More explicitly, define $f$ using homeomorphisms
of $\R$ with $(0,1)$ and $(2,3)$ to produce a configuration of $n+m$ points from two given configurations 
by taking a union of $n$ points whose first coordinate is in $(0,1)$ and $m$ points whose first
coordinate is in $(2,3)$.  To  guarantee linear independence, we may fix  modifications of coordinates beyond $|\Gamma|$.

Since $f$ is an inclusion of a codimension zero manifold, it is transverse to any submanifold.  The image of $f$
only contains points whose $n$th and $n+1$st points differ in first coordinate, so the pullback of cochains associated to $\Gamma$  
with $a_{n} \neq 0$ will be zero.  
For $\Gamma$ with $a_n = 0$ 
the preimage of  $\UInd{\Gamma^+}{d}$ in $Y$ will be the union of $\UInd{[a_1, \cdots, a_{n-1}]^\pm}{d} \times \UInd{[a_{n+1}, \cdots, a_{n+m-1}]^\pm}{d}$.
As the K\"unneth map is given by product of submanifolds, we obtain the result.  The statement for $\UInd{\Gamma^-}{d}$ is similar.
\end{proof}

\begin{definition}
Let $ \UInd{n,m}{d}$ be the space of configurations which are bi-colored  with $n$ points of the first color and $m$ points of the second, each
collection oriented.
\end{definition}

\begin{theorem}\label{T:cellular_transfer_product}
The transfer product is modeled at the Fox-Neuwirth cochain level by sending $\Gamma^\pm  \otimes \Lambda^\pm$ to the sum over sequences whose
$0$-blocks are shuffles of the $0$-blocks of $\Gamma$ and $\Lambda$, with charge which is the product of their charges.
\end{theorem}

\begin{proof}
We model the transfer product using the maps 
$$  \UInd{n}{d} \times \UInd{m}{d}  \overset{p_1 \times p_2}{\longleftarrow} \UInd{n,m}{d}  \overset{\phi}{\longrightarrow} \UInd{n+m}{d},$$
where $p_1$ (respectively $p_2$)  projects onto the first (respectively second) colored subset,
and where $\phi$ forgets colors altogether and takes the direct sum orientation.  As $d$ goes to $\infty$, the 
product $p_1 \times p_2$ models the equivalence $B(\alt_n \times \alt_m) \simeq B \alt_n \times B \alt_m$,
and  $\phi$ is a covering map model for $B(\alt_n \times \alt_m) \to B\alt_{n+m}$.

Take a chain  $\sigma$ on $\UInd{n+m}{d}$ which is transverse to all Fox-Neuwirth cells.  
The transfer of $\Gamma^\pm  \otimes \Lambda^\pm$
is defined by evaluating $\Gamma^\pm$ and $\Lambda^\pm$ on  $p_1 (\tilde{\sigma})$ and $p_2 (\tilde{\sigma})$ as $\tilde{\sigma}$ ranges
over lifts of $\sigma$,  which correspond to bi-colorings of underlying configurations along with compatible pairs of orientations.
Under our transversality assumption on $\sigma$,
the configurations in $p_1 (\tilde{\sigma})$ and $p_2 (\tilde{\sigma})$ will respect $\Gamma$ and $\Lambda$ if and only if those in $\sigma$
satisfy some sequence which is a shuffle of the zero-blocks $\Gamma$ and $\Lambda$.  Moreover,  the orientations can then be chosen compatibly 
if and only if the orientation is the product of orientations, which establishes the result.
\end{proof}

We obtain the following refinement of Proposition~\ref{transferneutral}, as the transfer product of
cochains which are invariant under conjugation will cancel in pairs.

\begin{corollary}\label{C:transfercancel}
The transfer product of two Fox-Neuwirth cochains which each are invariant under conjugation is zero.
\end{corollary}

\medskip 

While Theorems~\ref{T:cellular_coprod} and \ref{T:cellular_transfer_product} establish that Fox-Neuwirth cochains can be used to calculate 
coproducts and transfer products, 
they cannot be used for cup products, as  claimed without proof in  \cite{GiSi12}.
In Appendix~\ref{appendixa} we bring in knowledge of cup coproduct on the Cohen-Lada-May approach to homology
for two cases of finding cochain representatives for cup products.

%%% Local Variables:
%%% TeX-master: "AltGroupsModTwo.V3.tex"
%%% End: % Fox-Neuwirth cell models
% !TEX root = AltGroupsModTwo.V3.tex

\section{Restriction to elementary abelian subgroups}\label{restriction_section}

Restriction to elementary abelian subgroups  is a primary tool in group cohomology, giving maps to classical rings of invariants.
See for example  \cite{AdMi94}, where starting in the third chapter such techniques are developed and used extensively. 

For symmetric groups, the starting point is rings of invariants of polynomial algebras by general linear groups over finite fields, which 
were first studied by Dickson (see for example Chapter~6 of \cite{AdMi94}).  One must take these Dickson algebras, which remarkably are polynomial, 
and look at  symmetric invariants of tensor powers.  Such symmetric polynomials in multiple variables are complicated.  
Previous investigators of cohomology of symmetric groups proceeded by computing these invariants, with Feshbach \cite{Fesh02} being the 
most successful.

While our Hopf semi-ring approach to symmetric groups in \cite{GSS12} did not proceed through such invariant theory, for reference
we did connect with that approach.   We make use of that connection here.

\begin{definition}\label{defvn}
For $n>1$ let $V_{n}^+ \cong (C_2)^{n}$ denote the subgroup of $\alt_{2^{n}}$ defined by having $(C_2)^{n}$ act on itself.  
Let $V_n^-$ be the conjugate of $V_n^+$
by any element of $\si_{2^n}$ not in $\alt_{2^n}$.  %Let $\iota_n^\pm$ denote the corresponding inclusion homomorphisms.
\end{definition}

The low-dimensional cases are exceptional.  As $\alt_2$ is trivial,  $V_1^\pm$ must be as well.  
The case of $\alt_4$  
is exceptional in that there is only one elementary abelian 2-subgroup, so $V_2^+ = V_2^-$, which we just call $V_2$.  
The cohomology of $\alt_4$, whose calculation is worked out in detail in the first section of Chapter III of \cite{AdMi94},
significantly complicates the cohomology of all alternating groups.

The  Weyl group for $V_2$ in $\alt_4$ is not $GL_2(\F_2)$, which has order $6$, but is instead
a cyclic group of order $3$.  Indeed $V_2$ is normal in $\alt_4$, and the quotient has order three.  
The invariant theory approach to cohomology is quite effective.  The following is Theorem~III.1.3 of \cite{AdMi94}.

\begin{theorem}\label{A4Invariant} 
$H^*(B\alt_4) \cong H^*(V_2)^{C_3} \cong \F_2[x_1, x_2]^{C_3},$
where $C_3$ acts by cyclicly permuting $x_1, x_2$ and $x_1 + x_2$.  
This ring of invariants has a generator in degree two, namely $a = {x_1}^2 + x_1 x_2 + {x_2}^2$, 
and two in degree three, namely $b_+ = {x_1}^3 + {x_1}^2 x_2 + {x_2}^3$ and $b_- = {x_1}^3 + {x_1} {x_2}^2 + {x_2}^3$.
 There is a relation, namely ${b_+}^2 + b_+ b_- + {b_-}^2 + a^3 = 0$.
 \end{theorem}
 
 This relation in the ring of invariants stands in contrast to the Dickson invariant setting, and will
propagate throughout the cohomology of alternating groups.

We realize the generators of this cohomology as Fox-Neuwirth cochains.  

\begin{definition}
\begin{itemize}
\item Let  $\gamma_{1,2;2}$  be the class represented by $[1, 0, 1]^o$.
\item Let $\gamma_{2,1}^+$ to be represented by $[1,1,1]^+ + [2, 0, 1]^o$.
\item Let $\gamma_{2,1}^-$ to be represented by $[1,1,1]^- + [2,0,1]^o$.  
\end{itemize}
\end{definition}

By Theorem~\ref{T:gysinmodel} and Theorem~4.9 of \cite{GSS12}, $\gamma_{1,2;2}$ is the restriction of $\gamma_{1,2} \in H^2(B\si_4)$.
Thus by Theorem~7.8 of \cite{GSS12}, it restricts to $a$ in the cohomology $V_2$.

That $\gamma_{2,1}^\pm$ are cocycles is straightforward.  Again using Theorem~\ref{T:gysinmodel} and Theorem~4.9 of \cite{GSS12} we 
see that $\gamma_{2,1}^\pm$ both transfer to $\gamma_{2,1}$, which restricts to their sum.  Thus they span $H^3(B\alt_4)$.    Because they
are non-trivially conjugated, they must restrict to $b_+$ and $b_-$.  Summarizing we have the following.

\begin{proposition}\label{first_gamma2_relation}
The cohomology of $\alt_4$ is generated by $\gamma_{1,2;2}$, $\gamma_{2,1}^+$ and $\gamma_{2,1}^-$ with the relation
$$\gamma_{2,1}^+ \cdot \gamma_{2,1}^- = (\gamma_{2,1}^+)^2 + (\gamma_{2,1}^-)^2 + (\gamma_{1,2;2})^3.$$
\end{proposition}

In contrast to this relation, we will see that Hopf semi-ring generators defined on $\alt_{2^n}$ for $n > 2$ have mixed-charge 
products which are zero.  %The Gysin sequence restricted to Hopf semi-ring monomials in generators $\gamma_{\ell,m}$ for $\ell>2$ is simple.  
Accordingly, the Gysin sequence
for $\alt_4$, and thus its contributions to the Gysin sequence for all alternating groups,  is relatively complicated.  For concreteness
we record the following.

\begin{proposition}\label{A4Gysin}
A class which maps to ${\gamma_{1,2}}^p {\gamma_{2,1}}^n   \in \mathcal{B}_a $  under transfer is 
${\gamma_{1,2;2}}^p \left( \sum_{\substack{i < \frac{n}{2}}} \binom{n}{i} {\gamma_{2,1}^+}^i {\gamma_{2,1}^-}^{n-i} \right).$
\end{proposition}

\begin{proof}
As mentioned above, while $V_2$ is a subgroup of both $\alt_4$ and $\si_4$ its Weyl group is different in each case, namely $C_3$ in $\alt_4$ 
and $GL_2(\F_2)$ for $\si_4$.  The $GL_2(\F_2)$-invariants 
sit inside the $C_3$ invariants.  This inclusion of invariants represents the restriction homomorphism, as the Euler class
$\gamma_{1,1} \odot 1_2$ also generates kernel of the restriction from $\si_4$ to $V_2$.  The transfer map takes a $C_3$ invariant polynomial to 
its symmetrization - that is, the sum of a that polynomial and its conjugate defined by interchanging $x_1$ and $x_2$.  

Given this model for the transfer map, the binomial theorem  implies that 
$\sum_{i < \frac{n}{2}} \binom{n}{i} {\gamma_{2,1}^+}^i {\gamma_{2,1}^-}^{n-i}$ transfers to ${\gamma_{2,1}}^n$. 
The fact that symmetrization of a product of polynomials $f \cdot g$ when $g$ is already symmetric  is the product of the symmetrization of $f$ with $g$
gives the result in general.
\end{proof}

\begin{remark}\label{NoGoodBasis}
The maps in the Gysin sequence on these generators are given by adding or removing labels, but unlike for Hopf semi-ring generators of greater levels
this is not the case for multiples of generators.
Thus it seems that there is no basis for $H^d(B\alt_4)$ which is readily compatible with the Gysin sequence for $d \geq 9$.   For example, when $d=9$
the restriction of ${\gamma_{2,1}}^3$ is ${\gamma_{2,1}^+}^3 + {\gamma_{2,1}^+}^2 {\gamma_{2,1}^-} + 
{\gamma_{2,1}^+} {\gamma_{2,1}^-}^2 + {\gamma_{2,1}^-}^3$ while that of ${\gamma_{1,2}}^3 {\gamma_{2,1}}$ 
is ${\gamma_{1,2;2}}^3 \gamma_{2,1}^+ + {\gamma_{1,2;2}}^3 \gamma_{2,1}^-$.  But $H^9$ is of rank four, so we must use relations among the six
terms in these restrictions.  Accounting for transfer maps as in Proposition~\ref{A4Gysin} points to no consistently good choices.
\end{remark}

%More generally, we have the following building blocks for Fox-Neuwirth cochains that generate the cohomology of alternating groups.

When $n>2$, the behavior of restriction to $V_n^+$ becomes regular, and more  parallel to the symmetric group setting.
The invariants to which the cohomology of alternating groups restrict are again
 the Dickson algebras $\F_{2}[x_{1}, \ldots, x_{n}]^{GL_{n}(\F_{2})}$, which are
polynomial on generators $d_{k,\ell}$ in dimensions $2^{k}(2^{\ell} - 1)$ where $k + \ell = n$.
%But there is a second conjugacy class of such groups when $n > 2$.
Because the index of the normalizer of $V_n^+$ in $\alt_{2^n}$ is twice that of its image $V_n$ in $\si_{2^n}$, 
 there are twice as many conjugates by using elements in $\si_{2^n}$, so $V_n^-$ will not be conjugate to $V_n^+$.
%Let  $V_n^-$ denote the conjugate of $V_n^+$ by say the permutation $(12) \in \si_{2^n}$.  
But because $V_n^+$ and $V_n^-$ are both conjugate to $V_n$ when included in $\si_{2^n}$ we have the following.
 
 \begin{proposition}\label{restrictions}
 Restriction maps to $V_n^\pm$ and $V_n$
 satisfy the following commutative diagram

\begin{equation*}
\xymatrix{
H^*(BS_{2^n}) \ar[r]^{\text{res}}\ar[d]^{\text{res}}&H^*(B\alt_{2^n} \ar[d]^{\text{res}\oplus\text{res}})\\
H^*(BV_n) \ar[r] &H^*(BV_n^+) \oplus H^*(BV_n^-),
}
\end{equation*}
where the bottom arrow is the diagonal map.
\end{proposition}

%Implicitly, we will also be restricting to the products of $V_n$ for $n \geq 2$,  through the coproduct.  
In the symmetric group setting restriction to products
of $V_n$, including $V_1$, detect cohomology.  Because  $(V_1)^n$ is not a subgroup of $\alt_n$ we restrict to the analogous 
subgroup for alternating groups.

\begin{definition}
\begin{itemize}
\item Let $\alt_I = \alt_{2^{i_1}} \times \cdots \times \alt_{2^{i_q}} \subset \alt_{2m}$, where $|I| = \sum 2^{i_k}$ is equal to $2m$.  
\item Let $AV_{1,m}$ be the subgroup of $\alt_{2m}$ obtained by intersecting with $(V_1)^{m}$  in $\si_{2m}$, which 
is the collection products of even numbers of the two-cycles $\tau_{2i} = (\{2i - 1\} \{2i \})$.
\end{itemize}
\end{definition}

The full set of maximal elementary abelian subgroups of alternating groups is described in \cite{Quan03}.   

We extend our use 
of geometric models for classifying spaces to that of $BAV_{1,m}$, and its inclusion in $B\alt_{2m}$.

\begin{definition}\label{defBAV}
We  model  $EAV_{1,2m}$ as $(S^\infty)^m$ by having each $\sigma$ act 
  by $-1$ on each factor corresponding to a two-cycle  $\tau_{2i}$ which occurs in $\sigma$.
This model then includes into our model $\Ind{2m}{\infty} $ for $E\alt_{2m}$ by sending 
$$(v_1, \cdots, v_m) \mapsto (x_1 - \varepsilon v_1, x_1 + \varepsilon v_1, x_2 - \varepsilon v_2, \cdots, x_m + \varepsilon v_m).$$  
Here we pick some  $x_i \in \R^\infty$ 
which are linearly independent and do not share their first coordinates and choose 
$\varepsilon$ so that collections of points on spheres of radius $\varepsilon$ about $x_i$ also have these two properties.

\end{definition}

In Section~\ref{detection_section}, we show that restriction to only these subgroups along with $V_n$ when appropriate is injective on cohomology.  
To apply that result requires calculations such as the following.

\begin{theorem}\label{AV_vanish}
The restriction to $AV_{1,m}$ of a Fox-Neuwirth cocyle whose constituent cochains each have two consecutive non-zero
terms is zero.
\end{theorem}

\begin{proof}
Passing to the quotient from the map given by Definition~\ref{defBAV} ,
  the map on classifying spaces induced by inclusion of $AV_{1,2m}$ as a subgroup of $\alt_{2m}$ by the composite
 $$(S^\infty)^m / AV_{1,2m} \to \Ind{2m}{\infty} / AV_{1,2m} \to  \Ind{2m}{\infty} / \alt_{2m}.$$

In this model, the restriction to $C^*(BAV_{1,2m})$ of a cochain on $B\alt_{2m}$ evaluates a chain by composing with this composite,
whose image consists of points which all respect sequences of the form $\Gamma = [a_1, 0, a_2, 0, \cdots, 0, a_m]$.  So 
a Fox-Neuwirth cochain with two consecutive non-zero terms will restrict to zero.
\end{proof}

%\begin{remark}
%For our Hopf semi-ring presentation in \cite{GSS12} it
%sufficed to build on the known product and cup coproduct structures on the homology of symmetric groups, presented as the Dyer-Lashof algebra
%as in \cite{CLM76}.  
%But we use invariant theory to study Steenrod action.  While this action could have also been determined from known work in homology, 
%the Steenrod action on rings of Dickson invariants are simple to understand and so provide a more useful starting point.
%\end{remark}

%%% Local Variables:
%%% TeX-master: "AltGroupsModTwo.V3.tex"
%%% End: % Restriction to elementary abelian subgroups
% !TEX root = AltGroupsModTwo.V3.tex

\section{Generators, with respect to both products}\label{generators}

\subsection{Hopf ring generators for the $\rho$ summand of $H^*(B\alt_{2n})$}\label{Ba}

Recall Theorem~\ref{AnnIdeal}  and Proposition~\ref{gysinsplit} which decompose the cohomology of $\alt_{2n}$,
into trivial and $\rho$ summands as $C_2$-representations, indexed by bases for the quotient by and annihilator ideal
of the Euler class in the Gysin sequence.  We establish multiplicative generators organized by this decomposition, starting
with explicit cochain models and some subgroup detection results before we can establish the main results.

Consider the Hopf ring generator $\gamma_{\ell,2^{k}}$, with $\ell + k = n$ and $\ell > 1$, 
which is in $H^{2^{n}- 2^k}(BS_{2^n})$. By Theorem~4.9 of \cite{GSS12} it s represented by the Fox-Neuwirth cochain 
$$\alpha_{\ell, 2^k} = [1,1,\dots,1, 0, 1, 1,\dots, 1, 0,\dots,0, 1,\dots 1, 1], $$ with $2^k$ groups of $2^\ell - 1$ consecutive ones, each
group separated by a single zero.  This  restricts to the sum  
$\alpha^+_{\ell, 2^k} + \alpha^-_{\ell, 2^k}$ in $FNA^*_{2^n}$.  
The $\alpha^\pm_{\ell, 2^k}$ have non-zero boundary, but we will complete each to a cocycle, 
and extend these to a family of generators $\gamma^+_{\ell,m} \in H^{m(2^\ell-1)}(B\alt_{m\cdot 2^\ell})$.

 \begin{definition}\label{gammaFNcochains}
Let $\alpha_{\ell,m}^\pm \in FNA_{m2^\ell}^{m (2^\ell - 1)}$ 
be the positive (resp. negative) Fox-Neuwirth cochain with $m$ blocks, each  a sequence of $2^\ell-1$ ones, separated by zeros.

Let $\beta_{\ell, m}(i, j)^o$ be the sum of positive and negative Fox-Neuwirth 
cochains each with $m+1$ blocks so that:
 \begin{itemize}
\item the  $i$th block is a singleton two;
\item the $j$th block is a sequence of $2^\ell - 3$ ones;
\item  all other blocks are sequences of $2^\ell - 1$ ones.
\end{itemize} 
Let $\beta_{\ell, m}^o$ denote the sum $\sum \beta_{\ell, m}(i, j)^o$  over all $i,j$ with $1 \leq i < j \leq m+1$.

Let $\Gamma_{\ell, m}^+ = \alpha^+_{\ell, m} + \beta_{\ell,m}^o,$
and let $\Gamma_{\ell, m}^- = \alpha^-_{\ell, m} + \beta_{\ell,m}^o.$ 

\end{definition}

\begin{proposition}\label{P:cocycle_reps}
 The $\Gamma_{\ell, m}^\pm$ are cocycles which represent distinct nonzero classes $\gamma_{\ell,m}^\pm \in H^{m(2^\ell-1)}(B\alt_{m\cdot 2^\ell})$. 
 When $m = 2^k$, these transfer to $\gamma_{\ell, 2^k}$ in the cohomology of symmetric groups. %which by abuse we refer to also as $\gamma_{\ell, 2^{k}}^\pm$.
\end{proposition}

\begin{proof}
Applying Lemma~\ref{FNlemma}, we have that 
$$\delta(\alpha^+_{\ell, m}) = \sum_{\{i \neq j\cdot 2^\ell\}} \alpha_{\ell,m}^o\langle i\rangle.$$

Lemma~\ref{FNlemma} also implies that 
$$\delta(\beta_{\ell, m}(j,j+1)^o) = \sum_{i=(j-1)\cdot2^\ell+1}^{j\cdot2^\ell-1} \alpha^o_{\ell, m}\langle i \rangle + \beta_{\ell, m}(j, j+1)^o\langle j2^\ell\rangle,$$
as the only terms with non-zero coefficients consist of shuffling the singleton two into the small block of consecutive ones, and concatenating the small 
block of ones with its neighboring larger block. When the two distinguished blocks are separated, the coboundaries that arise consist of similar 
collections of non-zero terms, and these cancel in pairs as the distinguished blocks vary, producing telescoping sums so that
$$\delta\left(\sum_{p = j+2}^{m+1}\beta_{\ell, m}(j,p)^o\right) = \beta_{\ell, m}(j, j+1)^o\langle j2^\ell\rangle.$$
Summing across all $\delta(\beta_{\ell, m}(j, p)^o)$, we obtain precisely $\delta(\alpha_{\ell, m}^o)$, and so $\Gamma^\pm_{\ell,m}$ is a cocycle. 

By Theorem~\ref{T:gysinmodel}, $\gamma^\pm_{\ell,2^k}$ each transfer to $\gamma_{\ell, 2^k}$ in the 
cohomology of symmetric groups, since  the $\beta_{\ell, 2^k}(i,j)^o$ vanish under transfer.  The cohomology classes they represent 
must be distinct, since their sum is the
restriction of $\gamma_{\ell, 2^k}$, which is non-zero.
\end{proof}

%By abuse, we will also denote the cohomology classes represented by the $\gamma_{\ell,m}^\pm$ using the same notation.
\begin{definition} \label{gamma1k}
For $k \geq 2$ let $\gamma_{1,k;m} \in H^k(\alt_{2m})$%, or by abuse $\gamma_{1,k}^o$ when $m$ is understood,
be the image of $\gamma_{1,k} \odot 1_{m-k}$ under restriction.  
\end{definition}

\begin{proposition}\label{P:cup_monos_vn}
For $n \geq 3$, cup  monomials of $\gamma_{\ell, 2^{k}}^+$ and $\gamma_{1, 2^n; 2^n}$  map injectively under the restriction to $V_n^+$ and,
except for powers of $\gamma_{1, 2^n;2^n}$ alone,  map to zero in $V_n^-$.
\end{proposition}

\begin{proof}
%\CG{Fix -- ripped this out without reformatting.}

Consider the restriction maps in Proposition~\ref{restrictions}.  
As shown in Section~7 of \cite{GSS12}, the image of $\gamma_{\ell, 2^{k}}$ under the vertical restriction map is the Dickson generator 
in degree $2^{n} - 2^k$.  
Because $V_n^+$ and $V_n^-$ are both conjugate in $\si_{2^n}$, the restriction to the lower right corner is the corresponding Dickson generator
on each factor.  %Now let $\ell \geq 2$.  
Following the diagram in the other direction, this says that the restriction of $\gamma_{\ell, 2^{k}}^+  + \gamma_{\ell, 2^{k}}^-$ for $\ell \geq 2$ must map to 
this direct sum of Dickson generators.  Since involution switches both $\gamma_{\ell, 2^{k}}^+$ and  $\gamma_{\ell, 2^{k}}^-$ as well as $V_n^+$
and $V_n^-$, and there are no other non-zero classes in this degree, $\gamma_{\ell, 2^{k}}^+$ must map to the Dickson generator in  $V_n^+$ 
and zero $V_n^-$ or vice versa.  Because  $\Gamma_{\ell, 2^{k}}^+$ and $\Gamma_{\ell, 2^{k}}^-$ differ only by the $\alpha_{\ell, 2^k}^{\pm}$ terms, and
these can only pair at the chain level with the $V_n^\pm$ with the same sign, in fact $\gamma_{\ell, 2^{k}}^+$ must map to the Dickson generator in  $V_n^+$
and to zero on $V_n^-$,
and similarly $\gamma_{\ell, 2^{k}}^-$ must map to the Dickson generator in  $V_n^-$ and to zero on $V_n^+$.

Proposition~\ref{restrictions} also implies that  $\gamma_{1, 2^n;2^n}$ restricts to the lowest-degree Dickson generator on both
 $V_n^+$ and $V_n^-$.  Thus polynomials in $\{\gamma_{\ell, 2^{k}}^+, \gamma_{1, 2^n;2^n}\}$ 
restrict to polynomials in the corresponding Dickson classes on
$V_n^+$, forming a polynomial ring which maps injectively.  
\end{proof}

Recall from Definition~\ref{GysinDef} the basis $\mathcal{G}_{ann}$ for the annihilator ideal of the Euler class in the cohomology
of symmetric groups. We now construct
lifts of these classes to the cohomology of alternating groups, thus by Theorem~\ref{AnnIdeal} accounting for of the annihilator ideal
portion of the Gysin sequence.  We do so through a filtration of the the cohomology of symmetric groups.

\begin{definition}
Define the $\odot$-partition of a Hopf ring monomial in $H^*(B\si_{2n})$ to be the partition of $n$ by the widths 
(that is, component numbers divided by two)
of the constituent cup monomials.
\end{definition}

While $\odot$-partitions of gathered representatives
give a direct sum decomposition of the cohomology of symmetric groups, we instead consider the filtration given by
 partition refinement, which is preserved by cup multiplication because of Hopf ring distributivity.
%We now establish the desired result by induction on  $\odot$-partitions.

For symmetric groups, the subgroups $V_n$ perfectly detect decomposibility with respect to transfer product.  We have seen that the $\gamma_{\ell, 2^k}$ 
restrict injectively to generators of the Dickson algebra, and now conversely we have the following.

\begin{theorem}\label{decomptozeroinVn}
Transfer product decomposibles in $H^*(BS_{2^n})$ restrict to zero in the cohomology of $V_n$.
\end{theorem}

\begin{proof}
This follows immediately from Theorem~7.8 of \cite{GSS12}, which implies through the definition of scale-$n$ quotient that   
decomposibles with respect to $\odot$-product, which have smaller scale,  restrict to zero in $V_n$. 

More directly, %we show in Theorem~7.5 of \cite{GSS12} that 
the image in homology of $V_n$ is exactly Dyer-Lashof generators in $H_*(BS_{2^n})$.    
Using work of Bruner-May-McClure-Steinberger, namely Theorem~1.5 of \cite{BMMS86}, we show in
Theorem~4.13 of \cite{GSS12} that these Dyer-Lashof generators
are primitive with respect to the coproduct dual to the transfer product.  As $\odot$-decomposibles evaluate to zero on the image of homology of $V_n$,
 $\odot$-decomposibles restrict to zero in the cohomology of $V_n$.
\end{proof}

Recall our main strategy, outlined in Section~\ref{mainstrategy}, to produce an additive basis through the Gysin sequence.  We now carry through 
Step~\ref{steptwo} of that strategy.

\begin{theorem}\label{T:Ba}
The $\rho$ summand of $H^*(B\alt_{2n})$, as given by Theorem~\ref{AnnIdeal} 
and Proposition~\ref{gysinsplit}, is contained in the almost-Hopf semi-ring generated by all $\gamma_{\ell, 2^{k}}^+$ and $\gamma_{1, 2^n;2^n}$, 
along with $1^-$
\end{theorem}

\begin{proof}
We  compute how chosen transfer and cup products of these generators map under the transfer map
in the Gysin sequence in order to see that their images generate $\mathcal{G}_{ann}$.  Since 
transfer maps do not preserve cup products,  we argue by detection in $V_n^\pm$.

By Proposition~\ref{P:cup_monos_vn} a cup-monomial $m^+$ in $\{\gamma_{\ell, 2^{k}}^+, \gamma_{1, 2^n;2^n}\}$ 
will map to the corresponding Dickson monomial in the cohomology of $V_n^+$,
and by conjugation the corresponding monomial $m^-$ in $\{\gamma_{\ell, 2^{k}}^-, \gamma_{1, 2^n;2^n}\}$ will map to the same monomial in $V_n^-$.  
By Proposition~\ref{restrictions}  the image of $m^+$ under transfer, which restricts back to $m^+ + m^-$, must map to that same Dickson monomial 
in the cohomology of $V_n$.  
Since cup-monomials in the $\gamma_{\ell, 2^k}$ restrict isomorphically to the Dickson invariants in the cohomology of $V_n$, and all 
$\odot$-decomposibles map to zero by Theorem~\ref{decomptozeroinVn}, 
the transfer of $m^+$ must equal the monomial $m \in H^*(B\si_{2^n})$ obtained by removing decorations,
modulo $\odot$-decomposibles.  

Inductively applying Proposition~\ref{transfercompatible}, a Hopf ring monomial $h = m_1 \odot m_2 \odot \cdots \odot m_i \in \mathcal{G}_{ann}$ 
will be the image under transfer of the alternating group monomial
$h^+ = m_1^+ \odot m_2^+ \odot \cdots \odot m_i^+$, modulo terms with finer $\odot$-partitions.  The induction  reduces to 
$B\si_4$ and $B\alt_4$ where the elements of $\mathcal{G}_{ann}$ are the transfer image of products of our
$\gamma_{2,1}^\pm$ and $\gamma_{1,2;2}$ in $\alt_4$ as explicitly shown in  Proposition~\ref{A4Gysin}.
The pair $h^+$ and $h^- = 1^- \odot h^+$, which are Hopf ring monomials in the stated generators, thus inductively account for 
$h \in \mathcal{G}_{ann}$  in 
the Gysin sequence. %, which means they generate the $\rho$ summand.
\end{proof}

We record the following for further reference.

\begin{definition}
Define $\mathcal{B}_+$ to be the collection of $h^+$ in the proof of Theorem~\ref{T:Ba} above, and $\mathcal{B}_-$ be its image under 
conjugation.
\end{definition}

While we argue by filtration here, %after establishing full detection results using other subgroups in the next section it will follow that 
 it follows from Theorem~\ref{T:presentation} that cup-monomials in the $\{\gamma_{\ell, 2^{k}}^+, \gamma_{1, 2^n;2^n}\}$ for $\ell > 2$
%other than products of $\gamma_{2,2^k}^+$ and $ \gamma_{1, 2^n}^o$, 
transfer to exactly  the corresponding monomials in $\mathcal{G}_{ann}$.   Proposition~\ref{A4Gysin} shows this is not the case for $\ell = 2$.

\subsection{Hopf ring generators for the $k$-summand of $H^*(B\alt_{2n})$ }

We next account for subset of $\mathcal{G}_{quot}$ generated by the $\gamma_{1,m}$,
 before moving on to $\mathcal{G}_{quot} \backslash \mathcal{G}_{ann}$ in general.   
%Applying Proposition~7.2 of \cite{GSS12},
%on a fixed component $\si_{2n}$ the subspace of the cohomology of a symmetric group  generated by the $\gamma_{1,m}$ under both products 
%is polynomial, generated by the collection of $\gamma_{1,k} \odot 1_{m-k}$.  Indeed, this subset
%of $H^*(B \si_{2m})$ restricts isomorphically to the ring of symmetric polynomials in the cohomology of $H^*(B\si_2 \times \cdots \times B\si_2)$,
%a fact we will adapt for alternating groups.  
%
Recall  $\gamma_{1,k;m} \in H^k(\alt_{2m})$ from Definition~\ref{gamma1k}.
%Thus our previous $\gamma_{1,2^n}^o$ is $\gamma_{1,2^n; 2^n}$.  
%Since restriction is a ring map, these classes generate the image of the scale-one subset
%$\mathcal{G}_{quot}$ on $\alt_m$ under cup product.  

\begin{theorem}\label{T:Bo}
The $k$-summand of $H^*(B\alt_{2n})$, as given by Theorem~\ref{AnnIdeal} 
and Proposition~\ref{gysinsplit}, is contained in the almost-Hopf semi-ring generated by 
the classes $\gamma_{\ell, 2^{k}}^+$, $\gamma_{1, k; m}$, and $1_m$.
\end{theorem}

\begin{proof}
Recall from Definition~\ref{1decomposition}  the $1$-decomposition 
$m = m_1 \odot m_2$ where $m_1 \in H^*(B\si_r)$ is level one and $m_2 \in H^*(B\si_s)$ is the maximal width $\odot$-summand with
 scale greater than one. Moreover, for $m \in \mathcal{G}_{quot}$, the 
 constituent $\cdot$-monomial in $m_1$ of the form ${\gamma_{1,k}}^p$ with the largest power $p$ has $k \geq 2$ or $p = 0$.  
 We claim that such $m_1$
 is in the sub-ring of $H^*(B\si_r)$ generated by the $\gamma_{1,k} \oplus 1_{r-2k}$ for $k \geq 2$.  For this, we can  again
 appeal to the isomorphism of Proposition~\ref{levelone}
 of the level-one subring the ring of symmetric polynomials in one variable, generated by all $\gamma_{1,k} \oplus 1_{r-2k}$.   
 By Theorem~\ref{AnnIdeal}, $\mathcal{G}_{quot}$ spans the quotient by $\gamma_{1,1} \oplus 1_{r-2}$, so these $m_1$ must consist precisely 
 of (cup) monomials in the remaining generators.  By definition and the face that  restriction preserves cup products, the restriction
 of $m_1$ is a polynomial in $\gamma_{1, k; r}$.
 % or explicitly see that  the decomposition of the skyline
 %diagram of $m_1$ by rows yields a product of these generators whose leading term is $m_1$ and all other terms which are lower in the dictionary order
 %by cup product powers.  
 %Thus $m_1$ is the image
 %under restriction of a polynomial in the $\gamma_{1, k; m}$.
 
 % We know that $m_1$ is a cup-monomial in $\gamma_{1,k;d}$ for $d$ the width of $m_1$ (this includes $1_d$ as a trivial cup-monomial).  
By Theorem~\ref{T:Ba},  $m_2$ is the image under transfer of a sum of Hopf semi-ring monomials in $\gamma_{\ell, 2^{k}}^{\pm}$ 
and  $\gamma_{1, 2^n;2^n}$.   Applying Proposition~\ref{restrictiontransfer}, we have that the image under restriction of $m_1 \odot m_2$ 
is the transfer product of the restriction of $m_1$ with that  which transfers to $m_2$.
\end{proof}

\begin{definition}
Define $\mathcal{B}_o$ to be the collection of $m_1 \odot m_2$ constructed in the proof of Theorem~\ref{T:Bo}, our preferred basis for 
the $k$-summand of $H^*(B\alt_{2n})$.
\end{definition}

Theorems \ref{T:Ba} and \ref{T:Bo} imply the following.

\begin{corollary}\label{Generators} 
$\gamma_{\ell,2^k}^+$ and $\gamma_{1,k;m}$ along with the $1_m$ and $1^-$ are almost Hopf semi-ring  generators for $H^*(B\alt_\bullet)$.
\end{corollary}

We contrast the situation at level onewith the cohomology of symmetric groups, where we have Hopf semi-ring generators $\gamma_{1, k} \in H^k(S_{2k})$
and then all of our level-one cup ring generators on different components are the transfer products of these with unit classes.
By Proposition~\ref{transferneutral}, classes pulled back from the cohomology of symmetric groups such as the $\gamma_{1,k;m}$ and the unit classes
 annihilate each other under transfer product, so in fact the 
$\gamma_{1,k; m}$ are $\odot$-indecomposible.  Needing such ``extra'' Hopf semi-ring generators is only a minor nuisance.

%%% Local Variables:
%%% TeX-master: "AltGroupsModTwo.V3.tex"
%%% End: % Generators, with respect to both products
% !TEX root = AltGroupsModTwo.V3.tex

\section{Detection by subgroups}\label{detection_section}

As often the case in group cohomology we manage relations by restricting to cohomology of subgroups,
showing that restriction to those defined in Section~\ref{restriction_section} is injective.  Our proof is primarily through calculation,
using the Hopf semi-ring generators constructed in the previous section.

\begin{theorem}\label{detection}
The mod-two cohomology of $\alt_{2m}$ for $m$ not a power of two 
is detected by the  subgroups  $\alt_I$, over all $I$ with $|I| = m$,
 and $AV_{1,m}$.  The cohomology of $\alt_{2^n}$ is detected by $\alt_I$, $AV_{1,2^{n-1}}$, and $V_n^{\pm}$.
\end{theorem}

Starting with the embedding of the cohomology of $\alt_4$ in that of $V_2$ in Theorem~\ref{A4Invariant}, we inductively deduce the following.

\begin{corollary}
The mod-two cohomology of alternating groups is detected on elementary abelian subgroups.  Because the cohomology of elementary abelian 
groups is polynomial, there are thus no nilpotent elements
in these cohomology rings.
\end{corollary}

This theorem and corollary are analogues of similar statements for symmetric groups originally due to Madsen and Milgram \cite{MaMi79}.  
The analogue of
Theorem~\ref{detection} for symmetric groups only involves the coproduct and restriction to $V_n$, as restriction to $(V_1)^n$ occurs inductively.

The proof will be through theorems we establish below.  We proceed through  analysis of the restriction of our additive basis
to each family of subgroups.  The ``matrix'' defined through these restrictions is represented in
the following table.  
%In our setting with $\widetilde{B\alt_0} = S^0$, we define primitive classes as those with $\Delta(x) = 1^+ \otimes x + x \otimes 1^+ + 1^- \otimes \overline{x}  + \overline{x} \otimes 1^-$.

\begin{center}
\begin{tabular}{|  c | c | c | c |}%\label{restrictionmatrix}
\hline
& $V_n^{\pm}$ & $\coprod \alt_{I} $ & $AV_{1,n}$ \\ \hline
Cup monomials in & Injective  & Follows from (\ref{coprod_charged}) and    & Mostly 0 by  \\ 
$\gamma_{\ell, 2^k}^\pm$ and $\gamma_{1,2^k; 2^k}$   & by Theorem \ref{P:cup_monos_vn} & (\ref{coprod_neutral})  of Theorem \ref{T:presentation} & Theorem~\ref{cupmonotoAV} \\ \hline
$\odot$-decomposibles & 0 by  & Injective  & -- \\
 in   $\mathcal{B}_\pm$  and $\mathcal{B}_o$ & Theorem \ref{T:transfer_vn} & by Theorem \ref{T:transfer_ai} & \\ \hline
 Level one in $\mathcal{B}_o$ & 0 by  & $\bigotimes$ Level one & Injective  \\
  & Proposition \ref{P:scale_one_vn} &  by Proposition \ref{P:scale_one_ai}  & by Theorem \ref{T:scale_one_v1n}. \\ \hline
\end{tabular}
\end{center}
\label{matrix}

\medskip

We will  show in Theorem~\ref{T:transfer_vn} and Propositions~\ref{P:scale_one_vn} and \ref{P:scale_one_ai} 
that the matrix representing all of these restrictions is effectively block upper-triangular, 
after quotienting the cohomology of $\alt_I$  by its level-one subset.  We  show that the homomorphisms
corresponding to the diagonal blocks are injective in Theorems~\ref{P:cup_monos_vn}, ~\ref{T:transfer_ai} and \ref{T:scale_one_v1n},
and moreover in Theorem~\ref{T:transfer_ai}  that
restriction to $\alt_I$  continues to be injective after quotienting by the level-one subset.    
%So the matrix representing these restrictions will be effectively
%block upper-triangular and injective on those blocks, which will prove the result.

%We start with the left column.

\begin{theorem}\label{T:transfer_vn}
The restriction to $V_n^\pm$ of transfer product decomposibles are zero.
\end{theorem}

\begin{proof}
By Proposition~\ref{restrictiontransfer}, 
transfer products involving cup-monomials in the $\gamma_{1,k;m}$ will be pulled back from the cohomology of symmetric groups,
from which the result follows from Theorem~\ref{decomptozeroinVn}.  
All of our other almost-Hopf semi-ring generators  are in the cohomology of alternating groups
indexed by powers of two, so a non-trivial transfer product in our basis $\mathcal{B}_\pm$, $\mathcal{B}_o$ factors through
$\alt_{2^{n-1}} \times \alt_{2^{n-1}}$.

We apply Proposition~\ref{pullback} for $G = \alt_{2^n}$, $H = \alt_{2^{n-1}} \times \alt_{2^{n-1}}$ and $K = V_n^+$.  In this case,
${B\iota_K}^* \circ {B\iota_H}^!$ is the restriction to $V_n$ of a transfer product.  Let $v$ be the vertical map of the diagram, thus from
the pull-back to $B\alt_{2^{n-1}}$, and let $t$ be the top map, from $BV_n^+$ to the pullback.  Because the composite of
natural and restriction maps commute for pull-backs, this composite coincides with $v^! \circ t^*$, %in the notation of Proposition \ref{pullback}, 
which we show is zero by producing an involution on the pull-back.

Recall in  general  that any $\sigma \in G$ with $\sigma H = H \sigma$ defines a permutation of double-cosets by
$H  g K \mapsto H (\sigma g) K$.  Here we choose $\sigma$ to be the product of 2-cycles $(1 \; 2)(\{2^{n - 1} + 1 \}\{ 2^{n-1} + 2\} )$,
which obviously normalizes $H  = \alt_{2^{n-1}} \times \alt_{2^{n-1}}$.  We  claim that $\sigma$ permutes all double-cosets and thus
components of the pull-back non-trivially.  From the definition, a double coset fixed by $\sigma$ indexed by $g$ would coincide with a non-trivial intersection
between the right-coset $H \sigma$ and the conjugate $g K g^{-1}$.  All non-zero elements of $K = V_n^+$ and thus its conjugates are involutions
of the form  $(i_1 i_2) (i_3 i_4) \cdots (i_{2^n - 1} i_{2^n})$, which we call a full product of transpositions.
Because $\sigma$ preserves the partition into the first $2^{n-1}$ and last $2^{n-1}$ elements,
and by definition $H = \alt_{2^{n-1}} \times \alt_{2^{n-1}}$, $H \sigma$ preserves the partition as well.  So in order for an element of 
$(\tau_1 \times \tau_2) \cdot \sigma$ to be a full product of transpositions, we would have that $\tau_1 \cdot (1 \; 2)$ and 
$\tau_2 \cdot (\{2^{n - 1} + 1 \}\{ 2^{n-1} + 2\} )$ would themselves be full products of transpositions, though of course half as long.
But such permutations are even, so $\tau_1$ and $\tau_2$ could not be even as well.

Because the action of $\sigma$ defines an involution of both covering maps, $t$ and $v$, 
the image of $t^*$ is invariant under the involution defined by 
$\sigma$ with no fixed components, and $v^!$ will send each 
in a pair of cohomology classes to the same image, at the cochain level.  Thus the composite $v^! \circ t^*$ is zero on mod-two cohomology,
proving the result for $V_n^+$.  The result for $V_n^-$ follows by applying the involution, or from similar analysis.
\end{proof}

While this theorem is the direct analogue for alternating groups of 
Theorem~\ref{decomptozeroinVn} for symmetric groups, and $V_n = V_n^+$ is a subgroup
of both, we could not find a unified line of argument.  In particular we could not deduce the  alternating group statement as a corollary
of the symmetric group case using the Gysin sequence and involution.
The original proof for symmetric groups uses facts about homology - that is, the Dyer-Lashof algebra - which are not known for alternating groups.
And the proof we give here for alternating groups does not translate to the symmetric group setting, 
as the normalizer of $H =  \si_{2^{n-1}} \times \si_{2^{n-1}}$
is contained in the identity double-coset of $H \backslash \si_{2^n} / V_n$.

In contrast to this, our knowledge of restrictions for symmetric groups does immediately lead to the following.

\begin{proposition}\label{P:scale_one_vn}
Level one generators $\gamma_{1,k; 2^n}$ with $k < 2^n$ map to zero under restriction to $V_n^\pm$.
\end{proposition}

\begin{proof}
By definition, these generators are the restriction of a $\odot$-decomposible class, namely $\gamma_{1,k} \odot 1_{2^n - k}$.
They thus share
their restriction to $V_n = V_n^+$, which is zero by Theorem~\ref{decomptozeroinVn}.
\end{proof}

We move on to consider the restriction to arbitrary products of alternating groups.  By definition, the coproduct $\Delta$ is the map induced by 
the embeddings of $\alt_n \times \alt_m$ in $\alt_{n+m}$.  

\begin{definition}
 Set $\Delta_I$ to be the restriction map to $\alt_I$.   Define
the pure level one subspace of $\alt_I$ to be the tensor product of level one subspaces of constituent $\alt_{2^i}$ (that is, all tensor factors
are level one).
\end{definition}

\begin{theorem}\label{T:transfer_ai}
The span of transfer product decomposables in $\mathcal{B_\pm}$ and $\mathcal{B}_o$  maps injectively under $\bigoplus \Delta_I$, and
continue to do so after quotienting by the pure level one subpace of the cohomology of the $\alt_I$.
\end{theorem}

\begin{proof}
Our basis elements which are $\odot$-decomposible are of the form $h = m_1^+ \odot \cdots \odot m_p^+ \odot m_\omega$, where 
the $m_i$ are cup monomials of our generators and 
$m_\omega$ is level one or could be $1^-$.  Let $w_i$ be the width of $m_i$ (including $i = \omega$), so by assumption $w_1 > 0$,
 and let $I_h = \{ w_i \}$.  

By repeated application of Theorem~\ref{bialg}, $\Delta_{I_h} (h)$ will, in the basis given by tensor products of our standard bases,
contain exactly one term which is a tensor product of cup-monomials, namely $m_1^+ \otimes \cdots \otimes m_p^+ \otimes m_\omega$.  
Thus, all $h$ with $I_h$ equal to a given $I$ map injectively under $\Delta_I$, and that continues to be the case after quotienting by the
level one subspace of $\alt_I$.   Taking direct sum over possible $I_h$ gives the result. 
\end{proof}

%\begin{proposition}\label{P:cup_monos_v1n}
%For $n > 2$, the restriction of cup product monomials to $V_{1,n}$ is zero.
%\end{proposition}
%\begin{proof}
%\end{proof}

We finish proving  our detection result by addressing level-one basis elements.  The next result
is immediate from Proposition~\ref{transfercompatible} and the corresponding fact for symmetric groups, since all level-one
elements are pulled back from symmetric groups.

\begin{proposition}\label{P:scale_one_ai}
Level one basis elements map to tensor product of level one basis elements under the coproduct restriction to $\alt_I$.
\end{proposition}

Unfortunately, there are level-one classes whose coproduct and restriction to $V_n$ (when applicable) are trivial, in particular the $\gamma_{1,3;m}$.
We detect such classes using the alternating groups versions of subgroups which detect level-one classes
for symmetric groups, as described in Theorem~7.8 of \cite{GSS12}.

\begin{theorem}\label{T:scale_one_v1n} 
Level one classes  define a polynomial subring which restricts injectively to $AV_{1,m}$.
\end{theorem}

\begin{proof}
Level-one classes are pulled back from symmetric groups, with $\gamma_{1,k;m}  = res(\gamma_{1,k} \odot 1_{m-k})$.  
The restriction from $\si_{2m}$ to $AV_{1,m}$ factors through $(V_1)^m$, the subgroup of $\si_{2m}$ generated by the two-cycles
$(\{2k - 1\}\{2k\})$.  There, by Theorem~7.8 of \cite{GSS12} 
the restriction of $\gamma_{1,k} \odot 1_{m-k}$ in the cohomology
of $(V_1)^m$ is $\sigma_k$, the $k$th symmetric polynomial.

So we calculate the restriction from $(V_1)^m$ to $AV_{1,m} \cong (C_2)^{m-1}$, whose cohomology rings are polynomials
in $m$ and $m-1$ variables respectively.  Choose generators
of $AV_{1,m}$ as $(\{2k-1\}\{2k\})(\{2m -1\}\{2m\})$, which we call $\tau_{2k} \tau_{2m}$.   
Then let $x_k$ be the generator of cohomology of
$B(V_1)^m$ corresponding to the two-cycle $\tau_{2k}$ and $y_k$ be the generator of the cohomology of 
$BAV_{1,m}$ corresponding to $\tau_{2k} \tau_{2m}$.  The restriction homomorphism is then given by
\begin{equation*}
x_i \mapsto y_i \;\;\;  i < m;   \;\;\;\;\;\;\;  x_m \mapsto \sum_{i=1}^{m-1} y_i.
\end{equation*}

The image of $\gamma_{1,k;m}$ is the image of the $k$th symmetric function in the $x_i$ under this homomorphism.
Because 
$$\sigma_i(x_1, \cdots, x_m) = \sigma_i(x_1, \cdots, x_{m-1}) + \sigma_{i-1}(x_1, \cdots, x_{m-1}) \cdot x_m,$$
this homomorphism sends 
$$\sigma_i(x_1, \cdots, x_m) \mapsto \sigma_i(y_1, \cdots, y_{m-1}) +  \sigma_{i-1}(y_1, \cdots, y_{m-1}) \cdot \sigma_{1}(y_1, \cdots, y_{m-1}).$$
Thus $\gamma_{1,k;m}$  are sent to polynomial ring generators modulo decomposibles, so their image is a polynomial ring.   The subring
generated by $\gamma_{1,k;m}$ must itself be polynomial, mapping isomorphically to its image under restriction.
\end{proof}

Our detection result is essentially proven.

\begin{proof}[Proof of Theorem~\ref{detection}]

We make precise the argument outlined through  table which recorded
restriction results after the statement of Theorem~\ref{detection}.  

\begin{enumerate}

\item Cup monomials in the classes $\gamma_{\ell, 2^k}^\pm$ and $\gamma_{1,2^k; 2^k}$ map injectively to $V_n^\pm$ by Theorem \ref{P:cup_monos_vn}.

\item Classes which are $\odot$-decomposible  map to zero in $V_n^\pm$ by Theorem~\ref{T:transfer_vn}, and map injectively to the cohomology of 
the collection subgroups $\alt_I$ modulo the level-one subspace by Theorem~\ref{T:transfer_ai}.

\item Outside of powers of $\gamma_{1,2^k; 2^k}$ which we included in the first case above, restrictions of level-one classes from the 
cohomology of symmetric groups map to zero in $V_n^\pm$ by Proposition \ref{P:scale_one_vn}.  They map to zero in the quotient of 
the cohomology of 
the collection subgroups $\alt_I$ by the level-one subspace by Proposition \ref{P:scale_one_ai}.  They map injectively to the cohomology  of 
$AV_{1,n}$ by Theorem \ref{T:scale_one_v1n}.

\end{enumerate}

Thus collectively these three types of classes map injectively to the cohomology of the three types of subgroups named.  Because
these three types of classes span the cohomology of alternating groups, the result follows.

\end{proof}

In order to apply our detection result to verify relations, we need one further calculation of a restriction map to $AV_{1,2^n}$.  Restriction
calculations which are coproducts will be completed below.

\begin{theorem}\label{cupmonotoAV}
The restriction of  $\gamma_{\ell,m}^\pm$  to $AV_{1,m2^{\ell-1}}$  is zero for $\ell \geq 2$ other than $\ell=2, m=1$.
\end{theorem}

\begin{proof}
We  apply Theorem~\ref{AV_vanish}.
For $\ell > 2$, the Fox-Neuwirth cochain representative of $\gamma_{\ell, 2^k}$ as given in Definition~\ref{gammaFNcochains} all have at least
five $1$'s in each block. %, and thus are disjoint from chains in this composite.  
For $\ell = 2$, $m > 1$ the representative has three consecutive $1$'s in 
some block.  %to obtain the result.
\end{proof}

%%% Local Variables:
%%% TeX-master: "AltGroupsModTwo.V3.tex"
%%% End: % Detection 
 % !TEX root = AltGroupsModTwo.V3.tex

\section{Presentation of the cohomology of alternating groups}\label{presentation}

%\subsection{Statement of presentation}

%Our main result is the presentation of cohomology of alternating groups as follows.

%\pagebreak

\begin{theorem}\label{T:presentation}
$H^*(B\alt_\bullet)$ is the almost-Hopf semi-ring under cup and transfer products generated by classes 
\begin{align*}
\gamma^+_{\ell,m}\in H^{m(2^\ell-1)}(B\alt_{m\cdot 2^\ell})&\;\; 2 \leq \ell, \; 1 \leq m,
\\ \gamma_{1,k;m} \in H^{k}(B\alt_{2m}) & \;\;2 \leq k \leq m,
\\ 1_m \in H^0(B\alt_{2m}) \;\; 1 \leq m, \; &{\text and} \; 1^\pm \in H^0(\widetilde{B\alt_0}), 
\end{align*}
where the $1_m$ are units for cup products on their components and $1^+$ is the unit for transfer product.
%we write $\gamma_{\ell, m}^o \stackrel{\text{def}}{=} \gamma_{\ell, m}^+ + \gamma_{\ell,m}^- $, and. 

Relations between transfer products are 
\begin{align}
  \gamma_{\ell, m}^+ \odot \gamma_{\ell, n}^+ &= {m+n \choose n} \gamma_{\ell, m+n}^+ \label{rel_gamma_transfer}
  \\ 1^- \odot 1^- &= 1^+ \label{rel_neg_transfer_neg}
  \\ (1^+ + 1^-) \odot  \gamma_{1,k;m} &= 0 \;\; {\rm and} \;\; (1^+ + 1^-) \odot 1_m = 0 \label{rel_pos_plus_neg}
  \\ \prod \gamma_{1,k;m} \odot \prod \gamma_{1,\ell; n} &= 0, \label{rel_neutral_transfer_neutral}
\end{align}
where the products of $\gamma_{1,k;m}$ are arbitrary cup products which by convention include the empty product $1_m$.

Cup products between classes on different components are zero, and further cup relations are
\begin{align}
%\gamma_{1,k;m} \odot 1_n  &= 0\\ 
  \gamma_{\ell, m}^+ \cdot \gamma_{k,n}^- &= 0 \text{ unless }k=\ell=2, \label{rel_pos_cup_neg_general}
  \\\gamma_{2, m}^+ \cdot \gamma_{2,m}^- &= 
  \begin{cases}
  \left(\gamma_{2,m}^+ + \gamma_{2,m}^-\right)^2 + \left(\gamma_{2,m-1}^+\right)^2\odot\left(\gamma_{1,2;2}\right)^3 \;\; &\text{ if $m$ is odd} \\
   \left(\gamma_{2,m-1}^+\right)^2\odot\left(\gamma_{1,2;2}\right)^3  &\text{ if $m$ is even}, \label{rel_pos_cup_neg_three_block}
   \end{cases}
  \\\gamma_{\ell, m}^+\cdot \gamma_{1, k; m 2^{\ell-1}} &= 
  \begin{cases}
  (\gamma_{\ell, q}^+ \cdot \gamma_{1, k} ) \odot \gamma_{\ell, m - q}  \;\; &\text{if $k = 2^{\ell-1} \cdot q$}\\
  0 \;\; &\text{if $k$ is not a multiple of $2^{\ell-1}$.}\label{rel_pos_cup_neutral}
  \end{cases}
\end{align}

While $H^*(B\alt_\bullet)$ does not form a Hopf semi-ring, for any $\alpha$ and $\beta$ there is in general the equality
\begin{equation}
  \Delta (\alpha \odot \beta) = \Delta \alpha \odot_{\rho^+} \Delta \beta, \label{rel_coprod_transfer}
  \end{equation}
where $\odot_{\rho^+}$ is transfer product after applying the polarization operator of Theorem~\ref{bialg}.

Let $\gamma_{\ell,m}^-$ denote $1^- \odot \gamma_{\ell,m}^+$, and by convention set $\gamma_{\ell,0}^\pm = 1^\pm$ and
$\gamma_{1,1;m} = 0$.  Coproducts of generators are given by 
\begin{align}
  \Delta \gamma_{\ell, m}^+ &= \sum_{i+j = m} \left(\gamma_{\ell,i}^+ \otimes \gamma_{\ell,j}^+ + \gamma_{\ell,i}^- \otimes \gamma_{\ell,j}^- \right) \label{coprod_charged}
  \\ \Delta \gamma_{1,k;m} &= \sum \gamma_{1,p;i} \otimes \gamma_{1,q,j} , \label{coprod_neutral}
%\Delta \gamma_{1,k;m} &= \sum_{i+j=m}\sum_{\substack{p+q = k\\ 0 \leq p  \leq i, 0 \leq q  \leq j}} \gamma_{1,p;i} \otimes \gamma_{1,q,j}.
\end{align}
where the last sum is over $i,j,p,q$ with $i+j=m$ and $p+q = k$, where  $0 \leq p  \leq i$ and $ 0 \leq q  \leq j$, and $\gamma_{1,0;0}$ is to 
be interpreted as $1_0 = 1^+ + 1^-$.

\end{theorem}

%\pagebreak

%Relation~\eqref{rel_coprod_transfer} allows us to compute coproducts inductively, based on 
%Relations~(\ref{coprod_charged}) and (\ref{coprod_neutral}).

The appropriate context for understanding this presentation is  the closely related presentation given in Theorem~\ref{T:symmgenandrel}
of the cohomology of symmetric groups, which as a Hopf semi-ring  is built from polynomial rings generated by the $\gamma_{\ell, 2^k}$ 
with $\ell + k = n$ 
in the cohomology of $B\si_{2^n}$.
For alternating groups,
we build two copies of these polynomial rings on each $\alt_{2^n}$, though with lowest generator shared.  
Most products between these copies are zero, except  for the $\gamma_{2,m}^\pm$ whose unique behavior is
due to the fact that for $\alt_4$ there is only one copy of the transitive maximal elementary abelian $2$-subgroup,
while for higher $\alt_{2^k}$ there are two.  
Finally, the coproducts and transfer products of these two sets of generators behave according to rules governing charge,
with the $\gamma_{1,k;m}$ and $1_m$ being neutral.

The main results of Section~\ref{generators} -- namely Theorems~\ref{T:Ba}~and~\ref{T:Bo} which lead to Corollary~\ref{Generators} -- 
show that the classes listed generate the cohomology of 
$H^*(B\alt_\bullet)$ under cup and transfer products.  We establish the rest of the result presently.

\subsection{Coproducts}

While listed as Relations~(\ref{rel_coprod_transfer})-(\ref{coprod_neutral}), we establish the coproduct calculations first so we can 
use them in the process of verifying the other relations.   Relation~(\ref{rel_coprod_transfer})  is just a restatement of Theorem~\ref{bialg}.
In practice, it means taking only half of the terms of the coproduct of a charged class (there must be one charged class to consider, or else the 
transfer product will be zero), for example only the $+ \otimes +$ terms and not the $- \otimes -$ terms in the coproduct of a positively charged class.

Next using Proposition~\ref{transfercompatible} we see 
that Relation~(\ref{coprod_neutral}) is immediate from the coproduct formula for $\gamma_{1,n}$ in the cohomology of symmetric groups.
For Relation~(\ref{coprod_charged}) on the other hand,
the statement for symmetric groups only implies that for alternating groups modulo neutral classes.

To establish Relation~(\ref{coprod_charged}) we apply Theorem~\ref{T:cellular_coprod} and 
for convenience we let $\overline{\Delta} = \bigoplus_{i,j>0} \Delta_{i,j}.$ 
 %first show that for $\ell > 2$, $\overline{\Delta} \gamma_{\ell, 1}^\pm \in \text{Im}(\delta)$. 
For the $m=1$ case, we consider the cochain level representative for $\gamma_{\ell,1}^\pm$,  given in \refP{cocycle_reps},
  $$\Gamma_{\ell,1}^\pm = [1, 1, \dots, 1]^\pm + [2, 0, 1, 1, \dots, 1]^o,$$
 The first term has trivial coproduct by Theorem~\ref{T:cellular_coprod}, 
 while the second decomposes as $[2]^o \otimes [1,1,\dots 1]^o$, which is the coboundary of $[1]^+ \otimes [1,1,\dots, 1]^o$.

%Building on this we can establish that 
%$\Delta \gamma_{\ell, m}^+ = \sum_{i+j = m} \left(\gamma_{\ell, i}^+ \otimes \gamma_{\ell,j}^+ + \gamma_{\ell, i}^- \otimes \gamma_{\ell,j}^-\right) $.

To establish the general case, we require two new types of FN cochains. Let $\sigma_{\ell, m}(p; r)^o$ and $\tau_{\ell, m}(i,j)^o$ each be the sum of positive 
and negative cochains so that
\begin{itemize}
  \item $\sigma_{\ell, m}(p; r)^\pm$ have $(m+1)$ blocks, the $p$th of which is a sequence of $r$ ones, and
  \item $\tau_{\ell, m}(p, q; r, s)^\pm$ have $(m+2)$ blocks, the $p$th of which is a sequence of $r$ ones, the $q$th of which is a sequence of $s$ ones,
\end{itemize}
and all other blocks in each consisting of sequences of $2^\ell -1$ ones. 
Thus, for example, $\sigma_{2, 3}(1; 2)^o = [1, 1, 0, 1, 1, 1, 0, 1, 1, 1]^o$ and $\tau_{3, 3}(1, 3; 2, 1) = [1, 1, 0, 1, 1, 1, 1, 1, 1, 1, 0, 1]^o$.

We apply Theorem \ref{T:cellular_coprod}. As required, 
$$\overline\Delta(\alpha_{\ell, m}^+) = \sum_{i+j=m} \left(\alpha_{\ell,i}^+ \otimes \alpha_{\ell,j}^+ + \alpha_{\ell,i}^- \otimes \alpha_{\ell,j}^-\right).$$
However $\overline{\Delta}(\beta_{\ell,m}^o)$ contains several additional terms, as follows
\begin{multline*}
%  \begin{split}
    \overline{\Delta} (\beta_{\ell,m}(p,q)^o) = \sum_{\substack{i+j = m\\1\leq i < p}} \alpha_{\ell,i}^o \otimes \beta_{\ell,j}(p-i, q-i)^o  
    + \sum_{\substack{i+j = m\\q \leq i \leq m}} \beta_{\ell,i}(p,q)^o \otimes \alpha_{\ell,j}^o\\
    + \sum_{\substack{i+j = m\\p \leq i < q}} \sigma_{\ell,i}(p; 1)\langle(p-1)2^\ell + 1\rangle^o \otimes \sigma_{\ell,j}(q-i; 2^\ell-3)^o.
 % \end{split}
\end{multline*}
%\begin{equation*}
 % \begin{split}
 %   \overline{\Delta} (\beta_{\ell,m}(p,q)^o) &= \sum_{\substack{i+j = m\\1\leq i < p}} \alpha_{\ell,i}^o \otimes \beta_{\ell,j}(p-i, q-i)^o  + \sum_{\substack{i+j = m\\q \leq i \leq m}} \beta_{\ell,i}(p,q)^o \otimes \alpha_{\ell,j}^o\\&+ \sum_{\substack{i+j = m\\p \leq i < q}} \sigma_{\ell,i}(p; 1)\langle(p-1)2^\ell + 1\rangle^o \otimes \sigma_{\ell,j}(q-i; 2^\ell-3)^o.
 % \end{split}
%\end{equation*}
For example, 
\begin{equation*}
  \begin{split}
    \overline\Delta(\beta_{2, 4}(2, 4)^o) &= [1,1,1]^o\otimes [2, 0, 1, 1, 1, 0, 1]^o + [1,1,1,0,2]^o \otimes [1, 1, 1, 0, 1]^o + [1,1,1,0,2,0, 1, 1, 1]^o \otimes [1]^o \\
    &=\alpha_{2,1}^o\otimes \beta_{2,3}(1,3)^o + \sigma_{2,2}(2;1)\langle 2 \rangle^o \otimes \sigma_{2,2}(2;1)^o + \sigma_{2,3}(2;1)\langle 2 \rangle^o \otimes \sigma_{2,1}(1;1)^o.
  \end{split}
\end{equation*}

To obtain the desired relation on cohomology, we produce cochains whose coboundary are precisely these $\sigma$ terms.

We first compute $\delta(\tau_{\ell, m+1}(p, q; 1, 2^\ell-5)^o)$ for $p < q$.  Let $\kappa_{a,b}$ be the Kronecker delta function and 
$\overline{\kappa_{a,b}} = 1 - \kappa_{a,b}$.
Applying Lemmas~\ref{FNlemma} and \ref{FNlemma2}, and using the fact that when $\ell > 2$, $|\text{Sh}_+(2^\ell-4, 2)|$ is odd and 
$|\text{Sh}_-(2^\ell-4, 2)|$ is even,  we have that the coboundary of $\tau_{\ell, m+1}(p,q; 1, 2^\ell-5)^o$ is equal to 
\begin{equation*}
 \overline{\kappa_{p,1}} \tau_{\ell, m}(p-1, q-1; 2^\ell+1, 2^\ell-5)^o + \sigma_{\ell, m}(p; 2^\ell-3)^o 
  				+ \overline{\kappa_{q,m}} \tau_{\ell, m}(p, q; 1, 2^{\ell+1}-5)^o,
\end{equation*}
if $q = p+1$.  Otherwise if $q < p+1$, it is equal to 
\begin{multline*}	
     \overline{\kappa_{p,1}} \tau_{\ell, m}(p-1, q-1; 2^\ell+1, 2^\ell-5)^o  + \tau_{\ell, m}(p, q-1; 2^\ell+1, 2^\ell-5)^o 
      			\\ + \tau_{\ell, m}(p, q-1; 1, 2^{\ell+1}-5)^o  + \overline{\kappa_{q,m}} \tau_{\ell, m-1}(p, q; 1, 2^{\ell+1}-5)^o.
\end{multline*}
Summing over $p$ and $q$ the resulting $\tau$ terms in the boundary telescope, and we have
$$\delta\left(\sum_{1 \leq p < q \leq m+1}\tau_{\ell, m+1}(p,q; 1, 2^\ell-5)^o\right) = \sum_{p=1}^m\sigma_{\ell, m}(p; 2^\ell-3)^o.$$
Using the same techniques,
\begin{equation*}
  \begin{split}
    \delta\left(\sigma_{\ell, m}(p; 1)\langle (p-1)2^\ell + 1 \rangle^o\right)  &=  
    		\overline{\kappa_{p,1}} \sum_{k=1}^{2^\ell+1}\sigma_{\ell, m-1}(p-1;2^\ell+1)\langle (p-2)2^\ell+k \rangle\\& 
		+ \overline{\kappa_{p,m}} \sum_{k=1}^{2^\ell+1}\sigma_{\ell, m-1}(p;2^\ell+1)\langle (p-1)2^\ell+k \rangle.
  \end{split}
  \end{equation*}
  These terms again telescope as we sum over the index $p$, so $\delta(\sum_{p=1}^m\sigma_{\ell, m}(p;1)\langle(p-1)2^\ell +1\rangle^o) = 0.$
  
  Finally, let 
  $$ \theta = \sum_{1 \leq p \leq i} \sigma_{\ell, i}(p;1)\langle(p-1) 2^\ell +1\rangle^o \otimes 
    									 \sum_{i+1 \leq s < q \leq j+1}\tau_{\ell, j+1}(s-i,q-i; 1, 2^\ell-5)^o.$$
									
Fixing indices $i$ and $j$ with $i + j = m$ and summing over pairs $p$ and $q$ which produce 
  $\sigma$ terms in the the coproduct computation for $\beta_{\ell, m}^o$, we have by the Leibniz rule that 
  \begin{equation*}\label{coboundingcochain}
%    \begin{aligned}
    \delta \theta 
  %  &=\delta\left(\sum_{1 \leq p \leq i} \sigma_{\ell, i}(p;1)\langle(p-1)2^\ell +1\rangle^o\right) \otimes \sum_{i+1 \leq s < q \leq j+1}\tau_{\ell, j+1}(s-i,q-i; 1, 2^\ell-5)^o  \\
  %  &\;\;\;\;+\sum_{1 \leq p \leq i} \sigma_{\ell, i}(p;1)\langle(p-1)2^\ell +1\rangle^o \otimes \delta\left(\sum_{i+1 \leq s < q \leq j+1}\tau_{\ell, j+1}(s-i,q-i; 1, 2^\ell-5)^o\right) \\
    = \sum_{1 \leq p \leq i} \sigma_{\ell, i}(p;1)\langle(p-1)2^\ell +1\rangle^o \otimes \sum_{1 \leq  q  \leq j}\sigma_{\ell, j}(q-i; 2^\ell-3)^o. \\
%    \end{aligned}
  \end{equation*}

We conclude that
\begin{equation*}
  \begin{split}
    \overline{\Delta} (\beta_{\ell,m}^o) &= \sum_{i+j = m} \alpha_{\ell,i}^o \otimes \beta_{\ell,j}^o  + \beta_{\ell,i}^o \otimes \alpha_{\ell,j}^o + \delta\theta,
  \end{split}
\end{equation*}

and thus
\begin{equation*}
  \begin{split}
    \overline{\Delta} \Gamma_{\ell, m}^+ %&= \overline{\Delta}(\alpha_{\ell, m}^+ + \beta_{\ell, m}^o)\\
    & = \sum_{i+j=m} \left(\alpha_{\ell,i}^+ \otimes \alpha_{\ell,j}^+ + \alpha_{\ell,i}^- \otimes \alpha_{\ell,j}^-\right) + \left(\alpha_{\ell,i}^o \otimes \beta_{\ell,j}^o  + \beta_{\ell,i}^o \otimes \alpha_{\ell,j}^o\right) + \delta\theta\\
%    & = \sum_{i+j=m} \left(\alpha_{\ell,i}^+ \otimes \alpha_{\ell,j}^+ + \alpha_{\ell,i}^- \otimes \alpha_{\ell,j}^- + \alpha_{\ell,i}^o \otimes \beta_{\ell,j}^o  + \beta_{\ell,i}^o \otimes \alpha_{\ell,j}^o + 2 \cdot \beta_{\ell,i}^o \otimes \beta_{\ell,j}^o\right) + \delta\omega \\
%    & = \sum_{i+j=m}\left(\alpha_{\ell,i}^+ \otimes \alpha_{\ell,j}^+ + \alpha_{\ell,i}^+ \otimes \beta_{\ell,j}^o + \beta_{\ell,i}^o \otimes \alpha_{\ell,j}^+ + \beta_{\ell,i}^o \otimes \beta_{\ell,j}^o\right.\\
%    & \qquad\qquad\qquad\left.+ \alpha_{\ell,j}^- \otimes \alpha_{\ell,j}^- + \alpha_{\ell,i}^- \otimes \beta_{\ell,j}^o + \beta_{\ell,i}^o \otimes \alpha_{\ell,j}^- + \beta_{\ell,i}^o \otimes \beta_{\ell,j}^o\right) + \delta\omega\\
    & = \sum_{i+j=m}\left((\alpha_{\ell,i}^+ + \beta_{\ell,i}^o) \otimes (\alpha_{\ell,j}^+ + \beta_{\ell,j}^o) + (\alpha_{\ell,i}^- + \beta_{\ell,i}^o) \otimes (\alpha_{\ell,j}^- + \beta_{\ell,j}^o)\right) + \delta\theta\\
    & = \sum_{i+j = m} \left(\Gamma_{\ell, i}^+ \otimes \Gamma_{\ell,j}^+ + \Gamma_{\ell, i}^- \otimes \Gamma_{\ell,j}^-\right) + \delta\theta.
    \end{split}
  \end{equation*}
  
\subsection{Transfer product relations}  

Most transfer product relations are immediate from results in Sections~\ref{almosthopf}~and~\ref{SymmGroupSection}.
Recall that by Definition~\ref{extended}
the transfer product with $1^-$ is the conjugation on the cohomology of $B\alt_n$ as a two-sheeted over over $B\si_n$.
Relation \eqref{rel_neg_transfer_neg} is just a rephrasing of the fact that this conjugation is an involution.
Relation \eqref{rel_pos_plus_neg} expresses the fact that $\gamma_{1,k;m}$ is fixed under conjugation, as it is
pulled back from the cohomology of symmetric groups.   
Relation \eqref{rel_neutral_transfer_neutral}   follows from   Proposition~\ref{transferneutral}.

 What requires further  argument is Relation \eqref{rel_gamma_transfer}, which
we recall is that   $\gamma_{\ell, m}^+ \odot \gamma_{\ell, n}^+ = {m+n \choose n} \gamma_{\ell, m+n}^+.$
We explicitly compute with  Fox-Neuwirth cochains from Definition \ref{gammaFNcochains}, with 
\begin{equation*}
\begin{split}
\Gamma_{\ell, m}^+ \odot \Gamma_{\ell,n}^+&= \left(\alpha^+_{\ell, m} + \beta_{\ell, m}^o\right) \odot \left(\alpha^+_{\ell, n} + \beta_{\ell, n}^o\right)\\ 
&=\alpha^+_{\ell, m} \odot \alpha^+_{\ell,n}
\;+\; \alpha^+_{\ell,m} \odot \beta_{\ell, n}^o
\;+\; \beta_{\ell, m}^o \odot \alpha^+_{\ell,n}
\;+ \; \beta_{\ell, m}^o\odot \beta_{\ell, n}^o
\end{split}
\end{equation*}

Applying Theorem~\ref{T:cellular_transfer_product},
the first term is given by the sum of the shuffles of the zero-blocks of 
$\alpha_{\ell,m}^+$ and $\alpha_{\ell,n}^+$, each of which results in a copy of $\alpha_{\ell, m+n}^+$, 
so $\alpha^+_{\ell, m} \odot \alpha^+_{\ell,n} = {m + n \choose n}\alpha_{\ell m+n}^+.$

The second and third terms consist of cochains obtained by shuffling an additional $m$ blocks of $2^\ell-1$ 
ones into the blocks of cochains of the form $\beta_{\ell, n}(i,j)^o$. These shuffles preserve  charge, 
as the associated permutation at the level of labeled configurations is even, so the transfer product produces 
cochains of the form $\beta_{\ell, m+n}(i', j')^o$. To compute  coefficients, consider  
the entire product $\alpha_{\ell, m}^+ \odot \beta_{\ell, n}^o$ at once. 
%The terms of the form $\beta_{\ell, m+n}(i', j')^o$ that appear in this product correspond to shuffles of the $m$ 
%blocks from $\alpha_{\ell, m}^+$ with the $(n-1)$ identical blocks from some $\beta_{\ell, n}(i,j)^o$.
For any $i'< j'$ and $(m, n-1)$-shuffle there is a choice of $i < j$ so that the resulting term in the product is $\beta_{\ell, m+n}(i', j')$: 
simply ``unshuffle'' to determine what $i$ and $j$ must be. 
Thus, $\alpha^+_{\ell,m} \odot \beta_{\ell, n}^o = {m+(n-1) \choose n-1} \beta_{\ell, m+n}^o$, and similarly for the third term.

%The third term proceeds analogously, with $\beta_{\ell, m}^o \odot \alpha^+_{\ell,n} = {(m-1)+n \choose m-1} \beta_{\ell, m+n}^o$. 

The final term $\beta_{\ell, m}^o\odot \beta_{\ell, n}^o$ is the transfer product of two neutral cochains, and thus is zero by Corollary~\ref{C:transfercancel}.

Combining these, we have
\begin{equation*}
\begin{split}
\Gamma_{\ell, m}^+ \odot \Gamma_{\ell,n}^+&= {m+n \choose n} \alpha^+_{\ell, m+n} 
\;+\; {m+(n-1) \choose n-1} \beta_{\ell, m+n}^o
\;+\; {(m-1)+n \choose m-1} \beta_{\ell, m+n}^o
%&= {m+n \choose n} \alpha^+_{\ell, m+n} \;+\; {m+n \choose n} \beta_{\ell, m+n}^o 
= {m+n \choose n} \Gamma^+_{\ell, m+n},
\end{split}
\end{equation*}
which implies the result in cohomology.

\subsection{Cup product relations}\label{cuprelations}

To establish cup product relations we apply our main detection result, with coproducts are among the detection homomorphisms.  
We start with Relation~(\ref{rel_pos_cup_neg_general}) and the case of the product $\gamma_{n,1}^+ \cdot \gamma_{k, 2^{n-k}}^-$ for $n>2$.
In the proof of Proposition~\ref{P:cup_monos_vn}, we showed that $\gamma_{n,1}^+$ mapped to zero on $V_{n}^-$ and that $\gamma_{k, 2^{n-k}}^-$
mapped to zero on $V_{n}^+$.  Thus their product maps to zero on both of these subgroups.  By Relation~(\ref{coprod_charged}),
the  coproduct of $\gamma_{n,1}^+$ is trivial, and as $n>2$  
the restriction of both classes to $AV_{1,2^{n-1}}$ is zero by Theorem~\ref{cupmonotoAV}.  Because $\gamma_{n,1}^+ \gamma_{k,2^{n-k}}^-$ 
restricts to zero on $V_n^\pm$, $\alt_I$ and $AV_{1,2^{n-1}}$, it 
is zero by Theorem~\ref{detection}.  

The argument extends inductively
for $\gamma_{\ell,m}^+ \cdot \gamma_{k, n}^-$ more generally, where the restrictions to $\alt_I$ are 
calculated by repeated application of 
Relation~(\ref{coprod_charged}).  As $(\cdot, \Delta)$ form a bialgebra, this will be zero by induction.   The rest of the argument applies to
see restrictions to $AV_{1,m2^{\ell-1}}$ and, if applicable, $V_{\ell+p}^\pm$ (where $p = \log_2(m)$) are zero in order to apply Theorem~\ref{detection}.

\bigskip

The first case, $m = 1$, of Relation~(\ref{rel_pos_cup_neg_three_block}) is established in Proposition~\ref{first_gamma2_relation}. 
We prove the other cases by induction. %using our coproduct and detection results.  
%The coproduct of $\gamma_{2,m}^+$ is   
%$\sum \gamma_{2,i}^+ \otimes \gamma_{2,m-i}^+  \; + \; \gamma_{2,i}^- \otimes \gamma_{2,m-i}^-$,  and the coproduct of $\gamma_{2,m}^-$ is
%  $\sum \gamma_{2,i}^+ \otimes \gamma_{2,m-i}^-  \; + \; \gamma_{2,i}^- \otimes \gamma_{2,m-i}^+$.
Let $\gamma_{2,i}^o = \gamma_{2,i}^+ + \gamma_{2,i}^-$.
Multiplying the coproduct of $\gamma_{2,m}^+$  with that of $\gamma_{2,m}^-$ and 
using Relation~(\ref{rel_pos_cup_neg_three_block}) inductively, when $m$ is odd we get the sum  
  $$(\gamma_{2,i}^o)^2 \otimes \left({\gamma_{2,m-i}^o}^2 + {\gamma_{2,m-i-1}^+}^2 \odot {\gamma_{1,2}}^3 \right) + 
    \left({\gamma_{2,i}^o}^2 + {\gamma_{2,i-1}^+}^2 \odot {\gamma_{1,2}}^3 \right) \otimes (\gamma_{2,m-i}^o)^2.$$
And when $m$ is even we get 
  $$(\gamma_{2,i}^o)^2 \otimes \left({\gamma_{2,m-i-1}^+}^2 \odot {\gamma_{1,2}}^3 \right) + 
     \left({\gamma_{2,i-1}^+}^2 \odot {\gamma_{1,2}}^3 \right) \otimes (\gamma_{2,m-i}^o)^2,$$
a key point being that terms $(\gamma_{2,i}^o)^2 \otimes (\gamma_{2,m-i}^o)^2$ cancel when both $i$ and $m-i$ are odd.
These agree with the corresponding coproducts of the right hand side of Relation~(\ref{rel_pos_cup_neg_three_block}).  
 
 We next show that the restriction to $AV_{1,2m}$ of both sides of Relation~(\ref{rel_pos_cup_neg_three_block}) are zero when $m \geq 2$.  
 The vanishing of the restriction of the left-hand side is immediate from Theorem~\ref{cupmonotoAV}.   For 
 $\left({\gamma_{2,m-i-1}^+}^2 \odot {\gamma_{1,2}^o}^3 \right)$, we  first show that 
 in general $(\gamma_{2,n}^+)^2$
 has Fox-Neuwirth representative which is the sum of $\alpha_{2,m}(2)$, which has $m$ blocks of 
three repeated $2$'s, with $\sum_{i=1...m} [2,2,2,0,...,0,3,1,2,0,...,0,2,2,2]^o,$ 
 where the $i$th block of the $i$th therm is  $[3,1,2]$, modulo further potential neutral terms.
 
 That these are cocycles is straightforward. We show in Proposition~\ref{firstcup} that ${\gamma_{2,m}}^2$ is represented by $\alpha_{2,m}(2)$.
 We deduce from the Gysin sequence ${\gamma_{2,n}^+}^2 + {\gamma_{2,n}^-}^2$ is represented by $\alpha_{2,m}(2)^o$.  
 Since ${\gamma_{2,n}^+}^2$ and ${\gamma_{2,n}^-}^2$
 are related by conjugation, they must be of the form $\alpha_{2,m}(2)^\pm$ plus neutral terms.  These terms are as given, but their form
 is immaterial because we are taking the transfer product with ${\gamma_{1,2;2}}^3$,  so all neutral terms will contribute zero to
 the transfer product by Corollary~\ref{C:transfercancel}.   Now applying Theorems~\ref{T:cellular_transfer_product}   we see that the
 remaining terms are shuffles of $[2,2,2]$ blocks and $[3]$ blocks, which  will restrict trivially to $AV_{1,2m}$ by Theorem~\ref{AV_vanish}
 because of the consecutive $2$'s.
 
 Finally, when $2m$ is a power of two which is greater than four we claim that
 both sides of Relation~(\ref{rel_pos_cup_neg_three_block}) restrict to zero on $V_n^\pm$. 
 Each factor in the left hand side is zero on one of these subgroups as in our proof of Relation~(\ref{rel_pos_cup_neg_general}), 
 and the right-hand-side is a non-trivial transfer product so  Theorem~\ref{T:transfer_vn} applies.
 
 \bigskip
 
Relation~\eqref{rel_pos_cup_neutral} is akin to a relation in the symmetric groups setting which follows from Hopf semi-ring distributivity, but
needs to be addressed on its own here because the $\gamma_{1,k;n}$ are $\odot$-indecomposible. The proof is similar to that of Relation~(\ref{rel_pos_cup_neg_three_block}). 
%showing agreement under restriction to our detecting subgroups.  
The coproduct calculation is 
straightforward, and the restriction of both sides to $V_n^\pm$ is zero by Theorems~\ref{decomptozeroinVn}~and~\ref{T:transfer_vn}.
The  left-hand side of the relation restricts to zero on $AV_{1,m}$, by Theorem~\ref{cupmonotoAV}.  For the right-hand-side, we  compute a Fox-Neuwirth 
representative of $ (\gamma_{\ell, q}^+ \cdot \gamma_{1, k} ) \odot \gamma_{\ell, m - q} $.
We show in Proposition~\ref{firstcup} that a representative of $\gamma_{\ell, q}^+ \cdot \gamma_{1, k}$ is $\alpha_{\ell,m}(1.5)$, to be defined
in the Appendix,
which has $m$ blocks of the form 
$[2,1,2,...,1,2]$, each of length $2^\ell - 1$.
Thus a representative for 
$\gamma_{\ell,q}^+ \cdot \gamma_{1,k}$ is 
$$\alpha_{\ell,m}(1.5)^+ +  \sum_{1 \leq i < j < m+1} B_m(i,j)^o,$$ where $B_m(i,j)$ also has $[2,1,2,...,2]$ blocks along with a $[3]$ block 
and a $[2]$ block.  The resulting transfer products will have  $[2,1,2,...,2]$ or  $[1,1,...,1]$ blocks (or both), of length at least three, and thus
restrict trivially to $AV_{1,2m}$ by Theorem~\ref{AV_vanish}.

\subsection{Completeness of relations}

In the course of proving Theorems~\ref{T:Ba}~and~\ref{T:Bo}, we found an additive basis of Hopf semi-ring monomials,
namely those of the form $1^\pm \odot m_1 \odot \cdots \odot m_p \odot m_\nu$ where each $m_i$ is a 
monomial in the $\gamma_{\ell, 2^k}^+$ and $\gamma_{1,2^{k+\ell}}^o$, or a chosen representative monomial in the $\gamma_{2,1}^\pm$
and $\gamma_{1,2}^o$ (see Remark~\ref{NoGoodBasis}), and where $m_\nu$ is a monomial in the of $\gamma_{1,k;m}^o$, possibly empty or 
with $m = 0$.  If $m_\nu$ has $m>0$ (including if empty which means equal to  $1_m^o$) then $1^\pm$ is chosen to be $1^+$.

Hopf semi-ring distributivity provides all that is needed to reduce to the chosen Hopf monomial basis in the symmetric groups setting, since 
the ``building block'' cup product algebras are polynomial.  Here we must show that an arbitrary Hopf monomial can be further
reduced.  By Relation~(\ref{rel_gamma_transfer}), we can focus on cup monomials of
width powers of two, except for cup monomials in only the $\gamma_{1,k;m}^o$, of which there can be at most one by 
Relation~(\ref{rel_neutral_transfer_neutral}).  

Relation~(\ref{rel_pos_cup_neutral}) can be applied so that in any monomials with both $\gamma_{1,m;2^{k+\ell}}^o$ (with $m < 2^{k + \ell}$) 
and $\gamma_{\ell, 2^k}^+$ can be reduced to transfer products of monomials where only $\gamma_{1,2^p; 2^p}$ occur, applying
Relation~(\ref{rel_gamma_transfer}) as needed.  Relations~(\ref{rel_pos_cup_neg_general})~and~(\ref{rel_pos_cup_neg_three_block})
can be applied to reduce to  monomials with all generators  of width greater than two purely positive or negative.  Here we are applying 
Relation~(\ref{rel_pos_cup_neg_three_block}) only in the $m$ even setting where additional ``mixed'' terms cannot arise,
except ultimately for $m=1$.   By Proposition~\ref{first_gamma2_relation},  Relation~(\ref{rel_pos_cup_neg_three_block}) when $m=1$ suffices to 
reduce any width-two monomial. 
We can ``factor out'' the $1^- \odot$ from monomials in negative generators by Hopf semi-ring 
distributivity, and then by Relation~(\ref{rel_neg_transfer_neg}) there will ultimately be either  $1^+$ or $1^-$ multiplying monomials
in only positive or neutral generators.  By Relation~(\ref{rel_pos_plus_neg}), we may assume this is $1^+$ if there is a cup product
monomial in the $\gamma_{1,k;m}$.

%%% Local Variables:
%%% TeX-master: "AltGroupsModTwo.V3.tex"
%%% End: % Main Result
% !TEX root = AltGroupsModTwo.V3.tex

\section{Steenrod action}

The action of the Steenrod algebra on the mod-two cohomology of alternating groups also parallels that of symmetric groups.
Because transfers are stable maps, they preserve Steenrod squares, so the external Cartan formula for Steenrod operations
yields one for the $\odot$-product.  
The Steenrod algebra structure on
$H^*(B\alt_\bullet)$ is thus determined by 
the action on Hopf semi-ring generators. 

The description of Steenrod operations on most of our Hopf semi-ring generators parallels that given in Defintion~8.2 and Theorem~8.3 of \cite{GSS12}, but
we translate that language of outgrowth monomials into more explicit formulae using partitions, which  better serve in accounting for  irregularities. 
We also elaborate a needed restriction on level which we neglected to explicitly list in \cite{GSS12}.

First recall from Chapter~6 of \cite{AdMi94} that restriction to $V_n$ takes image in the classical Dickson invariants,
$D_n = \F_2[x_1, \ldots, x_n]^{GL_n(\F_2)}$.  These rings are polynomial, and starting with the cohomology of $\R P^\infty$ when $n=1$ they 
support rich Steenrod structure.  We reformulate Hu'ng's calculation 
 \cite{Hung91} of Steenrod action, starting with notation for Dickson invariants which are not standard in the literature
namely letting $d_{\ell} \in D_n$ (or alternatively, when needing to resolve ambiguity, $d_{\ell,n}$)
be the generator in degree $2^{k}(2^\ell - 1)$. With this convention, $d_{\ell}$  is the image of $\gamma_{\ell, 2^k}$  with $k + \ell = n$
under restriction to $V_n$.
We allow $\ell = 0$
in which case $d_ {0} = 1$, the unit class.  

With $n$ fixed, given $\ell$ let $p_\ell$ be the set of inferior and superior pairs  $(\ell_i, \ell_s)$,  with $0 \leq \ell_i \leq \ell \leq \ell_s \leq n$,
and if $\ell_i = \ell$ then $\ell_s$ must equal $\ell$ as well.
Steenrod action is given by the total square
$$\sq( d_{\ell}) = \sum_{(\ell_i, \ell_s) \in p_\ell} d_{\ell_s} d_{\ell_i}.$$
  Thus  Steenrod operations  raise algebraic degree by one,
except in cases with $\ell_i = 0$, in which case we see Steenrod operations binding all of the Dickson generators.
 We  track  this data and resulting degrees through the following.
 
\begin{definition}
An SD (Steenrod-Dickson) sequence is a sequence of four non-negative integers 
$ {\sigma} = (\ell_i, \ell, \ell_s, n)$, so that either $\ell_i < \ell < \ell_s \leq n $ or $\ell_i \leq \ell = \ell_s \leq n$.

Define the degree of such a sequence to be 
% Old: $s( {\sigma}) = (2^{\ell_s} + 2^{\ell_2-\ell}) - (2^{\ell_s - \ell_i} + 1)$ and the width to be $w(\sigma) = 2^{\ell_i - \ell}$.  
$s( {\sigma}) = (2^{n} + 2^{n-\ell}) - (2^{n - \ell_i} + 2^{n - \ell_s})$ and the width to be $w(\sigma) = 2^{n - 1}$.  
The number $\ell$ is the level of the sequence.
\end{definition}

 We rephrase the Steenrod action on Dickson generators as $\sq^j d_\ell = d_{\ell_s} d_{\ell_i}$ in $D_n$ if $j$ is
 equal to $s( {\sigma})$ for some SD sequence $ {\sigma} = (\ell_i, \ell, \ell_s, n)$, and $\sq^j d_\ell$ is zero otherwise.

\begin{definition}\label{bipartition}
For $\ell \geq 3$, a level-$\ell$ bi-partition $\pi$ of $(j,m)$ is a set $S$ of distinct level-$\ell$ SD sequences
such that % equality in $\N \oplus \N$ 
$$(j,m 2^{\ell-1}) =   \sum_{ {\sigma} \in S}   \left(s( {\sigma}), w(\sigma) \right).$$ 
%where 
%the sum is over a set of distinct
%SD sequences with $\ell$ the given level. % and the  $c_{ {\sigma}}$ are nonnegative integers.
%For a level-$k$ bi-partition $\Pi$, we let $c_{p, \ell'}(\Pi)$, or just $c_{p, \ell'}$ when $\Pi$ is understood, be the coefficient of $ {x}_{p,\ell'}$.
\end{definition}

%Compatibility of Steenrod operations with coproduct means that there is a variant of the Cartan formula over
%the ``$m$ copies of $\gamma_{\ell,1}^+$ in $\gamma_{\ell,m}^+$''.
In the following,  let 
$\gamma_{0,2m}^+$ be the unit class in the cohomology of $\alt_{2m}$ and any
$\gamma_{\ell, 0}^+$ be $1^+$.  By abuse interpret $\gamma_{1,2^n}^+$ as $\gamma_{1,2^n; 2^n}$.

\begin{theorem}\label{T:steen}
For $\ell \geq 3$, $\sq^{j} \gamma_{\ell,m}^+$ is the sum over all level-$\ell$ bi-partitions $\pi$ of $(j,m)$ of Hopf ring monomials $h_\pi$,
where $h_\pi$ is the transfer product over all $\sigma$ of  
$ \gamma^+_{\ell_s,  2^{n-\ell_s}} \gamma^+_{\ell_i, 2^{n-\ell_i}}$.   
\end{theorem}

In short, by a variant of the Cartan formula arising from compatibility of Steenrod operations with coproduct,
each $\gamma_{\ell, 2^{n-\ell}}^+$ ``portion'' of $\gamma_{\ell, m}^+$ can be replaced by a product of two Hopf semi-ring generators
in $\alt_{2^n}$. The first can either be of level greater than $\ell$, in which case the second must have level smaller than $\ell$,
 or the first can have level $\ell$, in which case the second has level less than or equal to $\ell.$ 
In both cases the second generator could have level zero so that portion is  
replaced by a single generator.  In the formula above each $\odot$-summand, which is associated
to an SD sequence $\sigma$, is of width $w(\sigma)$ and has degree increased by $s(\sigma)$ relative to the 
corresponding $\gamma_{\ell, 2^{n-\ell}}^+$ ``portion'' of $\gamma_{\ell, m}^+$. %, with  total degree increase across all portions of $j$.one with the same or greater level 

Consider the example presented in Figure~\ref{F:steenrod}.  While this is an $\ell = 2$ example,
the formula for Steenrod operations given in Theorem~\ref{T:steen} agrees with the $\ell > 2$ case except 
for  the last term listed.   In the first term, the $\odot$-factor
of $ \gamma_{3,1}^+ \gamma_{1,4;4} $ arises, through coproduct and restriction to $V_3$, 
because of the Steenrod operation $\sq^5 d_2 = d_3 d_1$ in $D_3$.
%, whose corresponding the  SD sequence would be  $(1,2,3,3)$.  
In the second term, the $\sq^5$ is ``Cartan distributed'' as an $\sq^3$ on the first portion, corresponding to 
$\sq^3 d_2 = {d_2}^2 \in D_2$ %so the SD sequence is $(2,2,2,2)$, 
and an $\sq^2$ on the second portion, corresponding to 
$\sq^2 d_2 = {d_2} d_1 \in D_2$. %, whose SD sequence is $(1,2,2,2)$.  
For skyline diagrams, the area added is the degree of the Steenrod operation being applied, and that added
area is distributed according to the first coordinate of the bi-partition.

After equating these Hopf semi-ring monomials coming from level-$\ell$ bi-partitions with outgrowth monomials, this
statement implies a corrected version Theorem~8.3 of \cite{GSS12} -- clarifying restrictions on the levels of pairs 
of Hopf ring generators which occur
-- by applying the transfer map.  
Conversely, the symmetric group
statement implies this up to the kernel of the transfer map, which are neutral classes.

\begin{proof}
The proof builds inductively from the Steenrod operations on Dickson algebras using the coproduct, which is  the 
proof  for symmetric groups given of Theorem~8.3 in \cite{GSS12}.  There we use a detection system established
by Madsen and Milgram \cite{MaMi79}, while here we use our detection Theorem~\ref{detection}.  

We start with $m = 1$, in which case the coproducts are trivial and 
the restriction to $AV_{1,2^n}$ of both sides of the equality are zero, as $\ell \geq 3$.  The restriction of $\gamma_{\ell,1}^+$ to $V_n$ is 
the corresponding Dickson class, and the sum given by level-$\ell$ bipartitions yields a single possibility,  which 
by construction coincides with the Steenrod action on the Dickson class  to which it maps.

For $m > 1$, we apply Relation~(\ref{coprod_charged}) of Theorem~\ref{T:presentation}. 
%$$ \Delta \gamma_{\ell, m}^+ = \sum_{i+k = m} \left(\gamma_{\ell,i}^+ \otimes \gamma_{\ell,k}^+ + \gamma_{\ell,i}^- \otimes \gamma_{\ell,k}^- \right).$$
Since coproduct also has a Cartan formula, we can inductively apply our present theorem.
The coproduct our formula for $\sq^{j} \gamma_{\ell,m}^+$ agrees with $\sq^j$ on the coproduct, as 
every level-$\ell$ bi-partition of $(j',i)$ and $(j'',k)$ with $j' + j'' = j$ and $i + k = m$ gives rise to a a level-$\ell$ bipartition
of $(j,m)$, whose corresponding term has the correct coproduct, and conversely.  
The restrictions to $AV_{1,2m}$ are trivial since we are taking transfer products
 of Hopf semi-ring monomials where each cup monomial has at least one Hopf semi-ring generator with $\ell \geq 3$.  
 For $m= 2^k$, the restriction to $V_n$ (where $n = \ell + k$) 
 maps transfer decomposibles to zero, so the restriction only depends on the bipartitions with a single term $(j,2^k)$, and the equality again follows
 from compatibility with Hung's calculation \cite{Hung91}.
   Theorem~\ref{detection} now establishes the induction step.
\end{proof}

While a corresponding statement holds for symmetric groups at all levels, for alternating groups we need modifications at the first two levels.

\begin{definition}
Add to the set of level-2 SD sequences the exceptional sequence $(1,2,1,2)$.    
A level-$2$ bi-partition of $(j,m)$ is  set $S$ of distinct level-2 SD sequences such that %an equality in $\N \oplus \N$ 
$$(j, 2 m) =   \sum_{ {\sigma} \in S}   (s( {\sigma}), w(\sigma)).$$ 
%an equality of $(j,m)$ with a sum of the $\vec{x}_{p, \ell'}$ from Definition~\ref{bipartition} for $\ell=2$,
%with arbitrary non-negative coefficients,  as well as $(1,1)$ with coefficient zero or one.
\end{definition}

\begin{theorem}
 $\sq^{j} \gamma_{2,m}^+$ is the sum over all level-$2$ bi-partitions $\pi$ of $(j,m)$ of Hopf ring monomials $h_\pi$,
where $h_\pi$ is the transfer product over all $\sigma$ of  
% Old: $\gamma^+_{\ell_i,c_{{\sigma}} \cdot 2^{\ell_s-\ell_i}} \gamma^+_{\ell_s, c_{{\sigma}}}$ 
$ \gamma^+_{\ell_s,  2^{n-\ell_s}} \gamma^+_{\ell_i,2^{n-\ell_i}}$ for standard SD sequences
along with a 
transfer product factor of ${\gamma_{1,2; 2}}^2$ if the exceptional sequence $(1,2,1,2)$ occurs.
\end{theorem}

\begin{figure}[tp]
\begin{center}
  \begin{tikzpicture}[line cap=round,line join=round,x=1.0cm,y=1.0cm, scale=0.5]

\draw [line width=1pt, color=black] (0,0) -- (4,0) -- (4,1.75) -- (0,1.75) -- cycle;
\draw [line width=1pt, color=black] (0,1.75) -- (4,1.75) -- (4,2.75) -- (0,2.75) -- cycle;
\draw [line width=1pt, dash pattern = on 3pt off 3pt, color=black] (1,1.75) -- (1,2.75);
\draw [line width=1pt, dash pattern = on 3pt off 3pt, color=black] (2,1.75) -- (2,2.75);
\draw [line width=1pt, dash pattern = on 3pt off 3pt, color=black] (3,1.75) -- (3,2.75);
\draw [line width=1pt, color=black] (4,0) -- (8,0) -- (8,1.5) -- (4,1.5) -- cycle;
\draw [line width=1pt, dash pattern = on 3pt off 3pt, color=black] (6,0) -- (6,1.5);

\node[align=center, scale=1] at (4,3.75) {$\gamma_{3,1}^+ \gamma_{1,4:4} \odot \gamma_{2,2}^+$};

\node[align=center, scale=0.7] at (0,-1) {$\sigma$ \\ ($s(\sigma), w(\sigma)$)};
\node[align=center, scale=0.7] at (2.5,-1) {(1,2,3,3) \\ (5, 4)};
\node[align=center, scale=0.7] at (6,-1) {(0,2,2,3) \\ (0, 4)};

\draw [line width=1pt, color=black] (10,0) -- (12,0) -- (12,3) -- (10,3) -- cycle;
\draw [line width=1pt, color=black] (10,1.5) -- (12,1.5);
\draw [line width=1pt, color=black] (12,0) -- (14,0) -- (14,1.5) -- (12,1.5) -- cycle;
\draw [line width=1pt, color=black] (12,1.5) -- (14,1.5) -- (14,2.5) -- (12,2.5) -- cycle;
\draw [line width=1pt, dash pattern = on 3pt off 3pt, color=black] (13,1.5) -- (13,2.5);
\draw [line width=1pt, color=black] (14,0) -- (18,0) -- (18,1.5) -- (14,1.5) -- cycle;
\draw [line width=1pt, dash pattern = on 3pt off 3pt, color=black] (16,0) -- (16,1.5);

\node[align=center, scale=1] at (14,3.75) {$\left(\gamma_{2,1}^+\right)^2\odot \gamma_{2,1}^+ \gamma_{1,2:2} \odot \gamma_{2,2}^+$};

\node[align=center, scale=0.7] at (11,-1) {(2,2,2,2)\\(3, 2)};
\node[align=center, scale=0.7] at (13,-1) {(1,2,2,2)\\ (2, 2)};
\node[align=center, scale=0.7] at (16,-1) {(0,2,2,3) \\(0, 4)};

\draw [line width=1pt, color=black] (20,0) -- (22,0) -- (22,3) -- (20,3) -- cycle;
\draw [line width=1pt, color=black] (20,1.5) -- (22,1.5);
\draw [line width=1pt, color=black] (22,0) -- (26,0) -- (26,1.75) -- (22,1.75) -- cycle;
\draw [line width=1pt, color=black] (26,0) -- (28,0) -- (28,1) -- (26,1) -- cycle;
\draw [line width=1pt, color=black] (26,1) -- (28,1) -- (28,2) -- (26,2) -- cycle;
\draw [line width=1pt, dash pattern = on 3pt off 3pt, color=black] (27,0) -- (27,2);

\node[align=center, scale=1] at (24,3.75) {$\left(\gamma_{2,1}^+\right)^2\odot \gamma_{3,1}^+ \odot \gamma_{1,2:2}^2$};

\node[align=center, scale=0.7] at (21,-1) {(2,2,2,2) \\(3, 2)};
\node[align=center, scale=0.7] at (24,-1) {(0,2,3,3) \\ (1, 4)};
\node[align=center, scale=0.7] at (27,-1.06) {(1,2,1,2) \\ (1, 2)};

\end{tikzpicture}
\caption{Skyline diagrams for the three summands of $Sq^5\gamma_{2,4}^+$. \\
For each diagram, each transfer product factor is labeled below by its corresponding 
SD sequence $\sigma = (\ell_i, \ell, \ell_s,n)$ and 
it contribution $(s(\sigma), w(\sigma))$ in the level-2 bipartition of $(5,4)$.%, except for  the $\gamma_{1,2:2}^2$ term. %, which is labeled only by its contribution to the bipartition.\\
}
\label{F:steenrod}
\end{center}
\end{figure}

\begin{proof}
Once again we use our detection theorem and  analysis of coproducts.  
Some exceptional behavior occurs, but propagates in a limited
way for similar reasons as in the behavior for cup products given in Relation~(\ref{rel_pos_cup_neg_three_block}) of Theorem~\ref{T:presentation}.

We treat $\gamma_{2,1}^+$  as in Theorem~\ref{A4Invariant} and 
Proposition~\ref{first_gamma2_relation}, through its restriction to the cohomology of $BV_2$, namely 
${x_1}^3 + {x_1}^2 x_2 + {x_2}^3 \in \F_2[x_1, x_2]^{C_3}$.    While $\sq^3$ is forced and $\sq^2$ is $\gamma_{2,1}^+
\gamma_{1,2}$ as expected, $\sq^1(\gamma_{2,1}^+)$ restricts to $x_1^4 + x_1^2 x_2^2 + x_2^4$, which is the image of $(\gamma_{1,2; 2})^2$.

As in the proof of Theorem \ref{T:steen}, we complete the proof by inductively showing that the sum indicated restricts appropriately to
the subgroups named in Theorem~\ref{detection}.  For $m > 1$ the restriction to $AV_{1,m \cdot 2^{\ell - 1}}$ of both $\gamma_{2,m}$ and
all of the terms which occur in the named sum  will be zero by Theorem~\ref{cupmonotoAV}.
Restriction to $V_n$ when $m = 2^k$ works as for $\ell > 2$, with transfer decomposibles restricting to zero and single-term bipartitions
giving rise to cup products of $\gamma_{\ell', 2^{k'}}$ which restrict to Hung's calculaltions in the corresponding Dickson algebras.

Induction is needed to check that the coproducts of the formula given for $\sq^j \gamma_{2,m}$ agree, after applying the Cartan formula
for coproduct, with the tensor product
of that for $\sq^i \gamma_{2,r}$ and $\sq^{i'} \gamma_{2,s}$ with $i + i' = j$ and $r + s = m$.  This is straightforward  except in seeing
 that terms with a transfer-factor of ${\gamma_{1,2}}^2$ on both tensor factors cancel in pairs.  This cancellation
 arises from the conjugation action on the coproduct,
as this term in $\sq^1 \gamma_{2,1}^+$ will also occur in $\sq^1 \gamma_{2,1}^-$.
Thus there can be at most one transfer product factor of ${\gamma_{1,2}}^2$ on both sides of the reduced form of the coproduct, 
inductively implying at most one transfer product factor of ${\gamma_{1,2}}^2$ for any terms in $\sq^j \gamma_{2,m}$.
\end{proof}

The Hopf semi-ring generators $\gamma_{1,k; m}$ are in a sense more regular than the $\gamma_{2,m}^+$ classes,
 as they are all restrictions from the cohomology of symmetric groups.  But our formula using level-$\ell$ bi-partitions or 
equivalently outgrowth monomials depends on transfer products, which are not preserved under the restriction map.  
To calculate Steenrod operations, we translate
between Hopf semi-ring presentation and presentation by cup product alone.  

For example, $\sq^1(\gamma_{1,2} \odot 1_n) = \gamma_{1,1}^2 \odot \gamma_{1,1} \odot 1_n + \gamma_{2,1} \odot 1_{n}$.
The  term $\gamma_{1,1}^2 \odot \gamma_{1,1} \odot 1_n$
is then equal to $(\gamma_{1,1} \odot 1_{n+1}) \cdot (\gamma_{1,2} \odot 1_n) + \gamma_{1,3} \odot 1_{n-1}$.  Thus for $m \geq 3$
 $$\sq^1(\gamma_{1,2;m}^o) =  \gamma_{1,3;m} + \gamma_{2,1}^+ \odot 1_{m-2}.$$

 The general case follows from classical work on the cohomology of $BO(n)$.  Note that
 $B(C_2)^n$ maps to both $BO(n)$ and $B\si_{2n}$ with Weyl group $\si_n$.  Thus both the cohomology of $BO(n)$ and $B\si_{2n}$ map to the ring
 of classical symmetric polynomials in $n$ variables, the former isomorphically.  Using the Cartan formula to compute Steenrod operations on $B(C_2)^n$
 yields symmetric monomials which can then be translated to elementary polynomials.  For example, $\sq^1 (\sigma_2)$ is the symmetrization of $x_1^2 x_2$,
 which if there are three or more variables is equal to $\sigma_1 \sigma_2 + \sigma_3$.  This shows that $\sq^1(w_2) = w_1 w_2 + w_3$, and 
 this formula for $\sq^1(\sigma_2)$ is also
 the image under restriction of the calculation of $\sq^1(\gamma_{1,2} \odot 1_n)$ above.   
 
 In general  Steenrod squares on $w_i \in H^*(BO(m))$, as first studied by Wu \cite{Wu50} and more recently Pengelley-Williams \cite{PeWi03},
  will be the image under scale-one quotient of the corresponding Steenrod squares
 on $\gamma_{1,i} \odot 1_{m-i}$, namely $\gamma_{1,j}^2 \odot \gamma_{1,i-j} \odot 1_{m-i}$.  This can be expressed as a polynomial in the
 $\gamma_{1,n} \odot 1_{m-n}$, which we then translate to the alternating groups setting.
 
  \begin{definition}
  Let $W(0,0;0) = 1^+ + 1^-$  and otherwise let $$W(i,j;m) = \sum_{l = 0}^{{\rm min}(j,m-i)}  \binom{i-j+l-1}{l} \gamma_{1,j-l;m} \gamma_{1,i+l;m}.$$
\end{definition}

If we replace all of the $\gamma_{1,k;m}$ by corresponding $w_k$ in $W(i,j)$, we obtain Wu's formula for $\sq^j w_i$ in the cohomology of $BO(m)$.

%Because the scale-one subset of the cohomology of symmetric groups maps isomorphically to the same symmetric algebra as that of $BO(m)$
%we obtain the following.

We take a ``gathered'' rather than binary approach to level-$1$ bi-partitions.

\begin{definition}
A level-$1$ bi-partition of $(j,k)$ is  a collection of non-negative integers $a, b, c_\ell$ such that 
$$(j, k) =   a \cdot (1,1) + b \cdot (0,1) + \sum_{\ell > 1} c_\ell (2^\ell - 1, 2^{\ell-1}).$$ 
%an equality of $(j,m)$ with a sum of the $\vec{x}_{p, \ell'}$ from Definition~\ref{bipartition} for $\ell=2$,
%with arbitrary non-negative coefficients,  as well as $(1,1)$ with coefficient zero or one.
\end{definition}

\begin{theorem}
$\sq^{j} \gamma_{1,k;m}^+$ is the sum over all  level-$1$ bi-partitions $(j,k)$ %= a \cdot (1,1) + b \cdot (0,1) + \sum c_\ell (2^\ell - 1, 2^{\ell-1})$
of the transfer product of $W(a+b,a;a+b+m-k)$ with the transfer product of all $\gamma^+_{\ell, c_\ell}$.
%In that case, $\sq^{j} \gamma_{1,k;m}^+$ is the  transfer product of all $\gamma^+_{1 + p, c_p}$ and $1^+ + 1^-$, which also yields
%a neutral class.
\end{theorem}

\begin{proof}
We apply the restriction map to $\sq^j (\gamma_{1,k} \odot 1_{m-k})$ in the cohomology of symmetric groups, which is the sum
over level-$1$ bi-partitions as indicated of the transfer product of $\gamma_{1,a}^2$,  $\gamma_{1,b}$, $1_{m-k}$ and all $\gamma_{\ell, c_\ell}$.
The transfer product of the first three is the image of $\sq^a (\gamma_{1,b} \odot 1_{m-k})$ which is then given by the Wu formula and
thus restricts to $W(a+b,a;a+b+m-k)$, including when all are zero, in which case $1_0$ restricts to $1^+ + 1^-$.  Applying 
Proposition~\ref{restrictiontransfer}, we replace any $\gamma_{\ell, c_\ell}$ for symmetric groups with a class that transfers to it, namely 
$\gamma_{\ell, c_\ell}^+$, to obtain the restriction of $\sq^j (\gamma_{1,k} \odot 1_{m-k})$ as needed.
\end{proof}

%%% Local Variables:
%%% TeX-master: "AltGroupsModTwo.V3.tex"
%%% End: % Examples
% !TEX root = AltGroupsModTwo.V3.tex

\section{Component Rings}

Our Hopf semi-ring presentation reduces the computationally imposing question of calculating ${\rm Ext}$ rings over
alternating groups to one which can be analyzed in a straightforward manner.  Namely, our Hopf semi-ring monomial basis
is enumerable, and we can simply search for indecomposible representatives by degree, accounting for our relations,
and using the Poincar\'e polynomial from our additive bases used to  terminate the process.  

In this section we give a more structured approach to the small example of $\alt_8$.  As a group of order 20160, it is at the limit of current 
computer-based techniques.  We then discuss why techniques that Feschbach developed for component rings of 
symmetric groups fail for alternating groups.  We first develop some shorter notation, relying on the fact that the Hopf semi-ring generators have distinct
degrees once we separate  the $\gamma_{1,k;m}$ classes.

\begin{definition}
\begin{itemize}
\item Let $\sigma_{k,m}$, or just $\sigma_k$ when the component is understood, denote $\gamma_{1,k;m}$.
\item Let $d_i^\pm$, where $i = m \cdot (2^\ell - 1)$, denote $\gamma_{\ell,m}^\pm$.  
\item When referring to a fixed $\alt_m$, any Hopf semi-ring monomial in $\sigma_k$ and $d_i^\pm$ is 
understood to define a class in its cohomology by assuming $\sigma_k = \sigma_{k,q}$ to result in full width, or taking a transfer product with a $1_q$ class.  
We drop signs from all notation
for neutral classes.  
\end{itemize}
\end{definition}

For example, for $\alt_8$, $\sigma_2 = \gamma_{1,2;4}$ and $d_3 = \gamma_{2,1}^+ \odot 1_4$.
Our generator names are similar to previous choices, which are natural  
because the $\sigma_k$ restrict to elementary symmetric polynomials in the cohomology of ${V_1}^n$ or $AV_{1,n}$
and the $d_i$ (for $i>3$) restrict to Dickson invariants on $V_n$ (with different indices than in the previous section).

\begin{theorem}\label{A8}
The mod-two cohomology of $\alt_8$ is generated as a ring under cup product by the following classes.
\begin{center}
\begin{tabular}{ l c c c c c c }
Degree & 2 & 3 & 4 & 5 & 6 & 7 \\  \hline
Classes & $\sigma_2 $ & $\sigma_{3 }$ &  $\sigma_{4 }$ &    & $d_6^+$  & $d_7^+$ \\
	       &                             & 	$d_3$	&    & $d_3 \odot \sigma_{2 }$ & $d_6^-$ & $d_7^-$
\end{tabular}	       
\end{center}
Relations are:
\begin{itemize}
\item Products of $\sigma_{2 }$, $d_3 $, or  $d_3 \odot \sigma_{2 } $ with 
$d_7^\pm$ are zero (six relations).
\item Products of  $\sigma_{3 }$  with $d_3 $, $d_3 \odot \sigma_{2 } $,
$d_6^\pm$,  or  $d_7^\pm$ are zero (six relations).
\item $d_6^+ \cdot d_7^- = d_6^- \cdot d_7^+ = 0$.
\item $d_7^+ \cdot d_7^- = 0$.
\item $(d_3 \odot \sigma_{2 } )^2 = \sigma_{2 } \cdot d_3 \cdot (d_3 \odot \sigma_{2 } ) \;  + \;
 {\sigma_{2 }}^2\cdot (d_6^+ + d_6^-) \; + \; (d_3)^2 \cdot \sigma_{4 }.$
\Item  $d_6^+ \cdot d_6^- = \sigma_{2 } \cdot d_3  \cdot (d_3 \odot \sigma_{2 } )
+ (\sigma_{2 })^3 \cdot (d_6^+ + d_6^-)  
 + \; d_3  \cdot \sigma_{4 } \cdot (d_3 \odot \sigma_{2 } ) + 
\sigma_{2 } \cdot \sigma_{4 } \cdot \left( (d_3 \odot \sigma_{2 } )^2 + d_6^+ + d_6^- \right).$\\
%\begin{align*}
% d_6^+ \cdot d_6^- = \sigma_{2 } \cdot d_3  \cdot (d_3 \odot \sigma_{2 } )
%\; &+ \; (\sigma_{2 })^3 \cdot (d_6^+ + d_6^-)  \\
% &+ \; d_3  \cdot \sigma_{4 } \cdot (d_3 \odot \sigma_{2 } ) \; + \; 
%\sigma_{2 } \cdot \sigma_{4 } \cdot \left( (d_3 \odot \sigma_{2 } )^2 + d_6^+ + d_6^- \right).
%\end{align*}
\end{itemize}

\noindent Steenrod squares on generators are:
\begin{center}
\begin{tabular}{ c | c c c c c c }
 & $\sq^1$ & $\sq^2$ & $\sq^3$ & $\sq^4$ & $\sq^5$ & $\sq^6$ \\  \hline
$\sigma_2$     & $d_3 $ & ${\sigma_2}^2$   &   &    &  &  \\ \hline
$\sigma_3$    &              & 	$\sigma_2 \sigma_3$& ${\sigma_3}^2$   &  &  & \\ \hline
$d_3$    &       ${\sigma_2}^2$       & 	$\sigma_2 d_3+ d_3 \odot \sigma_2 $	&  ${d_3}^2$   &  &  & \\ \hline
$\sigma_4$    &       $\sigma_2 \odot d_3$       & 	$\sigma_2 \sigma_4 + d_6^+ + d_6^-$	&  $\sigma_3 \sigma_4 + d_7^+ + d_7^-$  
          	 & ${\sigma_4}^2$         &  & \\ \hline
$d_3 \odot \sigma_2$    &              & 	$\sigma_2 (d_3 \odot  \sigma_2) $	& $\sigma_2 (d_6^+ + d_6^-)$   
		&  $\sigma_4 (d_3 \odot \sigma_2)$ & $(d_3 \odot \sigma_{2 } )^2 $  & \\ \hline
$d_6^\pm$     & $d_7^\pm + 	 d_3 \sigma_4$ & $\sigma_2 d_6^\pm$ &  $\sigma_4 (d_3 \odot \sigma_2)$ 
                 & $\sigma_4 d_6^\pm$    & $(d_3 \odot \sigma_2) d_6^\pm$  &  ${d_6^\pm}^2$ \\
     &  $+ \sigma_2 (d_3 \odot  \sigma_2)$ &  & $+d_3 d_6^\pm$  &  $+ (d_3 \odot \sigma_2)^2$  & $+ \sigma_4 d_7^\pm$ &  \\ \hline
$d_7^\pm$     &  &  &   &  $\sigma_4 d_7^\pm $&  & $d_6^\pm d_7^\pm$.  

\end{tabular}	       
\end{center}
\end{theorem}

\begin{proof}%[Proof of Theorem~\ref{A8}]
The Hopf monomial basis consists of polynomials in the $\sigma_4, d_6^+$ and $d_7^+$, those in $\sigma_4, d_6^-$ and $d_7^-$,
those in $\sigma_2$, $\sigma_3$ and $\sigma_4$, and then transfer products $f_1 \odot f_2$ where the $f_i$ are distinct polynomials
in $d_3^\pm$ and $\sigma_2$ chosen among some preferred but unnamed basis. For example $\sigma_4 d_3 = (\sigma_2 \cdot d_3^+) \odot \sigma_2$.
%, accounting for the relation in $H^*(B\alt_8)$.   
Only the last requires an argument that it is generated by the named classes.  

If $f_1$ and $f_2$ both contain a $d_3^\pm$ then $f_1 \odot f_2 = d_6^\pm \cdot (f_1' \odot f_2')$ where $f_i'$ are obtained by dividing by the 
$d_3^\pm$.  Inductively, we need to only generate Hopf semi-ring monomials where $f_2$ is $\sigma_2$ to some power.  Similarly, if both $f_1$ and $f_2$
have factors of $\sigma_2$ then $f_1 \odot f_2 = \sigma_4 \cdot (f_1' \odot f_2')$ where the $f_i'$ have one fewer factor of $\sigma_2$.  
Thus it suffices to generate Hopf semi-ring monomials of the form $d_3^i \sigma_2^j \odot 1_2$ or $d_3^i \odot \sigma_2^j$.  If $i $ is greater than one, 
such a monomial is a product of $d_3$ with the monomial with $i$ replaced by $i-1$ plus a term which will have a factor of $d_3$ on both sides
and thus can be reduced.  A similar argument reduces those with $j > 1$.  Finally, $d_3 \sigma_2 \odot 1_2 = d_3 \cdot \sigma_2 + 
d_3 \odot \sigma_2$, establishing our generating set.

Relations all can be verified through reducing to our Hopf ring monomial basis.  That they are complete is largely a matter of observing that all
of the zero products as well as the very last relation for $d_6^+ \cdot d_6^-$ 
leave only products of the $\sigma_i$, or those in $\sigma_4, d_6^+$ and $d_7^+$ or those in 
$\sigma_4, d_6^-$ and $d_7^-$ or those in $\sigma_2$, $d_3$, $d_3 \odot \sigma_2$ and $d_6^\pm$.   
The next-to-last relation allows us to reduce terms with $\sigma_2 \cdot d_3 \cdot (d_3 \odot \sigma_2)$ to terms where all three do not occur.
Such products of $\sigma_2$, $d_3$, $d_3 \odot \sigma_2$ and $d_6^\pm$ in which $\sigma_2 \cdot d_3 \cdot (d_3 \odot \sigma_2)$ do not occur are 
linearly independent.  This can be seen by reducing to the Hopf monomial basis and observing that the Hopf monomial which occurs in a product
with the largest constituent algebraic degree uniquely determines such a product.

The Steenrod operations are immediately verified by taking the calculations of the previous section and checking that the classes given by
 cup products of generators here reduce to them.
\end{proof}

We  compare this to the computer-generated minimal presentation on Simon King's group cohomology web-page \cite{King10}.  
Up to isomorphism,
our choice of generators is the same through degree five; his $b_{6,0}$ and $b_{6,2}$ are our $d_6^\pm + \sigma_{3 }^2$;
and his $b_{7,6}$ and $b_{7,4}$ are our $d_7^\pm + \sigma_{3 } \cdot \sigma_{4 }$.   In our presentation
fifteen relations -- all but two -- are  products of two elements which vanish, while only six of the relations in the computer presentation are of this form.
This cohomology was also considered by Adem, Maginnis and Milgram in \cite{AMM90} -- see Corollary~6.5 of \cite{AdMi94}.  
We have not been able to find an abstract isomorphism
of their presentation with ours or King's, and suspect errors in relations between the generators in degrees $6$ and $7$ 
stemming from the $V_3^\pm$ detection work (which inspired our own).

The last two cup product relations  are fairly complicated, but are in fact simple to state in the Hopf semi-ring monomial basis.  They read as
$(d_3 \odot \sigma_{2 } )^2 = (d_3^+)^2 \odot (\sigma_{2 } )^2$
and
$d_6^+ \cdot d_6^- = (d_3)^2\odot (\sigma_{2} )^3$ respectively.
The Steenrod operations as given in the Hopf monomial basis instead of using cup product alone would also be simpler,  especially for $d_6^\pm$.

\bigskip

Presentations using cup product alone 
look to be substantially more complicated starting for $\alt_{32}$.   To see why, we recall Feshbach's techniques 
for symmetric groups \cite{Fesh02}.  While he did not utilize a Hopf semi-ring structure, his approach anticipated such.  
For symmetric groups, a 
Hopf monomial $m_1 \odot \cdots \odot m_p$ is equal to the cup product of  $m_p$ and $m_1 \odot \cdots \odot m_{p-1} \odot 1_n$ 
plus terms with fewer non-trivial cup monomials (or columns, in the skyline basis).
Inductively applying this column-reduction process,
pure cup monomials transferred with unit classes generate the cohomology of any symmetric group.  Which cup monomials are  
decomposible, along with relations, are determined by column reduction as well.  Here we can find 
classes on larger symmetric groups which reduce to the single column or relation under 
the column reduction process, and thus agree in the indecomposibles.  But classes on larger symmetric groups which are too wide (that
is, with more non-trivial cup monomials than can be supported on the symmetric group in question) must restrict to zero.

This basic column reduction process does not work for alternating groups and in particular for the charged classes because any class less
than full width must be neutral, and the product of such must be neutral as well.  Thus at the moment we only have the naive method alluded to at
the beginning of this section for building
from our additive basis of Hopf monomials with multiplication rules to produce generators and relations.   Better understanding ring structure
in general remains an interesting open question.

\medskip 

While finding generators and relations remains a question of interest, Hopf semi-ring presentation has been more fruitful for applications.

For symmetric groups, the Hopf semi-ring approach 
along with formulae for Steenrod operations have allowed us to calculate quotient rings and annihilator ideals for the Euler
classes in the Gysin sequence, to more fully capture the relationship between Dickson algebras and cohomology of symmetric groups (which are as far
from split as one could imagine, as algebras over the Steenrod algebra), to better understand the cohomology of $QS^0$ as an algebra 
over the Steenrod algebra, and to start computing Margolis homology of symmetric groups and $QS^0$.  The approach also was  extendable to odd primes by Guerra \cite{Guer17}, who is extending calculations
to the $B$ and $D$ series of Coxeter groups in his PhD thesis.  
%In ongoing work with Pengelley we are finding a presentation for the cohomology of $BS_\infty$ as an unstable algebra over the Steenrod algebra, 
%and with Guerra and Salvatore we are treating the divided powers and $\Omega^\infty S^\infty$ functors.

For alternating groups, our presentation gives a computational proof of the lack of nilpotent elements, as the powers of an element will always yield a nonzero 
``tallest'' Hopf semi-ring monomial, and clearly indicates how the cohomology of the $B\alt_{4n}$ contain interesting, manifestly unstable charged ideals. 
We suspect this presentation would be useful in for example applying the theory of support varieties for both symmetric and alternating groups.

%%% Local Variables:
%%% TeX-master: "AltGroupsModTwo.V3.tex"
%%% End:

%\item $(\gamma_{2,1}^+ \odot 1_2^o) \cdot \gamma_{2,2}^+ = (\gamma_{2,1}^+ \odot 1_2^o) \cdot \gamma_{2,2}^-$
 % Steenrod action
% !TEX root = AltGroupsModTwo.V3.tex

\appendix
\section{Cup product input}\label{appendixa}

%Our calculations rest on four techniques, namely almost Hopf semi-ring structure, relationships with cohomology of symmetric groups, Fox-Neuwirth
%cochains and restriction to subgroups.  

%Our calculation of the cohomology of symmetric groups as a
%Hopf semi-ring has been foundational for our work on alternating groups.   Our work on symmetric groups in turn built on the additive basis
%and transfer coproduct structure of the homology of symmetric groups as described in \cite{Naka61}, \cite{CLM76} and \cite{BMMS86}.  
For the study of symmetric groups in \cite{GSS12} we 
did not use the previously calculated cup coproduct in homology  \cite{CLM76}, which can be difficult to apply because of the need to account for Adem relations.  
Because Fox-Neuwirth cochains do not model cup product, contrary to what is sketched in \cite{GiSi12}, we use two such cup coproduct calculations 
for our study of alternating groups.

To set notation, we denote the product associated to the inclusion $\si_n \times \si_m \to \si_{n+m}$ by $*$, which is thus dual to our coproduct in cohomology.
There are ``wreath product'' operations $q_i: H_k (B\si_n) \to H_{2k+i} (B\si_{2n})$ which satisfy Adem relations:
$$
{\rm For} \;\;  m > n, \;\; 
q_{m} \circ q_{n} = \sum_{i} \binom{i - n - 1}{2i - m - n} q_{m + 2n - 2i} \circ q_{i}.
$$
(We prefer ``lower index'' notation.)
Given a sequence $I = i_{1}, \cdots, i_{k}$ of non-negative integers, let $q_{I} = q_{i_{1}} \circ \cdots
\circ q_{i_{k}}$.  Using the Adem relations, these relations are spanned by $q_{I}$ whose entries are  non-decreasing.  We call such an $I$ {\em admissible}.   
If such an $I$ has no zeros we call it {\em strongly admissible}.  Let $\iota \in H_0(B\si_1)$ be the non-zero class, and by abuse let $q_I$ denote 
$q_I(\iota)$ 

Calculations of Nakaoka \cite{Naka61} imply that the homology of symmetric groups is a polynomial
algebra over the product $*$ generated by strongly admissible $q_I$.  Cohen-Lada-May \cite{CLM76} develop the theory much further, including 
showing that the cup coproduct is given by $\Delta_{\cup} q_I = \sum_{J+K = I} q_J \otimes q_K$.  %We build on the additive calculation 
In \cite{GSS12} we define $\gamma_{\ell,m}$ to be the linear dual of $q_{I}^{*m}$, where $I$ consists of $\ell$ ones, in the Nakaoka basis.
In Theorem~4.9 of \cite{GSS12} we essentially show that $\gamma_{\ell,m}$ is represented by the Fox-Neuwirth cocycle with $m$ blocks of 
$2^\ell - 1$ repeated ones.  To verify cup product relations, we establish Fox-Neuwirth representatives of two types of products of these.
 
\begin{proposition}\label{firstcup}
The product ${\gamma_{2,m}}^2$ is represented by the Fox-Neuwirth cochain $\alpha_{2,m}(2)$ with $m$ blocks of 
three repeated $2$'s.  

The product $\gamma_{\ell,m} \cdot \gamma_{1,m 2^{\ell-1}}$ is represented by the Fox-Neuwirth cochain $\alpha_{\ell,m}(1.5)$, which has $m$ blocks of the form 
$[2,1,2,...,1,2]$, each of length $2^\ell - 1$.
\end{proposition}

 \begin{proof}
 We perform the arguments in parallel.  
 We start with $m=1$, showing that both ${\gamma_{2,1}}^2$ and $[2,2,2]$ are linear dual in the Dyer-Lashof basis of ${q_{2,2}}$
(respectively,  both $\gamma_{\ell, 1} \cdot \gamma_{1,2^{\ell - 1}}$ and $[2,1,2,...,1,2]$ %(of length $2^\ell - 1$)
are the linear dual of $q_{1,\dots,1,2}$).
 We  calculate pairings, starting with ${\gamma_{2,1}}^2$, whose value on some $x$ is equal to that of $\gamma_{2,1} \otimes \gamma_{2,1}$ 
 (respectively $\gamma_{\ell, 1} \otimes \gamma_{1,2^{\ell - 1}}$) on $\Delta x$.
 For $x = q_{2,2}$ (respectively $q_{1,\dots,1,2}$) 
 we see that the only term in $\Delta q_{2,2}$ in the correct pair of degrees is $q_{1,1} \otimes q_{1,1}$ (respectively $q_{1,\dots,1}  \otimes q_{0, \dots,0, 1}$), 
 which pairs to one by definition.
  Any *-decomposible will have decomposible coproduct, and these will evaluate to zero on a tensor factor of  $\gamma_{\ell,1}$ by definition.
  
 For pairings with $[2,2,2]$ (respectively $[2,1,2,...,1,2]$), $*$-decomposible classes also evaluate to zero, by the analogue of  
 Theorem~\ref{T:cellular_coprod}.   We then apply the symmetric groups version of Proposition~\ref{intersection}
 to evaluate $[2,2,2]$ on $q_{2,2}$, which by definition is represented by a map from $S^2 \times_{\si_2} (\R P^2 \times \R P^2)$ to $\UConf{4}{\infty}$ sending
 $(u, v, w)/\sim$ to the configuration with points at $u \pm \varepsilon v$ and $-u \pm \varepsilon w$.  We get a single transversal intersection when $u$, $v$ and
 $w$ are all the equivalence class of $(0,0,1)$.  To evaluate $[2,1,2,...,1,2]$ on $q_{1,\dots,1,2}$ we represent the latter by a standard map of the iterated
 wreath product $S^1 \int \cdots \int \R P^2$, where $S^1 \int X = S^1 \times_{\si_2} (X \times X)$, to $\UConf{2^\ell}{\infty}$.  Once again, we will get a single
 transversal intersection when each coordinate entry of $S^1$ is $(0,\pm 1)$ and each of $\R P^2$ is the equivalence class of $(0,0,1)$.
 
To pass to higher $m$, we perform an induction based on the detection
 result of Madsen and Milgram \cite{MaMi79}, that the cohomology of symmetric groups is detected by coproduct, along with restriction to $V_n$ 
 for $\si_{2^n}$.  The formulae for coproducts of ${\gamma_{2,m}}^2$ and $\alpha_{2,m}(2)$ (respectively $\gamma_{\ell,m} \cdot \gamma_{1,m 2^{\ell-1}}$
 and $\alpha_{\ell,m}(1.5)$), the former given by Theorem~\ref{T:symmgenandrel} and Hopf semi-ring distributivity and the latter given by 
 Theorem~\ref{T:cellular_coprod},  will be equal by inductive
 assumption, and application of detection establishes the induction step when $m \neq 2^n$.  When $m = 2^n$ we note that by having the same coproducts
 they must agree up to the kernel of restriction to $V_n$, which is generated by $\gamma_{n,1}$.  But there are no non-zero multiples of $\gamma_{n,1}$
 in the relevant degrees. 
 \end{proof}

The cup product representatives thus far have be obtained simply by adding Fox-Neuwirth entries of the factors.  
In \cite{GiSi12} we define a Hopf semi-ring structure on 
Fox-Neuwirth cochains in this way, the homology of which agrees with the cohomology of symmetric groups.  But, contrary to what is claimed and sketched in
\cite{GiSi12}, this abstract isomorphism is not induced by the map between then given by Alexander duality as in the proof of Theorem~\ref{T:fnagood}.
  For example, while $\gamma_{1,2}$ is represented by $[1,0,1]$,
its cube ${\gamma_{1,2}}^3$ cannot be represented by $[3,0,3]$.  Its cube evaluates non-trivially with the class $q_{2,2}$ in homology, whose
coproduct includes $q_{2,0} \otimes q_{0,2} = q_{0,1} \otimes q_{0,2}$ by the first Adem relation.  But the image of $q_{2,2}$ can be embedded
in configurations in $\R^3$ and so cannot pair with $[3,0,3]$.  Indeed, ${\gamma_{1,2}}^3$ must be represented by $[3,0,3] + [2,2,2]$.  
There are  
variants of Fox-Neuwirth cochains which should result in cup product models as well.  Thankfully, our need for cup product input at the cochain level
was limited in the present work.

%%% Local Variables:
%%% TeX-master: "AltGroupsModTwo.V3.tex"
%%% End: % Appendix on cup product calculations

\bibliographystyle{alpha}
\bibliography{altgroups}
\end{document}